\documentclass{compositio}

\usepackage[utf8]{inputenc}
\usepackage[english]{babel}
\usepackage{newlfont}
\usepackage{amsmath,amssymb,amsfonts,amsbsy,bbm}
\usepackage{mathtools,braket,mathrsfs}
\usepackage[all]{xy}
\usepackage{tikz}
\usetikzlibrary{arrows}

 \usepackage[cal=euler]{mathalpha}

\usepackage{stmaryrd}
\usepackage{hyperref}
\usepackage[new]{old-arrows}
\usepackage{extarrows}
\usepackage{textcomp}

\newcommand{\plectic}[0]{\text{\textmarried}}
\newcommand{\llangle}{\langle\hspace{-1mm}\langle}
\newcommand{\rrangle}{\rangle\hspace{-1mm}\rangle}

\newcommand{\bb}{\mathbb}
\newcommand{\mbf}{\mathbf}
\newcommand{\msf}{\mathsf}
\newcommand{\scr}{\mathscr}
\newcommand{\mrm}{\mathrm}
\newcommand*{\bfcdot}{\scalebox{0.6}{$\bullet$}}
\newcommand{\et}{{\acute{\mathrm{e}}\mathrm{t}}}

\newcommand{\Z}{\ensuremath{\mathbb{Z}}}
\newcommand{\OO}{\ensuremath{\mathcal{O}}}
\newcommand{\Q}{\ensuremath{\mathbb{Q}}}

\newcommand{\p}{\ensuremath{\mathfrak{p}}}
\newcommand{\q}{\ensuremath{\mathfrak{q}}}
\newcommand{\A}{\ensuremath{\mathbb{A}}}

\newcommand{\too}{\longrightarrow}								

\newcommand{\twoheadlongrightarrow}{\relbar\joinrel\twoheadrightarrow}
\newcommand{\hooklongrightarrow}{\lhook\joinrel\too}
\newcommand{\into}{\hookrightarrow}

\newcommand{\G}{\ensuremath{\mathcal{G}}}
\newcommand{\n}{\ensuremath{\mathfrak{n}}}

\DeclareMathOperator{\Hom}{Hom}

\DeclareMathOperator{\Gal}{Gal}

\DeclareMathOperator{\GL}{GL}
\DeclareMathOperator{\PGL}{PGL}

\DeclareMathOperator{\disc}{disc}

\renewcommand{\det}{\operatorname{det}}
\DeclareMathOperator{\ord}{ord}

\DeclareFontFamily{U}{wncy}{}
    \DeclareFontShape{U}{wncy}{m}{n}{<->wncyr10}{}
    \DeclareSymbolFont{mcy}{U}{wncy}{m}{n}
    \DeclareMathSymbol{\Sh}{\mathord}{mcy}{"58}

\def\XXint#1#2#3{{\setbox0=\hbox{$#1{#2#3}{\int}$}%
\vcenter{\hbox{$#2#3$}}\kern-.5\wd0}}%

\theoremstyle{plain}
\newtheorem{theorem}{Theorem}[section]
\newtheorem{lemma}[theorem]{Lemma}
\newtheorem{proposition}[theorem]{Proposition}
\newtheorem{corollary}[theorem]{Corollary}
\newtheorem{conjecture}[theorem]{Conjecture}

\newtheorem{thmx}{Theorem}

\theoremstyle{definition}
\newtheorem{remark}[theorem]{Remark}

\newtheorem{definition}[theorem]{Definition}
\newtheorem{assumptions}[theorem]{Assumption}

\def\XXint#1#2#3{{\setbox0=\hbox{$#1{#2#3}{\int}$ }
\vcenter{\hbox{$#2#3$ }}\kern-.585\wd0}}

\makeatletter

\newcommand{\exterior}[1]{\mathop{\mathpalette\exterior@{#1}}}
\newcommand{\exterior@}[2]{%
  \raisebox{\depth}{%
  \fontsize{\sf@size}{0}%
  \m@th
  $\ifx#1\displaystyle\textstyle\else#1\fi\bigwedge$}%
  ^{\mspace{-2mu}#2}
  \kern-\scriptspace
}

\begin{document}

\title{Plectic Heegner classes}

\author{Michele Fornea}
\email{michele.fornea@unipd.it}
\address{Universit\`a degli Studi di Padova, Italia.}


\begin{abstract}
We introduce a new collection of partially global Galois cohomology classes 
subsuming both \emph{plectic Heegner points} and   \emph{mock plectic invariants}. The former are recovered as localizations of \emph{plectic Heegner classes}, while the latter arise as eigenspace projections with respect to a ``partial Frobenius''-action. By overcoming some limitations of previous constructions, plectic Heegner classes are expected to provide finer control over the arithmetic of higher rank elliptic curves. 

We are able to perform our construction via a systematic use of certain automorphic functions whose coefficients are $p$-adic measures valued in Galois cohomology. As we produce these functions through the uniformization of Shimura curves -- rather than higher dimensional quaternionic Shimura varieties -- our results are compatible with a plectic refinement of Tate's conjectures.
\end{abstract}

\maketitle

\tableofcontents

\section{Introduction}
Departing from previous approaches inspired by the uniformization of higher dimensional  quaternionic Shimura varieties (\cite{PlecticInvariants}, \cite{plecticHeegner}, \cite{plecticJacobians}, \cite{nonarchimedeanplecticjacobians}), this article constructs new  Galois cohomology classes relying on the uniformization of Shimura curves and their CM points.     
This  perspective emerged while exploring the ``mock plectic setting'' (\cite[End of Section 3.8]{DarmonFornea},  \cite[Remark 2.10]{Iwasawa4mock}) and, in hindsight, might have been anticipated by considering a plectic refinement of Tate's conjectures. Interestingly, it also seems to shed some light on the connection between  Nekov\'a$\check{\text{r}}$--Scholl's plectic conjectures and the Iwasawa theory of elliptic curves. 

Our approach builds on the work of Bertolini--Darmon \cite{BDHidaRational} and Mok \cite{MokHeegner}. 
To study finer invariants than those previously considered, we have to work with deeper level structures. Therefore, the relevant $p$-adic Hecke operators are defined following Ash--Stevens \cite{Ash-Stevens}, and their dependency on the choice of uniformizers is carefully analyzed throughout.  One of our main contributions is the introduction and use of certain automorphic functions whose coefficients are $p$-adic measures valued in Galois cohomology. These functions provide an efficient way to organize and manipulate arithmetic information, but they introduce new difficulties since the measures take values in large coefficient spaces.

The rest of the introduction is organized as follows: Section \ref{Intro: pHc} presents plectic Heegner classes and explains their relation to previous constructions. Section \ref{conjectures} formulates conjectures describing their significance for the arithmetic of higher rank elliptic curves. Finally, Section \ref{Intro: overview of the construction} gives  an overview of their construction.

\subsection{Plectic Heegner classes}\label{Intro: pHc}
Let $p$ be a rational prime and  $F$ a totally real number field. For any rational place $v$, denote by $\Sigma_v$ the set of $v$-adic places of $F$. We consider a quadratic CM extension $E/F$ and a modular elliptic curve $A_{/F}$  whose conductor $\frak{f}_A$ is unramified in $E$. We can then write
\begin{equation}\label{factorization}
\frak{f}_A=\n^{\mbox{\tiny $+$}}\cdot \n^{\mbox{\tiny $-$}}
\end{equation}
where $\n^{\mbox{\tiny $-$}}$ denotes the largest divisor of $\frak{f}_A$ divisible only by primes which are inert in $E$. We let $\Sigma^{\mbox{\tiny $-$}}$ represent the set of prime divisors of $\n^{\mbox{\tiny $-$}}$.
\begin{assumptions}\label{assum: square-free}
The ideal $\n^{\mbox{\tiny $-$}}$ is \emph{square-free}. 
\end{assumptions}

\begin{definition}
 A subset $\cal{S}\ \subseteq\ \Sigma^{\mbox{\tiny $-$}}\cap \Sigma_p$ is \emph{incoherent} for the triple $(A_{/F},\hspace{.5mm} E/F,\hspace{.5mm} p)$ if
\[
 \lvert \cal{S}\rvert\not\equiv \lvert \Sigma^{\mbox{\tiny $-$}}\cup\Sigma_\infty\rvert\pmod{2}.
\]
Let $w\in(\Sigma^{\mbox{\tiny $-$}}\cup\Sigma_\infty)\setminus\cal{S}$ be an auxiliary place  and set 
\[
S:=\cal{S}\cup\{w\}.
\]
We say that $S$ is \emph{coherent} for the quadruple $(A_{/F},\hspace{.5mm} E/F,\hspace{.5mm} p,\hspace{.5mm} \cal{S})$. Moreover, if $w\not\in\Sigma_\infty$ (resp. $w\in\Sigma_\infty$) we call $S$ definite (resp. indefinite).
\end{definition}
\begin{remark}
    The sign of the functional equation $\epsilon(A/E)$ is $(-1)^{\lvert \cal{S}\rvert+1}$ for any incoherent set $\cal{S}$.
\end{remark}

Fix an incoherent set $\cal{S}$ for the triple $(A_{/F},\hspace{.5mm} E/F,\hspace{.5mm}p)$ and a coherent set  $S$ for the quadruple $(A_{/F},\hspace{.5mm} E/F,\hspace{.5mm}p,\hspace{.5mm}\cal{S})$.  We let $r\in\bb{N}$ be the positive integer such that  
\[
\lvert\cal{S}\rvert=r-1\quad\text{and}\quad \lvert S\rvert=r.
\]
Under our assumptions, there exists a quaternion algebra $B/F$ with ramification set 
\[
(\Sigma^{\mbox{\tiny $-$}}\cup\Sigma_\infty)\setminus S.
\]
In particular, $B$ splits at every place in $S$, and it is definite (resp. indefinite) exactly when $S$ is definite (resp. indefinite).  We denote by $G=B^\times/F^\times$ the associated $F$-algebraic group of units modulo center and set 
\[
\frak{f}:=\frak{f}_A/\mrm{disc}(B).
\]
By hypothesis, there is a cuspidal automorphic representation $\Pi_A$ of $\PGL_2(\A_F)$ whose Hecke eigenvalues count points of $A_{/F}$ over finite fields, and admitting a Jacquet--Langlands transfer $\pi_A$ to $G(\A_F)$ of conductor $\frak{f}$. 

\medskip
Fix $\frak{c}\subseteq\cal{O}_F$ an ideal prime to $\frak{f}_A$ and let $\cal{O}_\frak{c}:=\cal{O}_F+\frak{c}\cal{O}_E$ be the associated order of $\cal{O}_E$. Since the prime divisors of $\n^{\mbox{\tiny $+$}}$ split in $E$, we can choose an $F$-algebra embedding $\psi\colon E\hookrightarrow B$ such that
\begin{equation}\label{choice of optimal embedding}
\psi(E)\cap R_{\n^{\mbox{\tiny $+$}}}=\psi(\cal{O}_\frak{c})
\end{equation}
for some  Eichler order $R_{\n^{\mbox{\tiny $+$}}}$ of level $\n^{\mbox{\tiny $+$}}$.
By a slight abuse of notation, we also write $\psi\colon T\hookrightarrow G$ for the induced morphism of $F$-algebraic groups where $T=E^\times/F^\times$. We denote by $L_\frak{c}$ the abelian extension  of $E$ whose Galois group $\cal{G}_\frak{c} :=\mrm{Gal}(L_\frak{c}/E)$ satisfies
\[
\mrm{rec}_E\colon T(F)\backslash T(\bb{A}^\infty)/\widehat{\cal{O}}^\times_\frak{c}\overset{\sim}{\longrightarrow}\cal{G}_\frak{c}.
\]
For our purposes, especially Lemma \ref{first vanishing of invariants} and Proposition \ref{NOinvariants}, it is convenient to require that the residual representation $\overline{\varrho}\colon \mrm{Gal}_F\to \mrm{GL}_2(\bb{F}_p)$ describing the action of the absolute Galois group $\mrm{Gal}_F$ on the $p$-torsion points of $A_{/F}$  has large non-solvable image :
 \begin{assumptions}\label{global classes assumptions}
The image $\overline{\varrho}(\mrm{Gal}_F)$ contains $\mrm{SL}_2(\bb{F}_p)$ and $p\ge5$.
 \end{assumptions}
Let $\chi\colon \cal{G}_\frak{c}\to\overline{\Q}^\times$ be a character and consider the tensor product of local cohomology groups
\[
\mrm{H}^1_f(E_{\otimes,\cal{S}},V_p(A)):=\bigotimes_{\p\in\cal{S}}\hspace{1mm}\mrm{H}^1_f(E_\p, V_p(A)).
\]
\begin{thmx}
Under Assumptions \ref{assum: square-free}, \ref{global classes assumptions},  we construct a ``plectic Heegner class''
\[
\kappa^\chi_{A,\cal{S}}\hspace{1mm}\in\hspace{1mm}\mrm{H}^1_f(E_{\otimes,\cal{S}},V_p(A))\otimes_{\Q_p} \mrm{H}^1(L_\frak{c}, V_p(A))^\chi
\]
belonging to a tensor product of local and global Galois cohomology groups.
\end{thmx}
We refer to Corollary \ref{field of definition of plectic Heegner class} in the body of the article. As the notation suggests, since $\frak{c}$ is coprime to $\frak{f}_A$, the plectic Heegner class $\kappa^\chi_{A,\cal{S}}$ depends only on the incoherent set $\cal{S}$ and not on the choice of the auxiliary place $w$ used to form $S=\cal{S}\cup\{w\}$.  

\bigskip
In Section \ref{Comparison with previous constructions} we compare plectic Heegner classes to previous constructions. We refer to Section \ref{section: review} for a detailed explanation of our strategy. We report below the main consequences which hold under the additional hypotheses:
 \begin{assumptions}\label{for semisimplicity}
Every $\frak{q}\mid \n^{\mbox{\tiny $+$}}$ has inertia degree one, does not divide the different of $F$, and 
\[
\mrm{ord}_\q(\n^{\mbox{\tiny $+$}})\le2.
\]
 \end{assumptions}
 \begin{remark}
     We use Assumption \ref{for semisimplicity} to deduce the semi-simplicity of a certain Hecke algebra as in \cite[Theorem 4.2]{Coleman-Edixhoven}. The key consequences for us are Corollary \ref{isolocalizations} and Lemma \ref{lemma: T to Totimes} which enable us to study the behavior of certain $p$-adic integrals in Section \ref{sect: On the behavior of certain integrals}.
 \end{remark}

\subsubsection{Relation to mock plectic invariants.}\label{Relation to mock plectic invariants}
When $F=\bb{Q}$, an incoherent set $\cal{S}$ contains at most one prime. If $\cal{S}=\emptyset$, the plectic Heegner class $\kappa^\chi_{A,\cal{S}}$ is simply a linear combination of (Kummer images of) classical Heegner points. While, if $\cal{S}=\{p\}$, the plectic Heegner class  is an element in 
\[
\kappa^\chi_{A,\cal{S}}\hspace{1mm}\in\hspace{1mm} \mrm{H}^1_f(E_p, V_p(A))\otimes \mrm{H}^1(L_\frak{c}, V_p(A))^\chi.
\]
Define $\varepsilon_p\in\{\pm1\}$ by setting $\varepsilon_p=1$ (resp. $-1$) if $A_{/\Q_p}$ has split (resp. non-split) multiplicative reduction. The ``partial Frobenius'' element $\sigma_p$, i.e., the generator of $\mrm{Gal}(E_p/\Q_p)$, acts on $\mrm{H}^1_f(E_p, V_p(A))$. We can consider the eigenspace $\mrm{H}^1_f(E_p, V_p(A))^{-\varepsilon_p}$ on which $\sigma_p$ acts as multiplication by $-\varepsilon_p$. In \cite{DarmonFornea} and \cite[Section 2.2.5]{Iwasawa4mock} a mock plectic invariant 
\[
\cal{Q}^\chi_{A,\cal{S}}\hspace{1mm}\in\hspace{1mm} \mrm{H}^1_f(E_p, V_p(A))^{-\varepsilon_p}\otimes \mrm{H}^1(L_\frak{c}, V_p(A))^\chi
\]
is associated to the triple $(A_{/\Q}, E/F, \chi)$. In perfect analogy with the setting of plectic Stark--Heegner points and plectic $p$-adic invariants (\cite{plecticHeegner}, End of Section 1.3), we find that: 
\begin{thmx}
Suppose $F=\bb{Q}$ and $\cal{S}=\{p\}$. Under Assumptions \ref{assum: square-free}, \ref{global classes assumptions}, $\&$ \ref{for semisimplicity}, plectic Heegner classes recover mock plectic invariants by the formula
\[
(1-\varepsilon_p\cdot\sigma_p)\hspace{1mm}\kappa^\chi_{A,\cal{S}}=\cal{Q}^\chi_{A,\cal{S}}.
\]
\end{thmx}

\subsubsection{Relation to plectic Heegner points.}\label{Relation to plectic Heegner points}
Let us suppose that the incoherent set $\cal{S}$ can be completed to a coherent set $S=\cal{S}\cup\{w\}$ such that the auxiliary place $w$ is $p$-adic. In this case, we can consider the plectic Heegner point attached to the triple $(A_{/F}, S, \chi)$
\[
\mrm{P}^\chi_{A,S}\hspace{1mm}\in\hspace{1mm}\mrm{H}^1_f(E_{\otimes,\cal{S}},V_p(A))\otimes \mrm{H}^1_f(E_w, V_p(A))^\chi
\]
 \cite[Definition 4.11]{plecticHeegner} -- an element in a tensor product of local Bloch--Kato Selmer groups. 

Since the $\OO_E$-prime ideal generated by $w$ splits completely in $L_\frak{c}$, the choice of an extension $\iota_w\colon L_\frak{c}\hookrightarrow E_w$ of the inclusion $E\hookrightarrow E_w$ determines a  map $\mrm{loc}_w\colon \mrm{H}^1(L_\frak{c}, V_p(A))\to \mrm{H}^1(E_w, V_p(A))$ which allows us to consider 
\[
\mrm{loc}_w(\kappa^\chi_{A,\cal{S}})\hspace{1mm}\in\hspace{1mm} \mrm{H}^1_f(E_{\otimes,\cal{S}},V_p(A))\otimes \mrm{H}^1(E_w, V_p(A))^\chi.
\]
\begin{assumptions}\label{local classes assumptions}
Suppose that $S=\cal{S}\cup\{w\}$ with auxiliary $p$-adic place $w$. We require that
\begin{itemize}
 \item [$(1)$] the local field $F_w$ does not contain $p$-th roots of unity,
    \item [$(2)$]  the Tate parameter $q_w\in F_w^\times$ of $A_{/F_w}$ satisfies 
    \[
    q_w\not\in (F_w^\times)^p.
    \]
\end{itemize}
 \end{assumptions}
 \begin{thmx}
Under Assumptions \ref{assum: square-free}, \ref{global classes assumptions}, $\&$ \ref{for semisimplicity},  \ref{local classes assumptions}, the localization at $w\in\Sigma_p$ belongs to
\[
\mrm{loc}_w(\kappa^\chi_{A,\cal{S}})\hspace{1mm}\in\hspace{1mm} \mrm{H}^1_f(E_{\otimes,\cal{S}},V_p(A))\otimes \mrm{H}^1_f(E_w, V_p(A))^\chi.
\]
Moreover, it recovers plectic Heegner points by the formula
\[
\mrm{loc}_w(\kappa^\chi_{A,\cal{S}})=\mrm{P}^\chi_{A,S}.
\]
\end{thmx}

 \begin{remark}
    We use Assumption \ref{local classes assumptions}  to obtain the vanishing $\mrm{H}^0(E_w,\overline{\varrho})=0$  for any unramified extension $E_w/F_w$ (Lemma \ref{local first vanishing of invariants}). Assumption \ref{local classes assumptions}(2) ensures that the local representation $\overline{\varrho}_{\lvert D_w}$ is not split, while Assumption \ref{local classes assumptions}(1) implies that the mod-$p$ cyclotomic character is non-trivial when restricted to the absolute Galois group of $E_w$. 
 \end{remark}

\subsection{Conjectures}\label{conjectures}
 In light of the comparison to previous constructions presented in Sections \ref{Relation to mock plectic invariants} $\&$ \ref{Relation to plectic Heegner points}, it is natural to extend the conjectures proposed in \cite[Section 1.4]{plecticHeegner}, \cite[Section 3.8]{DarmonFornea}, and \cite[Section 2.3]{Iwasawa4mock}.
To that end, we say that the plectic Heegner class $\kappa^\chi_{A,\cal{S}}$ is \emph{Selmer} if 
    \[
  \kappa^\chi_{A,\cal{S}}\hspace{1mm}\in\hspace{1mm}\mrm{H}^1_f(E_{\otimes,\cal{S}},V_p(A))\otimes \mrm{H}^1_f(L_\frak{c}, V_p(A))^\chi
    \]
    where $\mrm{H}^1_f(L_\frak{c}, V_p(A))\subseteq \mrm{H}^1(L_\frak{c}, V_p(A))$ denotes the global Bloch-Kato Selmer group.   
\begin{conjecture}
If the character $\chi\colon\G_{\mathfrak{c}}\to\overline{\Q}^\times$ is \emph{primitive}, the following equivalence holds
	\[
	\kappa^{\chi}_{A,\cal{S}}\ \text{ is Selmer}\quad\iff\quad r_\mrm{an}(A/E,\chi)\ge r.
	\]	
    In particular, if $r_\mrm{an}(A/E,\chi)< r$, then $\kappa^{\chi}_{A,\cal{S}}\not=0$.
\end{conjecture}

When plectic Heegner classes are Selmer, we expect them to be related to points on elliptic curves as follows: let $\cal{K}\colon A(L_\frak{c})\to \mrm{H}^1_f(L_\frak{c}, V_p(A))$ denote the Kummer map. Since the $\OO_E$-prime ideals generated by the elements of $\cal{S}$ split completely in the anticyclotomic extension $L_\frak{c}$, we can choose an extension $\iota_\mathfrak{p}\colon L_\frak{c}\hookrightarrow E_\mathfrak{p}$ of the inclusion $E\hookrightarrow E_{\mathfrak{p}}$ for every $\mathfrak{p}\in \cal{S}$. Consider the homomorphism 
\[
\det\colon \wedge^rA(L_\frak{c})\longrightarrow  \mrm{H}^1_f(E_{\otimes,\cal{S}},V_p(A))\otimes \mrm{H}^1_f(L_\frak{c}, V_p(A))
\]
given by 
\[
\det\big(P_1\wedge\dots\wedge P_{r}\big):=\det \begin{pmatrix}
		\iota_{\mathfrak{p}_1}(P_1)&\dots& \iota_{\mathfrak{p}_{r-1}}(P_1)&\cal{K}(P_1)\\
        \vdots&\ddots&\vdots&\vdots\\
	\vdots&\ddots&\vdots&\vdots\\
	\iota_{\mathfrak{p}_{1}}(P_r)&\dots& \iota_{\mathfrak{p}_{r-1}}(P_r)&\cal{K}(P_r)
	\end{pmatrix}.
\]

\begin{conjecture}
If $r_\mrm{alg}(A/E,\chi)\ge r$, then
    \[ 	\kappa^{\chi}_{A,\cal{S}}\not=0\quad\iff\quad r_\mrm{alg}(A/E,\chi)=r.
 	\]	
    Moreover, there exists $w_{A,\cal{S}}^\chi\in \exterior{r} A(L_\frak{c})^\chi$ such that $\kappa^{\chi}_{A,\cal{S}}=\det\big(w_{A,\cal{S}}^\chi\big)$.
\end{conjecture}

\subsection{Overview of the construction}\label{Intro: overview of the construction}
In Section \ref{Intro: pHc} we fixed an incoherent set $\cal{S}$ for the triple $(A_{/F}, E/F, p)$, a coherent set $S=\cal{S}\cup\{w\}$, and an embedding $\psi\colon T\hookrightarrow G$ of $F$-algebraic groups. In Section \ref{Towers of Shimura curves} we will introduce the relevant tower of Shimura curves of increasing level at the primes in $\cal{S}$
\[\xymatrix{
X_0& X_1\ar[l]_-{\pi_1}\ar[l]&\cdots\ar[l]_-{\pi_1}& X_m\ar[l]&\ar[l]\cdots
}\] 
The $E$-CM points of $X_m$   ``unramified at $w$'' can be described by a double coset space 
\[
\mrm{CM}^w_E(X_m)= G(F)\backslash \big(\cal{H}_w^E\times G(\A^{w,\infty})/K_m\big),
\]
and the theorem of complex multiplication ensures that they are defined over the maximal abelian extension ${^{w}}E_\mrm{ab}$ of $E$ where $w$ is totally split.

\smallskip
 Letting $J_m$ denote the  Jacobian of $X_m$, the direct limit of Tate modules $\varinjlim_m T_p(J_m)$ with respect to the maps $\pi_1^*\colon J_m\to J_{m+1}$ inherits a right action of 
\begin{itemize}
    \item [$\bfcdot$] diamond operators $D_\cal{S}:=\varprojlim_m K_m^\circ/K_m$, denoted $\{\llangle a\rrangle\ :\ a\in \cal{O}^\times_\cal{S}\}$,
    \item [$\bfcdot$] and Hecke operators
\[
\big\{\textbf{T}_\q\ :\ \q\nmid\frak{f}_A\big\}\cup\big\{\textbf{U}_\q\ :\ \q\mid\n^{\mbox{\tiny $+$}}\big\}\cup\big\{\textbf{T}_{\varpi_\p}\ :\ \p\in\cal{S},\ \varpi_\p\in\cal{O}_\p\ \text{uniformizer} \big\}.
\] 
\end{itemize}

\subsubsection{Automorphic functions.}
The data above can be packaged into a single automorphic function  $\Phi_\cal{S}\in \cal{A}_{w,\cal{S}}\big(K_0;\hspace{0.5mm}\cal{H}_w^E,\hspace{0.5mm}\cal{M}_\cal{S}\big)$ of full level at $\cal{S}$ (denoted $\Phi_\cal{S}^\mrm{glo}$ in Corollary \ref{distinguished autfun}), whose coefficients $\cal{M}_\cal{S}=\cal{M}_\cal{S}^\mrm{glo}$ are certain $p$-adic measures on  $\cal{X}_\cal{S}:=\varprojlim_m K_0/K_m$  valued in  the colimit
\begin{equation}\label{intro: colimit}
\varinjlim_m\ \mrm{H}^1\big({^{w}}E_\mrm{ab}, T_p(J_m)\big).
\end{equation}
\begin{remark}
  The $p$-adic space $\cal{X}_\cal{S}$  admits a right action $\cal{X}_\cal{S}\times D_\cal{S}\rightarrow \cal{X}_\cal{S}$, $(\gamma,a)\mapsto \gamma\hspace{0.5mm}\bfcdot\hspace{0.5mm} a$, of diamond operators induced by \emph{right} multiplication. A measure $\mu$ belongs to $\cal{M}_\cal{S}$ if  $\mu(U)\llangle a\rrangle=\mu(U\hspace{0.5mm}\bfcdot\hspace{1mm} a)$ for all compact open $U\subseteq\cal{X}_\cal{S}$ and all diamond operators $a \in D_\cal{S}$. We refer to \eqref{iso1} for a description of $\cal{X}_\cal{S}$ in terms of primitive vectors.
\end{remark}
 It is possible -- but, as explained in Sections \ref{Semigroup action on p-adic domains} $\&$ \ref{action on distribution},  delicate at the primes in $\cal{S}$ -- to define a natural left action of Hecke operators 
\[
\big\{T_\q\ :\ \q\nmid\frak{f}_A\big\}\cup\big\{U_\q\ :\ \q\mid\n^{\mbox{\tiny $+$}}\big\}\cup\big\{T_{\varpi_\p}\ :\ \p\in\cal{S},\ \varpi_\p\in\cal{O}_\p\ \text{uniformizer} \big\}
\] 
on the space of automorphic functions $\cal{A}_{w,\cal{S}}\big(K_0;\hspace{0.5mm}\cal{H}_w^E,\hspace{0.5mm}\cal{M}_\cal{S}\big)$ (see Corollary \ref{Hecke-action of families}). The distinguished function  $\Phi_\cal{S}$ has the remarkable property of being a Hecke eigenfunction (Corollary \ref{Hecke action on special form}): 
\begin{equation}\label{intro: eigenfunction}
    T_\q.\Phi_\cal{S}=\Phi_\cal{S}.\mbf{T}_\q,\qquad  U_\q.\Phi_\cal{S}=\Phi_\cal{S}.\mbf{U}_\q,\qquad T_{\varpi_\p}.\Phi_\cal{S}=\Phi_\cal{S}.\mbf{U}_{\varpi_\p}.
    \end{equation}
    Note that in the equations \eqref{intro: eigenfunction} the LHS denotes the Hecke action on automorphic functions, while the RHS represents the Hecke action on the coefficients through the Hecke module \eqref{intro: colimit}.  
The automorphic function $\Phi_\cal{S}$ is a key ingredient in the definition of a ``plectic Abel--Jacobi map'' for a (mock) quaternionic Shimura variety.

\subsubsection{Plectic Abel--Jacobi map.}  For any $v\in S$, consider the $G_v$-space $\cal{H}_v:=\bb{P}^1(E_v)\setminus\bb{P}^1(F_v)$ with $G_v$-action obtained through the twisting of the natural M\"obius transformations action by 
\begin{equation}\label{twisting character}
\chi_v\colon G_v\longrightarrow \mrm{Gal}(E_v/F_v),\qquad g\mapsto (\sigma_v)^{\zeta_v(g)}
\end{equation}
where $\sigma_v\in \mrm{Gal}(E_v/F_v)$ is the generator, and $\zeta_v\colon G_v\rightarrow \Z/2\Z$ is the homomorphism
\[
\zeta_v(g)=\begin{cases}
   \mrm{ord}_v(\det(g))& \text{if } v\not\in\Sigma_\infty\\
   0&\text{if } v\in\Sigma_\infty.
\end{cases}
\]  
Set $\cal{H}_\cal{S}=\prod_{v\in S}\cal{H}_v$ and let $\cal{H}_S^E\subset \cal{H}_\cal{S}$ denote the $G(F)$-orbit of a fixed point for the action of $T$ through $\psi\colon T\hookrightarrow G$. We consider the (mock) quaternionic Shimura variety 
\[
X_G^S(\frak{f}):=G(F)\backslash \big(\cal{H}_S\times G(\A^{S,\infty})/K_0^S\big)
\]
whose $E$-CM points ``unramified at $S$'' are defined by the double coset space
\[
\mrm{CM}_E\big(X_G^S(\frak{f})\big):=G(F)\backslash \big(\cal{H}_S^E\times G(\A^{S,\infty})/K_0^S\big).
\]
\begin{remark}
When $S\subseteq \Sigma_p$, the set $X_G^S(\frak{f})$ consists of $p$-adic points in a $r$-dimensional quaternionic Shimura variety attached to the indefinite quaternion algebra obtained from $B$ by switching invariants at $S$ and a subset of $\Sigma_\infty$ of the same cardinality.
\end{remark}
 Since the conductor of $A_{/F}$ is divisible by every $\p\in\cal{S}$, the descent results of Section \ref{Descent of Galois cohomology classes}, in particular Corollary \ref{cor: descent}, allow us to prove the following:
 \begin{thmx}
 Under Assumptions \ref{assum: square-free}, \ref{global classes assumptions}, there exists a ``plectic Abel--Jacobi map''
\[
\varphi_A\colon\mrm{CM}_E\big(X_G^S(\frak{f})\big)\longrightarrow \mrm{H}^1_f(E_{\otimes,\cal{S}},V_p(A))\otimes_{\Q_p} \mrm{H}^1({^{w}}E_\mrm{ab}, V_p(A)).
\]
\end{thmx}
We refer to \eqref{varphi localized}, Definition \ref{def: Exotic AJ map}, and Corollary \ref{piA-isotypic subspace} for more details.  Here we can say that it is defined by an explicit expression of the form 
\[
[\tau_S,h]\mapsto (\exp_{q_\cal{S}}^\mrm{Tate}\otimes\hspace{0.3mm}\mrm{pr}_A)\left(\int_{\bb{X}_{\mrm{v}_\cal{S}}}f_{\tau_\cal{S},_\varpi}(\gamma)\hspace{1mm}\mrm{d}\mu^{\varpi}_{(\tau_{S},h)}(\gamma)\otimes 1\right)
\]
\begin{itemize}
\item [$\bfcdot$] where $\varpi=\{\varpi_\p\}_{\p\in\cal{S}}$ denotes a choice of uniformizers $\varpi_\p\in \cal{O}_\p$,
    \item [$\bfcdot$] $f_{\tau_\cal{S},_\varpi}$ is a function depending on $\varpi$ and the $\cal{S}$-component $\tau_\cal{S}$ of $\tau_S$,
     \item [$\bfcdot$] $\bb{X}_{\mrm{v}_\cal{S}}$ is a translate of $\cal{X}_\cal{S}$ depending only on the reduction $\mrm{v}_\cal{S}=\mrm{red}(\tau_\cal{S})$ of $\tau_\cal{S}$ in the Bruhat--Tits building of $G_\cal{S}$, 
      \item [$\bfcdot$] $\mu^{\varpi}_{(\tau_{S},h)}$ is a $p$-adic measure constructed from $\Phi_\cal{S}$ and depending on $\varpi$, $\mrm{v}_\cal{S}$, the $w$-component of $\tau_S$, and $h\in G(\A^{S,\infty})$.
\end{itemize}  
As the notation suggests, by Corollary \ref{independence of choice of uniformizers} the function $\varphi_A$ does not depend on the choice of uniformizers $\varpi$. Moreover, in Theorem \ref{thm: Galois invariance} we show that for every $x\in \mrm{CM}_E\big(X_G^S(\frak{f})\big)$ there exists an explicit finite abelian extension $L_x$ of $E$, contained in ${^{w}}E_\mrm{ab}$, such that
\[
\varphi_A(x)\ \in\ \mrm{H}^1_f(E_{\otimes,\cal{S}},V_p(A))\otimes \mrm{H}^1(L_x, V_p(A)).
\]
Plectic Heegner classes are obtained as linear combinations of values of $\varphi_A$.

 \vspace{2em}
 \textbf{Acknowledgements}

\noindent This article develops an idea proposed  by Henri Darmon for refining mock plectic invariants.  I am deeply grateful to him for generously sharing his ideas during our stay at MSRI/SLMath, Berkeley, in the Spring of 2023. Our joint work brought renewed energy and inspiration to the exploration of plectic points. I  warmly thank Lennart Gehrmann and Mart\'i Roset for many exchanges on the subject, especially during more challenging phases of the project. Finally, I thank Federico Bambozzi and James Taylor for patiently explaining me aspects of $p$-adic geometry.

\section{Preliminaries}

Let $F$ be a totally real number field  with ring of integers $\cal{O}_{F}$. 
For any rational place $v$, we denote by $\Sigma_v$ the set of $v$-adic places of $F$. If $\q\subset \cal{O}_{F}$ is a prime ideal, we write $\OO_\q$ for the completion of $\cal{O}_{F}$ at $\q$.
We let $\A$ stand for the ring of adeles of $F$, and $\A^{\Sigma}$ for the $\Sigma$-truncated adeles
for any finite set of places $\Sigma$ of $F$. For notational convenience, we set $\A^\infty=\A^{\Sigma_\infty}$.
If $H$ is an $F$-algebraic group and $w$ is a place of $F$, we write $H_w=H(F_w)$ and put $H_\Sigma=\prod_{w\in \Sigma} H_v$.

\smallskip
Let $B/F$ be a quaternion algebra of discriminant $\disc(B)$. Denote by $G=B^\times/F^\times$ the associated $F$-algebraic group of units modulo center. For every prime $\q\nmid \disc(B)$, we fix an isomorphism of $B_\q$ with the algebra $\mrm{M}_2(F_\q)$ of two-by-two matrices over $F_\q$ to obtain an identification
\begin{equation}\label{choseniso}
G_\q= \PGL_2(F_\q)\qquad \forall\ \q\nmid \disc(B).
\end{equation}
Given an integer $m\ge0$ and a prime $\q\nmid \disc(B)$, define the compact open subgroups 
\[
K_{\q^m}\subseteq K^\circ_{\q^m}\subset G_\q
\]
as the images  under (\ref{choseniso}) of the subgroups
\[
\Big\{g \in \GL_2(\OO_\q)\ \Big\vert\ g \mbox{ upper unipotent} \hspace{-2.5mm}\pmod{\q^m}\Big\}\subseteq\Big\{g \in \GL_2(\OO_\q)\ \Big\vert\ g \mbox{ upper triangular}\hspace{-2.5mm} \pmod{\q^m}\Big\}.
\]
For $\q\mid \disc(B)$ define $K_{\q^0}=K_{\q^0}^\circ$ to be the image in $G_\q$ of the group of units of the maximal order in $B_\q$.
Moreover, if $\n\subseteq\OO_F$ is a non-zero ideal coprime to $\disc(B)$, $\Sigma$ is a finite set of primes of $F$, we set
\[
K(\n)^{\Sigma}:=\prod_{\q\notin \Sigma}K_{\q^{\ord_\q(\n)}}\hspace{2mm} \subseteq \hspace{2mm} K^\circ(\n)^{\Sigma}:=\prod_{\q\notin \Sigma}K^\circ_{\q^{\ord_\q(\n)}}.
\]
When $\Sigma=\emptyset$ we simply write $K(\n)=K(\n)^{\emptyset}$ and $K^\circ(\n)=K^\circ(\n)^{\emptyset}$.


\subsubsection{$p$-adic Measures and Distributions.}\label{generalities on measures and distributions}
Let $\cal{X}$ be a profinite topological space, $M$ and $N$ two $p$-adically separated and complete $\Z_p$-modules. Denote by $\mrm{LC}(\cal{X},M)\subseteq \cal{C}(\cal{X},M)$ the inclusion of locally constant functions in the space of continuous functions. The associated spaces of measures and distributions are
\[
\cal{M}(\cal{X};M,N):=\mrm{Hom}_{\Z_p}(\mrm{LC}(\cal{X},M),N),\qquad \cal{D}(\cal{X};M,N):=\mrm{Hom}_{\Z_p}^\mrm{ct}(\cal{C}(\cal{X},M),N).
\]
We recall some useful properties:
\begin{itemize}
    \item[$\bfcdot$] By
\cite[Lemma 4.3]{plecticHeegner} the map 
\begin{equation}\label{measures=distributions}
    \cal{D}(\cal{X};M,N)\overset{\sim}{\longrightarrow} \cal{M}(\cal{X};M,N)
\end{equation}
induced by the inclusion $\mrm{LC}(\cal{X},M)\subseteq \cal{C}(\cal{X},M)$ is an isomorphism.

    \item[$\bfcdot$] As $N$ is $p$-adically complete and separated, so is $\mrm{Hom}_{\Z_p}(M,N)$. Then, the isomorphism $\mrm{LC}(\cal{X},M)\cong\mrm{LC}(\cal{X},\Z_p)\otimes_{\Z_p}M$ and the tensor-hom adjunction induce the following commuting diagram
\begin{equation}\label{diagram of iso}\xymatrix{
\cal{D}(\cal{X};M,N)\ar[r]^-\sim\ar[d]_\sim&\cal{M}(\cal{X};M,N)\ar[d]_-\sim\\
\cal{D}\big(\cal{X};\Z_p,\mrm{Hom}_{\Z_p}(M,N)\big)\ar[r]^-\sim&\cal{M}\big(\cal{X};\Z_p,\mrm{Hom}_{\Z_p}(M, N)\big).
}\end{equation}


    \item[$\bfcdot$]If $\cal{X}$ is a product $\cal{X}_1\times \cal{X}_2$ of profinite topological spaces, then \eqref{diagram of iso} together with the factorization $\mrm{LC}(\cal{X},\Z_p)\cong \mrm{LC}(\cal{X}_1,\Z_p)\otimes_{\Z_p}\mrm{LC}(\cal{X}_2,\Z_p)$
give
\begin{equation}\label{factorization distributions}
\xymatrix{
\cal{D}(\cal{X}_1\times\cal{X}_2;M,N)\ar[r]^-\sim\ar[d]_\sim&\cal{M}(\cal{X}_1\times\cal{X}_2;M,N)\ar[d]_-\sim\\
\cal{D}\big(\cal{X}_1;\Z_p,\cal{D}(\cal{X}_2;M,N)\big)\ar[r]^-\sim&\cal{M}\big(\cal{X}_1;\Z_p,\cal{M}(\cal{X}_2;M, N)\big).
}\end{equation}
\end{itemize}

\subsection{Shimura curves}\label{Shimura curves}
Fix an embedding $\iota_\infty\colon \overline{\Q}\to \bb{C}$, and one embedding $\iota_\ell\colon \overline{\Q}\to \overline{\Q}_\ell$ for every rational prime $\ell$. Thus, 
\[
\mrm{Hom}_{\Q-\mrm{alg}}(F,\overline{\Q}_\ell)=\mrm{Hom}_{\Q-\mrm{alg}}(F,\bb{C}),
\]
and we obtain a surjective function $\Sigma_\infty\twoheadrightarrow \Sigma_\ell$ for every $\ell$.

\medskip
\noindent As in Section \ref{Intro: pHc}, we have fixed an incoherent set $\cal{S}$ for the triple $(A_{/F},\hspace{.5mm} E/F,\hspace{.5mm}p)$, a coherent set  $S=\cal{S}\cup\{w\}$ for the quadruple $(A_{/F},\hspace{.5mm} E/F,\hspace{.5mm}p,\hspace{.5mm}\cal{S})$, and an embedding $\psi\colon T\hookrightarrow G$ of $F$-algebraic groups. We consider the $G_w$-space $\cal{H}_w:=\bb{P}^1(E_w)\setminus\bb{P}^1(F_w)$ with $G_w$-action obtained through the twisting of the natural M\"obius transformations action by the character \eqref{twisting character}.
Setting $G(\A^{w,\infty}):=
    G(\A^{\{w\}\cup\Sigma_\infty})$, for any  compact open subgroup $K\subset G(\A^{w,\infty})$ there is a Shimura curve $X_K$ defined over $F$ whose $E_w$-valued points $X_K(E_w)$ are identified to a double quotient:
\begin{equation}\label{def:Shimura curve}
\phi_w^\mrm{CD}\colon G(F)\backslash \big(\cal{H}_w\times G(\A^{w,\infty})/K\big)\cup\{\text{cusps}\}\overset{\sim}{\longrightarrow} X_K(E_w)
\end{equation}
 -- see for example \cite[Sections 1.3-1.4]{NekCanadian}.

\medskip
\noindent     If the admissible set $S$ is indefinite, the auxiliary place $w\in S$ is Archimedean and the Shimura curve $X_K$ arise from the Shimura datum $(G,\cal{H}_w)$. While, if $S=\cal{S}\cup\{w\}$ is definite with associated definite quaternion algebra $B/F$, it is possible to choose an Archimedean place $\nu\in\Sigma_\infty$ mapping to $w\in\Sigma_\ell$ under the fixed surjection $\Sigma_\infty\twoheadrightarrow\Sigma_\ell$ and define the indefinite admissible set
    \[
    S':=(S\setminus\{w\})\cup\{\nu\}.
    \]
    The quaternion algebra $B'/F$ with ramification set $(\Sigma^{\mbox{\tiny $-$}}\cup\Sigma_\infty)\setminus S'$ is obtained from $B/F$ by switching invariants at $w$ and $\nu$. Let $G'$ be the associated $F$-algebraic group of units modulo center, then the Shimura curve $X_K$  arises from the Shimura datum $(G',\cal{H}_\nu)$.

    \begin{remark}
    The modularity of $A_{/F}$ and the Jacquet--Langlands correspondence ensure the existence of a non-constant $F$-rational morphism $\varphi\colon X_K\to A$ for an appropriate level structure $K$. When $w\not\in\Sigma_\infty$ the Shimura curve $X_K$ admits an Atkin-Lehner involution $W_w\colon X_K\to X_K$ associated to the ramified prime $w$ of $B'$ such that 
    \[
    \varphi\circ W_w=\varepsilon_w\cdot \varphi
    \]
    where $\varepsilon_w\in\{\pm1\}$ is $+1$ (resp. $-1$) if $A_{/F_w}$ has split (resp. non-split) multiplicative reduction -- for example, see \cite[Theorem 1.2]{MumfordHeegner}. Moreover, the uniformization morphism $\phi_w^\mrm{CD}$ satisfies 
    \[
    \phi_w^\mrm{CD}(-^{\sigma_w})=W_w\circ\phi_w^\mrm{CD}(-)^{\sigma_w}
    \]
   where $\sigma_w\in\mrm{Gal}(E_w/F_w)$ is the generator (\cite{CDuniformization}, Theorem 4.7). We deduce that the composition $\varphi_w:=\varphi\circ\phi_w^\mrm{CD}$ satisfies
    \begin{equation}\label{Galois equivariance CD parametrization}
    \varphi_w(-^{\sigma_w})=\varepsilon_w\cdot \varphi_w(-).
    \end{equation}
    \end{remark}

\subsubsection{CM points.} 
The embedding $\psi$  let $T(F)$ act on $\cal{H}_w$ with two fixed points $z_\psi,\overline{z}_\psi$.  Let $\cal{H}_w^E$ denote  the $G(F)$-orbit of $z_\psi$ in $\cal{H}_w$. For $\Box\in\{\text{loc}, \text{glo}\}$ we set 
 \begin{equation}\label{loc-glo upper half}
 \cal{H}_w^\Box:=\begin{cases}
    \cal{H}_w&\Box=\text{loc}\\
     \cal{H}_w^E&\Box=\text{glo}.\\
\end{cases}
\end{equation}
The $E$-CM points of $X_K$ ``unramified at $w$''  are described by the quotient
\begin{equation}\label{uniformization CM points}
\mrm{CM}^w_E(X_K)= G(F)\backslash \big(\cal{H}_w^E\times G(\A^{w,\infty})/K\big).
\end{equation}
Let ${^{w}}E_{\mrm{ab}}$ denote the maximal abelian extension of $E$ where $w$ is totally split.
The theory of complex multiplication implies that  all points in $\mrm{CM}^w_E$ are defined over ${^{w}}E_{\mrm{ab}}$, i.e.,
\[
\mrm{CM}^w_E(X_K)\subseteq \mrm{Im}\Big(X_K({^{w}}E_{\mrm{ab}})\longhookrightarrow X_K(E_w)\Big)
\]
where the inclusion ${^{w}}E_{\mrm{ab}}\hookrightarrow E_w$ is determined by the fixed embedding $\iota_w\colon \overline{\Q}\hookrightarrow \overline{\bb{Q}}_w$.
Moreover, normalizing the reciprocity map of class field theory $\mrm{rec}_E\colon T(F)\backslash T(\A^\infty)\to \mrm{Gal}(E^\mrm{ab}/E)$ by matching uniformizers with geometric Frobenius elements, the Galois action is described by Shimura's reciprocity law  
\begin{equation}\label{reciprocityLAW}
([z_\psi,b])^{\mrm{rec}_E(u)}=[z_\psi,\psi(u)^{w}b]\qquad \forall\ u\in T(\A^{\infty})
\end{equation}
where $\psi(u)^{w}$ denotes the projection to $G(\A^{w,\infty})$ (for instance \cite[Section 1.8]{NekCanadian}).

\subsubsection{Hecke correspondences.}
Let $K,K'\subset G(\A^{w,\infty})$ be compact open subgroups, $g\in G(\A^{w,\infty})$. We consider the Hecke correspondence $T_g=[KgK']\colon X_K \dashrightarrow X_{K'}$  described by the diagram
\[\xymatrix{
X_{K\cap g K'g^{-1}}\ar[r]^{[\cdot g]}\ar[d]^{\mrm{pr}}& X_{g^{-1}Kg\cap K'}\ar[d]^{\mrm{pr}'}\\
X_K\ar@{-->}[r]^{T_g}&X_{K'}
}\]
where $[\cdot g]\colon [x,b]\mapsto [x,bg]$. In this normalization,  $T_g$ acts on non-cuspidal points by the formula
\[
T_g\colon [x,b]_K\mapsto \sum_{j\in I}[x,bg_j]_{K'}\qquad\text{where}\qquad KgK'=\coprod_{j\in I} g_jK'.
\]
\begin{remark}
    Fix $K\subset G(\A^{w,\infty})$. The free group generated by the set $\{[KgK]\ \lvert\ g\in G(\A^{w,\infty})\}$ admits a ring structure (for example \cite[Section 3.1]{IntroShimura}).
\end{remark}
\begin{definition}\label{def: Hecke operators}
Let $K=\prod_\frak{q}K_\frak{q}\subset G(\A^{w,\infty})$ be a compact open subgroup. For every prime ideal $\frak{q}\nmid\mrm{disc}(B)$ and uniformizer $\varpi\in \cal{O}_\frak{q}$, we consider the double coset
\[
 \big[K\cdot \eta_\varpi\cdot K\big]\qquad\text{where}\qquad \eta_\varpi=\begin{pmatrix}
    1&0\\0&\varpi
\end{pmatrix}.
\]
Let $m\ge1$. We will write 
\[
 \big[K\cdot \eta_\varpi\cdot K\big]=\begin{cases}
    T_\frak{q}&\text{if}\hspace{3mm} K_\frak{q}=K_{\frak{q}^0},\\
U_{\varpi}&\text{if}\hspace{3mm}K_{\frak{q}^m}\subseteq K_\frak{q}\subseteq K_{\frak{q}^m}^\circ.
\end{cases}
\]
   As the notation suggests, $T_{\frak{q}}$ is independent of the choice of uniformizer $\varpi\in \cal{O}_\frak{q}$. Moreover, whenever we want to stress the dependence on $m\ge1$, we will write ${_{m}}U_{\varpi}$ instead of $U_{\varpi}$.
These double cosets admit standard decompositions into left-cosets
\begin{equation}\label{cosets-decompositions}
\begin{split}
&T_\frak{q}=\coprod_{\bar{a}\in\bb{P}^1(\kappa_\frak{q})}\sigma_aK\qquad\text{where}\qquad \sigma_a=\begin{cases}
    \begin{pmatrix}
    \varpi&a\\0&1
\end{pmatrix}&\text{if}\ \bar{a}\not=\infty,\\
\\
\begin{pmatrix}
    1&0\\0&\varpi
\end{pmatrix}&\text{if}\ \bar{a}=\infty,
\end{cases}\\
\\
&U_\varpi=\coprod_{\bar{a}\in\kappa_\frak{q}}\widehat{\sigma}_{m,a}K \qquad\text{where}\qquad
\widehat{\sigma}_{m,a}=\begin{pmatrix}
    1&0\\a\varpi^m&\varpi
\end{pmatrix}
\end{split}
\end{equation}
where the coset representatives differ from the identity only at the place $\q$.
\end{definition}

\subsection{Towers of Shimura curves}\label{Towers of Shimura curves}
Fix a subset $\Sigma\subseteq\cal{S}$.
For $m\ge0$ we consider the subgroups of $G(\A^{w,\infty})$
\[
K_m:=K^\circ(\n^{\mbox{\tiny $+$}})\cap K\big(p_\Sigma^m\big)\cap G(\A^{w,\infty})\hspace{2mm} \subseteq\hspace{2mm} 
K_m^\circ:=K^\circ\big(\n^{\mbox{\tiny $+$}}\cdot p_\Sigma^m\big)\cap G(\A^{w,\infty}),
\]
and the associated Shimura curves
\[
X_m:=X_{K_m},\qquad X_m^{\circ}:=X_{K^\circ_m}.
\]
 Since $K_m$ is a normal subgroup of $K_m^\circ$, its cosets $D_{\Sigma,m}:=K_m^\circ/K_m$ form a group. The inclusion $K_{m+1}^\circ\subseteq K_m^\circ$ induces a homomorphism $D_{\Sigma, m+1}\to D_{\Sigma, m}$ which, under the identification
\begin{equation}\label{explict description diamond}
\zeta\colon D_{\Sigma, m}\overset{\sim}{\longrightarrow} (\cal{O}_{\Sigma}/p_{\Sigma}^m)^\times,\qquad\begin{pmatrix}
    a&b\\c&d
\end{pmatrix}\mapsto a_{\Sigma}^{-1}\cdot d_{\Sigma},
\end{equation}
corresponds to the natural quotient map $(\cal{O}_{\Sigma}/p_{\Sigma}^{m+1})^\times\twoheadlongrightarrow (\cal{O}_{\Sigma}/p_{\Sigma}^m)^\times$.  It is convenient to set
\begin{equation}\label{def: gamma_a}
 \gamma_a:=\begin{pmatrix}
    1&0\\0&a
\end{pmatrix}\ \in\ K^\circ_m\qquad \forall \
 a\in \cal{O}_{\Sigma}^\times
\end{equation}
for the following definition:
\begin{definition}
    For any $k\in K_m^\circ$ and $a\in\cal{O}_\Sigma^\times$ we set 
    \[
    \langle k\rangle_m:=[K_m\cdot k\cdot K_m],\qquad \langle a\rangle_m:=[K_m\cdot \gamma_a\cdot K_m].
    \]
\end{definition}
Clearly, for any lift $\widetilde{\zeta(k)}\in\cal{O}_\Sigma^\times$ of $\zeta(k)$, we have $\langle k\rangle_m=\langle \widetilde{\zeta(k)}\rangle_m$. Moreover, if $\Lambda_{\Sigma,m}$ is the $\Z_p$-algebra generated by the diamond operators $\{\langle k\rangle_m\ \lvert \ k\in D_{\Sigma,m}\}$, the map 
\begin{equation}\label{Iwasawa algebra}
\Z_p[(\cal{O}_{\Sigma}/p_{\Sigma}^m)^\times]\overset{\sim}{\longrightarrow} \Lambda_{\Sigma, m},\qquad [a\bmod{p_\cal{S}^m}]\mapsto\langle a\rangle_m
\end{equation}
is an isomorphism of $\Z_p$-algebras. 
\begin{remark}\label{rmk: dependence on uniformizer}
   Let $\p\in\Sigma$, $\varpi\in\cal{O}_\p$ a uniformizer, and $a\in\cal{O}_\p^\times$ then  
\[\
{_{m}}U_{a\varpi}\hspace{1mm}=\hspace{1mm}{_{m}}U_{\varpi}\cdot \langle a \rangle_m\hspace{1mm}=\hspace{1mm}\langle a\rangle_m\cdot  {_{m}}U_{\varpi}.
\] 
\end{remark}

\begin{definition}
    Let $m\ge0$. The abstract Hecke algebra of level $K_m$ is the $\Lambda_{\Sigma,m}$-algebra $\frak{h}_{\Sigma,m}$ generated by the Hecke operators
    \[
  \big\{T_\q\ :\ \q\nmid\frak{f}_A\big\}\cup\big\{U_\q\ :\ \q\mid \n^{\mbox{\tiny $+$}}\cdot p_{\cal{S}\setminus\Sigma}\big\}\cup\big\{{_{m}}U_{\varpi}\ :\ \p\in\Sigma,\ \varpi\in\cal{O}_\p\ \text{uniformizer} \big\}.
    \]
\end{definition}

\begin{remark}\label{abstract compatibility}
    There exist ring homomorphisms $\Lambda_{\Sigma,m+1}\to\Lambda_{\Sigma,m}$, $\langle a\rangle_{m+1}\mapsto\langle a\rangle_m$, and
\[
\frak{h}_{\Sigma,m+1}\to\frak{h}_{\Sigma,m},\qquad T\mapsto\begin{cases}
    T_\q \text{ (resp. }U_\q\text{) }&\text{if }\  T=T_\q \text{  (resp. }U_\q\text{)},\\
    {_{m}}U_\varpi&\text{if }\ T={_{m+1}}U_\varpi.\\
\end{cases}
\]
\end{remark}

\subsubsection{Directed systems.}\label{Directed systems}
For any $m\ge1$ the morphism $X_m\to X^\circ_m$ induced by the inclusion $K_m\subset K_m^\circ$ is Galois with  Galois group $D_{\Sigma,m}:=K^\circ_m/K_m$. 
The inclusions $K_{m+1}\subset K_m\subset K_0$ determine morphisms $\pi_1\colon X_{m+1}\to X_m$ which are part of the projective system
\[\xymatrix{
X_0& X_1\ar[l]_-{\pi_1}\ar[l]&\cdots\ar[l]_-{\pi_1}& X_m\ar[l]&\ar[l]\cdots
}\] 
We denote by $J_m$ (resp. $J_m^\circ$) the Jacobian of $X_m$ (resp. $X_m^\circ)$ and let Hecke correspondences act on Jacobians by covariant functoriality, for instance (\cite{NekEuler}, (1.16.3)).
\begin{definition}
Let $m\ge0$. The (physical) Hecke algebra of level $K_m$ is defined as
\[
\bb{T}_{\Sigma,m}:=\mrm{Im}\Big(\frak{h}_{\Sigma,m}\rightarrow \mrm{End}(J_m)\otimes\Z_p\Big).
\]
It is a $\Z_p$-module of finite rank.
\end{definition}
\begin{proposition}\label{prop: compatibility Hecke}
Let $m\ge1$. The map $\pi_1\colon X_{m+1}\to X_m$ produces a morphism $\pi^*_1\colon J_m\to J_{m+1}$ between the corresponding Jacobians, and  ring homomorphisms $\Lambda_{\Sigma,m+1}\to \Lambda_{\Sigma,m}$, $\bb{T}_{\Sigma,m+1}\to\bb{T}_{\Sigma,m}$ compatible with the maps in Remark \ref{abstract compatibility}.
\end{proposition}
\begin{proof}
The argument to verify the claim is standard, we sketch it for the convenience of the reader because it is used several times in the article. As  $K_m$ and $K_{m+1}$ differ only at primes $\p\in\Sigma$, from the coset decompositions \eqref{cosets-decompositions} it is clear that the diagram
\[\xymatrix{
J_{m+1}\ar[rr]^{T} & &J_{m+1}\\
J_{m}\ar[u]^{\pi^*_1}\ar[rr]^{T}&& J_{m}\ar[u]^{\pi^*_1}
}
\]
commutes if $T\in \big\{T_\q\hspace{1mm} :\hspace{1mm} \q\nmid\frak{f}_A\big\}\cup \big\{U_\q\hspace{1mm} :\hspace{1mm} \q\mid\n^{\mbox{\tiny $+$}}\big\}$. For any $a\in\cal{O}_\Sigma^\times$, the diagonal matrix $\gamma_a$ normalizes every $K_m$, thus a direct computation gives
\[
 \langle a\rangle_{m+1}\circ\pi^*_1=\pi^*_1\circ\langle a\rangle_m\qquad\forall\ a\in \cal{O}_\Sigma^\times.
\]
We claim that the equality
\[
({_{m+1}}U_{\varpi})\circ\pi^*_1=\pi^*_1\circ({_{m}}U_{\varpi}),\qquad
\]
holds for every $\p\in\Sigma$ and every uniformizer $\varpi\in\cal{O}_\p$. 
Fix $\p\in\Sigma$ and a uniformizer $\varpi\in\cal{O}_\p$. It is advantageous to consider the intermediate congruence subgroup 
\[
K_{\p^{m+1}}\subset (K_{\p^m}\cap K_{\p^{m+1}}^\circ)\subset K_{\p^m}.
\]
From the isomorphism
\[
(K_{\p^m}\cap K_{\p^{m+1}}^\circ)/K_{\p^{m+1}}\overset{\sim}{\longrightarrow}(1+\p^m\cal{O}_\p)/(1+\p^{m+1}\cal{O}_\p),\qquad\begin{pmatrix}
    a&b\\c&d
\end{pmatrix}\mapsto a^{-1}d,
\]
one deduces that $\left\{\begin{pmatrix}
    1&0\\0&\delta
\end{pmatrix}\right\}_{\delta}$
is a set of representatives for $(K_{\p^m}\cap K^\circ_{\p^{m+1}})/K_{\p^{m+1}}$ if $\{\delta\}_\delta$ is a set of representatives for $(1+\p^m\cal{O}_\p)/(1+\p^{m+1}\cal{O}_\p)$. A direct computation then shows that ${_{m+1}}U_\varpi$ commutes with the pullback from level $(K_{\p^m}\cap K_{\p^{m+1}}^\circ)$ to level $K_{\p^{m+1}}$. 

\noindent After taking  $\left\{\begin{pmatrix}
    1&0\\\varpi^ma&1
\end{pmatrix}\right\}_{\bar{a}\in\kappa_\p}$ as set of representatives
for $K_{\p^{m}}/(K_{\p^m}\cap K^\circ_{\p^{m+1}})$, the equality
\[
\begin{pmatrix}
    1&0\\\varpi^{m}a&\varpi
\end{pmatrix}\begin{pmatrix}
    1&0\\\varpi^ma'&1
\end{pmatrix}=\begin{pmatrix}
    1&0\\\varpi^ma&1
\end{pmatrix}\begin{pmatrix}
    1&0\\\varpi^{m+1}a'&\varpi
\end{pmatrix}
\]
implies that the pullback from level $K_{\p^{m}}$ to level $(K_{\p^m}\cap K_{\p^{m+1}}^\circ)$ intertwines ${_{m}U_\varpi}$ with ${_{m+1}U_\varpi}$.
\end{proof}

 Proposition \ref{prop: compatibility Hecke} implies that  the projective limits 
\[
\Lambda_\Sigma:=\varprojlim_m\Lambda_{\Sigma,m},\qquad \bb{T}_\Sigma:=\varprojlim_m\bb{T}_{\Sigma,m},
\]
admit a ring homomorphism $\Lambda_\Sigma\rightarrow \bb{T}_\Sigma$.
We denote by
 \begin{equation}
\big\{\mbf{T}_\q\ :\ \q\nmid\frak{f}_A\big\}\cup\{\mbf{U}_\q\ :\ \q\mid\n^{\mbox{\tiny $+$}}\cdot p_{\cal{S}\setminus\Sigma}\big\}\cup\big\{\mbf{U}_{\varpi}\ :\ \p\in\Sigma,\ \varpi\in\cal{O}_\p\ \text{uniformizer} \big\}
\end{equation}
the natural generators of $\bb{T}_\Sigma$ as a $\Lambda_\Sigma$-algebra. Letting $\llangle a\rrangle:=(\langle a\rangle_m)_{m\ge1}\in \Lambda_\Sigma$, equation \eqref{Iwasawa algebra} implies that the character $\cal{O}_\Sigma^\times\ni a\mapsto \llangle a\rrangle\in\Lambda^\times_\Sigma$ induces an isomorphism $\Z_p\llbracket\cal{O}_\Sigma^\times\rrbracket\overset{\sim}{\longrightarrow}\Lambda_\Sigma$. 

\smallskip
\noindent For $\Box\in\{\text{loc}, \text{glo}\}$ we set 
 \begin{equation}\label{loc-glo Jacobian points}
 J_m^\Box:=\begin{cases}
    J_m(E_w)&\Box=\text{loc}\\
     J_m({^{w}}E_{\mrm{ab}})&\Box=\text{glo}.\\
\end{cases}
\end{equation}
\begin{corollary}\label{Hecke-action on J}
   The direct limit
    \[
J_\Sigma^\Box:=\varinjlim_m \hspace{1mm}\Big(J_m^\Box\otimes\bb{Z}_p, \pi^*_1\Big)
\] 
    inherits the structure of a right  $\bb{T}_\Sigma$-module.
   
\end{corollary}

\subsubsection{Nearly ordinary idempotent.} The nearly ordinary idempotent $e_\Sigma\in \bb{T}_\Sigma$ is defined in the usual fashion: For every $m\ge1$ consider an element ${_{m}}U_\Sigma \in \bb{T}_{\Sigma,m}$ obtained as the product $\prod_{\p\in\Sigma}{_{m}}U_{\varpi_\p}$ for some choice of uniformizers $\varpi_\p\in\cal{O}_\p$. Since $\bb{T}_{\Sigma,m}$ is finite over $\Z_p$, the limit 
\[
e_{\Sigma, m}:=\underset{n\to+\infty}{\lim}({_{m}}U_\Sigma)^{n!}\ \in \ \bb{T}_{\Sigma,m}
\]
exists, and by Remark \ref{rmk: dependence on uniformizer} is independent of the choices. The idempotents $\{e_m\}_m$ are compatible and define the nearly ordinary projector 
\begin{equation}
e_\Sigma:=\varprojlim_m e_{\Sigma, m}\ \in\ \bb{T}_\Sigma.
\end{equation}
For any $\bb{T}_\Sigma$-module $M$, we will write
\[
M^\mrm{n.o.}:=e_\Sigma M.
\]

\subsubsection{Further observations.}\label{Further observations}
Let  $\frak{l}\nmid \frak{f}_A$ be an auxiliary prime such that $a_\frak{l}(E)\not\equiv\mrm{N}_{F/\bb{Q}}(\frak{l})+1\pmod{p}$ and define the family of morphisms  
\[
\varphi_m\colon X_m\longrightarrow J_m\otimes\bb{Z}_p,\qquad x\mapsto (T_\frak{l}-\mrm{N}_{F/\bb{Q}}(\frak{l})-1)[x]\otimes(a_\frak{l}(E)-\mrm{N}_{F/\bb{Q}}(\frak{l})-1)^{-1}.
\]
They are defined over $F$ and satisfy the compatibility relation 
\begin{equation}\label{compatibility}
\varphi_{m+1}\circ\pi^*_1=\pi^*_1\circ\varphi_m.
\end{equation}
We write $j_m\colon J_m^\Box\otimes\bb{Z}_p\to J_\Sigma^\Box$ for the natural map in the direct limit and denote the $p$-adic completion of $J_\Sigma^\Box$ by
\begin{equation}
(J_\Sigma^\Box)^\wedge :=\varprojlim_n \big(J_\Sigma^\Box\otimes\Z/p^n\Z\big).
\end{equation}
\begin{remark}
  If $S\not\subseteq\Sigma_p$, then  $(J_\Sigma^\text{loc})^\wedge=0$ because the auxiliary place $w$ is not $p$-adic.
\end{remark}
Since tensor products commute with direct limits, and both direct and inverse limits are left exact functors, we can use Kummer maps to obtain an injective homomorphism 
\[
(J_\Sigma^\Box)^\wedge\longhookrightarrow\varprojlim_n\varinjlim_m\ \mrm{H}^1\big(E^\Box, J_m[p^n]\big)
\]
where $E^\text{loc}=E_w$ and $E^\text{glo}=E^\mrm{ab}$.
The goal of the next section is to understand the submodule of elements fixed by the action of diamond operators
\[
D_\Sigma:=\varprojlim_m D_{\Sigma,m}\cong \cal{O}_{\Sigma}^\times.
\]

\begin{remark}\label{relation hom-cohm}
   For any smooth projective curve $C_{/F}$, with Jacobian $J_C$, there is a functorial isomorphisms of $G_F$-modules (\cite{etale-Milne}, \S III, Corollary 4.18)
   \[
   \mrm{H}^1_{\et}\big(C_{\overline{\Q}},\mu_{p^n}\big)\cong J_C[p^n].
   \]
    In particular, the isomorphism is Hecke equivariant for $C=X_m^\circ$ or $X_m$. Further, the functorial isomorphism of the universal coefficient theorem 
    \[
    \mrm{H}^1(C,\Z)\overset{\sim}{\longrightarrow}\mrm{Hom}_{\Z}\big( \mrm{H}_1(C,\Z),\Z)
    \]
    together with the freeness of the (co)homology of curves  will be used in Section \ref{Descent of Galois cohomology classes} to study the Hecke module structure of $J_m[p^n]^\mrm{n.o.}$,  and in Section \ref{hida theory} to establish properties of the Hecke algebra $\bb{T}^\mrm{n.o.}_\cal{S}$.
\end{remark}

\subsection{Descent of Galois cohomology classes}\label{Descent of Galois cohomology classes}
We continue to consider a fixed subset $\Sigma\subseteq\cal{S}$.
For $m\ge1$, the Galois group $D_{\Sigma,m}:=K^\circ_m/K_m$ of the covering $X_m\to X^\circ_m$  canonically decomposes as the product 
\[
D_{\Sigma,m}= \mu_\Sigma\times \Gamma_{\Sigma,m}
\]
of its prime-to-$p$ part $\mu_\Sigma$ and its $p$-primary part $\Gamma_{\Sigma,m}$.

\begin{lemma}\label{pullbackCOMP}
    For any $m,n\ge1$ the map
    \[
J^\circ_m[p^n]^\mrm{n.o.}\overset{\sim}{\longrightarrow}\mrm{H}^0\big(D_{\Sigma,m}, J_m[p^n]^\mrm{n.o.}\big),
    \]
    obtained from $ X_m\to X_m^\circ$, is an isomorphism of Hecke and $G_F$-modules.
\end{lemma}
\begin{proof}
The claim is deduced from Remark \ref{relation hom-cohm}, the Hochschild--Serre spectral sequence for \'etale cohomology (\cite{etale-Milne}, \S III, Theorem 2.20),  and the concentration of nearly ordinary cohomology of curves in middle degree.
\end{proof}

The proof of the next lemma was inspired by \cite[\S 10]{Mokrane-Tilouine} and \cite[Proposition 2.4]{Dimitrov-Ihara}. 

\begin{proposition}\label{precursor} 
For every $m\ge1$ the $\Z_p[\Gamma_{\Sigma,m}]$-module $\mrm{H}_1(X_m,\Z_p)^\mrm{n.o.}$ is finite and free. 
\end{proposition}
\begin{proof}
 Since $\mrm{H}_1(X_m,\Z_p)^\mrm{n.o.}$ is finite and free over $\Z_p$, by Nakayama's lemma it suffices to show that $\mrm{H}_1(X_m,\bb{F}_p)^\mrm{n.o.}$ is finite and free over $\bb{F}_p[\Gamma_{\Sigma,m}]$.
 We claim that 
 \begin{equation}\label{tor-to-vanish}
\mrm{Tor}_1^{\bb{F}_p[\Gamma_{\Sigma,m}]}\big(\bb{F}_p,\ \mrm{H}_1(X_m,\bb{F}_p)^\mrm{n.o.}\big)=0,
\end{equation}
then the local criterion for flatness implies that the finite $\bb{F}_p[\Gamma_{\Sigma,m}]$-module $ \mrm{H}_1(X_m,\bb{F}_p)^\mrm{n.o.}$ is flat and hence free. As the maximal ideal and the augmentation ideal of $\bb{F}_p[\Gamma_{\Sigma,m}]$ coincide, we can interpret the LHS of \eqref{tor-to-vanish} in terms of group homology, i.e., 
\[
\mrm{Tor}_1^{\bb{F}_p[\Gamma_{\Sigma,m}]}\big(\bb{F}_p,\ \mrm{H}_1(X_m,\bb{F}_p)^\mrm{n.o.}\big)=\mrm{H}_1(\Gamma_{\Sigma,m}, \mrm{H}_1(X_m,\bb{F}_p)^\mrm{n.o.}).
\]
Set $X_m':=X_m/\Gamma_{\Sigma,m}$ and consider the Cartan--Leray spectral sequence for the Galois covering $X_m\to X_m'$ (\cite{Brown}, \S VII, Theorem 7.9 \& Exercise 4)
 \[
 E^2_{j,k}=\mrm{H}_j(\Gamma_{\Sigma,m},\mrm{H}_k(X_m,\bb{F}_p)^\mrm{n.o.})\implies \mrm{H}_{j+k}(X_m',\bb{F}_p)^\mrm{n.o.}.
 \]
 As the nearly ordinary part of the homology of curves is concentrated in middle degree, we obtain 
 \[
\mrm{H}_j(\Gamma_{\Sigma,m},\mrm{H}_k(X_m,\bb{F}_p)^\mrm{n.o.})=\begin{cases}
    \mrm{H}_{1}(X_m',\bb{F}_p)^\mrm{n.o.}&j=0, k=1\\
   0&\text{otherwise}.
\end{cases}
\]
In particular, the claim \eqref{tor-to-vanish} holds. 
\end{proof}

\begin{corollary}\label{coinduced}
   Let $\chi\colon\mu_\Sigma\to\overline{\bb{Q}}^\times$ be a character. For every $m\ge1$ the $\chi$-isotypic component $\mrm{H}_1(X_m,\Z_p)^\mrm{n.o.}_\chi$ is finite and free over $\bb{Z}_p(\chi)[\Gamma_{\Sigma,m}]$. 
\end{corollary}
\begin{proof}
    The claim follows from Proposition \ref{precursor} as projective modules over a local ring are free.
\end{proof}

\subsubsection{Vanishing of Galois invariants.}
Let  $\overline{\varrho}\colon G_F\to\mrm{GL}_2(\bb{F}_p)$ be the mod-$p$ Galois representation attached to $A_{/F}$.
\begin{lemma}\label{first vanishing of invariants}
     Suppose Assumption \ref{global classes assumptions} holds. Then, for any solvable extension $L/F$ we have 
     \[
     \mrm{H}^0(L,\overline{\varrho})=0.
     \]
\end{lemma}
\begin{proof}
  Without loss of generality we can assume that $L/F$ is Galois. Moreover, it suffices to prove that $\mrm{SL}_2(\bb{F}_p)\subseteq\overline{\varrho}(G_L)$. 
 To show the required inclusion, note that since $p\ge5$ the group $\mrm{SL}_2(\bb{F}_p)$ is quasisimple, i.e. it is perfect, $\mrm{SL}_2(\bb{F}_p)=[\mrm{SL}_2(\bb{F}_p),\mrm{SL}_2(\bb{F}_p)]$, and its inner automorphism group, $\mrm{PSL}_2(\bb{F}_p)$, is simple. It follows that the proper normal subgroups of $\mrm{SL}_2(\bb{F}_p)$ are contained in its center $\{\pm1\}$. We deduce that the morphism from $\mrm{SL}_2(\bb{F}_p)$ to the solvable group $\overline{\varrho}(G_F)/\overline{\varrho}(G_L)$ induced by the inclusion $\mrm{SL}_2(\bb{F}_p)\subseteq\overline{\varrho}(G_F)$ has to be trivial because $\mrm{SL}_2(\bb{F}_p)$ and $\mrm{PSL}_2(\bb{F}_p)$ are not solvable for $p\ge5$.
\end{proof}

\begin{lemma}\label{local first vanishing of invariants}
    Let $S$ be definite with auxiliary $p$-adic place $w$. Suppose that Assumption \ref{local classes assumptions} holds, then  $\mrm{H}^0(E_w,\overline{\varrho})=0$.
\end{lemma}
\begin{proof}
     The local representation $\overline{\varrho}_{\lvert D_w}$ is an extension $\xymatrix{
     0\ar[r]&\chi_{\mrm{cyc}}\ar[r]&\overline{\varrho}_{\lvert D_w}\ar[r]&\mathbbm{1}\ar[r]&0
     }$
     of the trivial character by the mod-$p$ cyclotomic character.  Note that $\chi_\mrm{cyc}\lvert_{G_{E_w}}\not=\mathbbm{1}$ because $p\ge3$, $\mu_p\not\subset F_w$, and $E_w/F_w$ is quadratic. Therefore, 
     \[\begin{split}
     \mrm{H}^0(E_w,\overline{\varrho})\not=0&\qquad\iff\qquad \overline{\varrho}_{\lvert G_{E_w}}\cong \mathbbm{1}\oplus\chi_\mrm{cyc}\\
     &\qquad\iff\qquad q_w\in (E_w^\times)^p.
     \end{split}\]
     The hypothesis $q_w\not\in (F_w^\times)^p$ implies the irreducibility of $X^p-q_w\in F_w[X]$. Since $[E_w:F_w]=2<p$ we can then deduce $q_w\not\in (E_w^\times)^p$.
\end{proof}

If the mod-$p$ representation $\overline{\varrho}$ appears in the $p$-torsion of $J_m$, for a $\bb{T}_{\Sigma,m}$-module $M$ we set 
\[
M_{\overline{\varrho}}:=M \otimes_{\bb{T}_{\Sigma,m}}(\bb{T}_{\Sigma,m})_{\frak{m}}
\]
 where $\frak{m}$ denotes the maximal ideal of $\bb{T}_{\Sigma,m}$ associated to $\overline{\varrho}$. Otherwise, we simply set $M_{\overline{\varrho}}=\{0\}$.

\begin{proposition}\label{NOinvariants}
   Suppose Assumption \ref{global classes assumptions} holds. Then, for any solvable extension $L/F$  and any pair of integers $m,n\ge 1$, we have 
   \[
  \mrm{H}^0\big(L, \hspace{0.7mm}J_m[p^n]_{\overline{\varrho}}\big)=0.
   \]
    Moreover, if $S$ is definite and Assumption \ref{local classes assumptions} also holds, then $\mrm{H}^0\big(E_w, \hspace{0.7mm}J_m[p^n]_{\overline{\varrho}}\big)=0$.
\end{proposition}
\begin{proof}
Fix integers $m,n\ge1$. Since $\overline{\varrho}$ is absolutely irreducible, the main theorem of \cite{BLR} together with the Eichler--Shimura relations, the Brauer--Nesbitt theorem and Chebotarev density theorem, imply that the Galois representation obtained by reducing $J_m[p^n]_{\overline{\varrho}}$ modulo the maximal ideal of $(\bb{T}_{\Sigma,m})_{\overline{\varrho}}$ is isomorphic to $\overline{\varrho}^{\oplus k}$ for some $k\ge0$. As in the proof of \cite[Proposition 3.5]{Longo-BLR}), a repeated use of Nakayama's lemma then provides lifts $\{\varrho_j\colon G_F\to\mrm{GL}_2(\bb{Z}/p^{n_j}\bb{Z})\}_{j=1}^k$ of  $\overline{\varrho}$ and a $G_F$-equivariant isomorphism
\[
J_m[p^n]_{\overline{\varrho}}\cong \bigoplus_{j=1}^kM_j
\]
where $M_j=(\bb{Z}/p^{n_j}\bb{Z})^2$ with $G_F$-action induced by $\varrho_j$. We now show that $\mrm{H}^0(L,M_j)=0$ for every $j=1,\dots k$. Indeed, for any $0\le a<n_j$ multiplication by $p^a$ induces a $G_F$-equivariant isomorphism 
\[
\overline{\varrho}\cong M_j/p M_j\overset{\sim}{\longrightarrow}p^aM_j/p^{a+1}M_j.
\]
Thus, $\mrm{H}^0(L,p^aM_j/p^{a+1}M_j)=0$ by Lemma \ref {first vanishing of invariants}. Since $\overline{\varrho}\cong p^{n_j-1}M_j$, we deduce that
\[
\mrm{H}^0(L,M_j)\cong \mrm{H}^0(L,pM_j)\cong\dots \cong\mrm{H}^0(L,p^{n_j-1}M_j)=0.
\]
The second claim is proved similarly by replacing Lemma \ref {first vanishing of invariants} with Lemma \ref{local first vanishing of invariants}.
\end{proof}

\begin{remark}
  Proposition \ref{NOinvariants} implies that if $L_1\subseteq L_2\subseteq\dots\subseteq L_m\subseteq\dots$ is a tower of solvable extensions of $F$, then 
   \[
   \varprojlim_n\varinjlim_m\ \mrm{H}^1(L_m,\hspace{0.7mm}J_m[p^n]_{\overline{\varrho}})\overset{\sim}{\longrightarrow}  \varprojlim_n\varinjlim_m\ \mrm{H}^1({{^w}}E_\mrm{ab},\hspace{0.7mm}J_m[p^n]_{\overline{\varrho}})^{\mrm{Gal}({{^w}}E_\mrm{ab}/L_m)},
   \]
   which injects in 
   \[
   \varprojlim_n\varinjlim_m\ \mrm{H}^1({{^w}}E_\mrm{ab},\hspace{0.7mm}J_m[p^n]_{\overline{\varrho}})
   \]
because projective limits are left exact and direct limits are exact.
\end{remark}

\subsubsection{Descent.}\label{subsub descent}
We combine Lemma \ref{pullbackCOMP}, Corollary \ref{coinduced} and Proposition \ref{NOinvariants} to obtain:
\begin{theorem}\label{descent}
Suppose Assumption \ref{global classes assumptions} holds. Then, for any solvable extension $L/F$  and any pair of integers $m,n\ge 1$, the map  
\[
    \mrm{H}^1\big(L,\hspace{0.7mm} J^\circ_m[p^n]_{\overline{\varrho}}\big)\overset{\sim}{\longrightarrow}\mrm{H}^1\big(L, \hspace{0.7mm}J_m[p^n]_{\overline{\varrho}}\big)^{D_{\Sigma,m}}
   \]
induced by the Galois covering $ X_m\to X^\circ_m$ is an isomorphism.
    Moreover, if $S$ is definite and Assumption \ref{local classes assumptions} also holds, then the local map
    \[
    \mrm{H}^1\big(E_w,\hspace{0.7mm} J^\circ_m[p^n]_{\overline{\varrho}}\big)\overset{\sim}{\longrightarrow}\mrm{H}^1\big(E_w, \hspace{0.7mm}J_m[p^n]_{\overline{\varrho}}\big)^{D_{\Sigma,m}}
   \]
   is an isomorphism.
\end{theorem}
\begin{proof}
The idea is to compute the cohomology group $\mrm{H}^{1}\big(G_L\times D_{\Sigma,m},\hspace{0.7mm} J_m[p^n]_{\overline{\varrho}}\big)$ in two different ways. Consider the short exact sequence of groups 
\[
\xymatrix{1\ar[r]&D_{\Sigma, m}\ar[r]&G_L\times D_{\Sigma,m}\ar[r]& G_L\ar[r]&1}.
\]
 The inflation-restriction exact sequence  allows us to deduce that the inflation map 
 \[
 \mrm{H}^{1}\big(G_L, \hspace{0.7mm}J^\circ_m[p^n]_{\overline{\varrho}}\big)\overset{\sim}{\longrightarrow}\mrm{H}^{1}\big(G_L\times D_{\Sigma,m}, \hspace{0.7mm}J_m[p^n]_{\overline{\varrho}}\big)
 \]
 is an isomorphism, using Lemma \ref{pullbackCOMP} to identify the domain and Corollary \ref{coinduced} to obtain the surjectivity. By considering instead the short exact sequence 
\[
\xymatrix{1\ar[r]&G_L\ar[r]&G_L\times D_{\Sigma, m}\ar[r]& D_{\Sigma,m}\ar[r]&1},
\]
 the inflation-restriction exact sequence  implies that the restriction map 
  \[
\mrm{H}^{1}\big(G_L\times D_{\Sigma,m}, \hspace{0.7mm}J_m[p^n]_{\overline{\varrho}}\big)\overset{\sim}{\longrightarrow}\mrm{H}^1(G_L, \hspace{0.7mm}J_m[p^n]_{\overline{\varrho}})^{D_{\Sigma,m}}
  \]
  is an isomorphism once Proposition \ref{NOinvariants} is invoked to ensure both injectivity and surjectivity. The proof of the local version is identical.
\end{proof}

Let
$(J_\Sigma^\Box)_{\overline{\varrho}}:=\varinjlim_m \hspace{1mm}\big(J_m^\Box\otimes\Z_p\big)_{\overline{\varrho}}$ and write $(J_\Sigma^\Box)^\wedge_{\overline{\varrho}}$ for its $p$-adic completion. For additional clarity, in the next corollary we make the dependence on $\Sigma$ more evident by writing $J_m(\Sigma), J_m^\circ(\Sigma)$ for $J_m, J_m^\circ$, the Jacobian of the Shimura curves of levels  $K^\circ(\n^{\mbox{\tiny $+$}})\cap K\big(p_\Sigma^m\big)$, $K^\circ(\n^{\mbox{\tiny $+$}}\cdot p_\Sigma^m)$.
\begin{corollary}\label{cor: descent}
Suppose Assumption \ref{global classes assumptions} holds. Then, for any abelian extension $L/E$, Galois over $F$, there is an injective homomorphism
 \[
\mrm{H}^0\big(G_L\times D_\Sigma,\hspace{1mm} (J_\Sigma^\text{glo})^\wedge_{\overline{\varrho}}\big)\hooklongrightarrow \mrm{H}^1\big(L,\hspace{1mm}  T_p(J_1^\circ(\Sigma))_{\overline{\varrho}}\big)
 \]
   where $T_p(J_1^\circ(\Sigma))$ denotes the $p$-adic Tate module of $J_1^\circ(\Sigma)$.  Moreover, if $S$ is definite and Assumption \ref{local classes assumptions} also holds, there is an injective map
    \[
\mrm{H}^0\big(D_\Sigma,\hspace{1mm} (J_\Sigma^\text{loc})^\wedge_{\overline{\varrho}}\big)\hooklongrightarrow \mrm{H}^1\big(E_w,\hspace{1mm}  T_p(J_1^\circ(\Sigma))_{\overline{\varrho}}\big).
 \] 
\end{corollary}
\begin{proof}
We prove only the first claim since the proof of the second is identical to part of the argument given below. Recall that the Kummer map induces the injection
\[
(J_\Sigma^\text{glo})^\wedge_{\overline{\varrho}}\longhookrightarrow\varprojlim_n\varinjlim_m\ \mrm{H}^1\big(E^\mrm{ab},\hspace{0.7mm} J_m(\Sigma)[p^n]_{\overline{\varrho}}\big).
\]
Proposition \ref{NOinvariants} and Theorem \ref{descent} give the isomorphisms
\[\xymatrix{
\mrm{H}^1\big(L,\hspace{0.7mm}J_m^\circ(\Sigma)[p^n]_{\overline{\varrho}}\big)\ar[r]^-\sim& \mrm{H}^1\big(L,\hspace{0.7mm}J_m(\Sigma)[p^n]_{\overline{\varrho}}\big)^{D_{\Sigma,m}}\ar[r]^-\sim&\mrm{H}^1\big(E^\mrm{ab},\hspace{0.7mm}J_m(\Sigma)[p^n]_{\overline{\varrho}}\big)^{G_L\times D_{\Sigma,m}}
}\]
which allows us to deduce the injection
\[
\mrm{H}^0\big(G_L\times D_\Sigma,\hspace{1mm} (J_\Sigma^\text{glo})^\wedge_{\overline{\varrho}}\big)\longhookrightarrow\varprojlim_n\varinjlim_m\ \mrm{H}^1\big(L,\hspace{0.7mm}J_m^\circ(\Sigma)[p^n]_{\overline{\varrho}}\big).
\]
Since the Hecke eigensystem associated to the elliptic curve $A_{/F}$  is ordinary at every prime in $\cal{S}$ and $U_{\Sigma}=(\pi_2)_*\circ(\pi_1)^*$, we can consider the following commuting diagram
\[\xymatrix{
J_{m+1}^\circ(\Sigma)[p^n]_{\overline{\varrho}}\ar[rrr]^-{U_{\Sigma}^{-m}\circ(\pi_2)_*^{m}}&&& J_1^\circ(\Sigma)[p^n]_{\overline{\varrho}}\\
J_m^\circ(\Sigma)[p^n]_{\overline{\varrho}}\ar[u]^{(\pi_1)^*}\ar[rrr]^-{U_{\Sigma}^{1-m}\circ(\pi_2)_*^{m-1}}&&& J_1^\circ(\Sigma)[p^n]_{\overline{\varrho}}\ar@{=}[u]
}\]
where the horizontal arrows are Galois equivariant isomorphisms, to deduce 
\[
\varinjlim_m\ \mrm{H}^1\big(L,\hspace{0.7mm}J_m^\circ(\Sigma)[p^n]_{\overline{\varrho}}\big)\overset{\sim}{\longrightarrow}\mrm{H}^1\big(L,\hspace{0.7mm}J_1^\circ(\Sigma)[p^n]_{\overline{\varrho}}\big).
\]
\end{proof}

\subsection{Hida theory}\label{hida theory}
This section focuses on a minor variant of Hida theory. As we are not aware of a suitable reference, for the reader's convenience we prove the results needed for our applications.

Fix $\Sigma\subseteq \cal{S}$ and recall the decomposition $D_\Sigma=\mu_\Sigma\times\Gamma_\Sigma$. We denote by $\Lambda^{\bfcdot}$, $\bb{T}^{\bfcdot}$, and $M_m^{\bfcdot}$ the components of $\Lambda_\Sigma$, $\bb{T}_\Sigma^\mrm{n.o.}$, and $\mrm{H}_1(X_m,\Z_p)^\mrm{n.o.}$ on which $\mu_\Sigma$ acts trivially. 
By Remark \ref{relation hom-cohm} we can study $\bb{T}^{\bfcdot}$ through its action on  completed homology
\[
M^{\bfcdot}:=\varprojlim_m \big(M_m^{\bfcdot}, (\pi_{1})_*\big).
\]
In particular, $\bb{T}^{\bfcdot}\subset\mrm{End}_{\Lambda^{\bfcdot}}(M)$.

\begin{lemma}\label{classic Hida}
For every $m\ge1$ the covering $\pi_1\colon X_{m+1}\to X_m$ induces the isomorphism
\[
M_{m+1}^{\bfcdot}\otimes_{\Lambda^{\bfcdot}_{m+1}}\Lambda^{\bfcdot}_m\overset{\sim}{\longrightarrow} M_m^{\bfcdot}.
\]
\end{lemma}
\begin{proof}
  Since $U_\cal{S}$ is an automorphism of $M_m^{\bfcdot}$ for every $m\ge1$. The claim  holds because the $\Lambda_m^{\bfcdot}$-module map $[K_m\eta_{\varpi_\cal{S}}K_{m+1}]\colon M_m^{\bfcdot}\to M_{m+1}^{\bfcdot}\otimes_{\Lambda^{\bfcdot}_{m+1}}\Lambda_m^{\bfcdot}$ makes the following diagram commute
    \[\xymatrix{
    M_{m+1}^{\bfcdot}\otimes_{\Lambda_{m+1}^{\bfcdot}}\Lambda_m^{\bfcdot}\ar[rr]^-{(\pi_{1})_*}\ar[d]_-{U_\cal{S}}&&M_m^{\bfcdot}\ar[d]^-{U_\cal{S}}\ar@{.>}[lld]\\
M_{m+1}^{\bfcdot}\otimes_{\Lambda_{m+1}^{\bfcdot}}\Lambda_m^{\bfcdot}\ar[rr]^-{(\pi_{1})_*}&& M_m^{\bfcdot}.
    }\]
\end{proof}

\begin{corollary}\label{cor: T finite torsion-free}
   The $\Lambda_{m}^{\bfcdot}$-rank of $M_m^{\bfcdot}$ is independent of $m\ge1$, and $M$ is finite free over $\Lambda^{\bfcdot}$. Thus, the $\Lambda^{\bfcdot}$-module $\bb{T}^{\bfcdot}$ is finite and torsion-free.
\end{corollary}
\begin{proof}
 Combine  Corollary \ref{coinduced} and Lemma \ref{classic Hida}.   
\end{proof}

 Let $\bb{T}_{1}^{\bfcdot}$ denote the component of $\bb{T}_{1}^\mrm{n.o.}$ on which $\mu_\Sigma$-acts trivially. It is naturally identified with the Hecke algebra acting on $\mrm{H}_1(X_1^{\circ},\Z_p)^\mrm{n.o.}$. We denote by $\cal{I}$  the augmentation ideal, i.e., the kernel of the augmentation map $\Lambda^{\bfcdot}\to\Z_p$.

\begin{proposition}\label{prop: exact control}
  The natural surjective map
    \[
\bb{T}^{\bfcdot}\otimes_{\Lambda^{\bfcdot}}\Q_p\twoheadlongrightarrow \bb{T}_{1}^{\bfcdot}\otimes_{\Lambda_1^{\bfcdot}}\Q_p
    \]
is an isomorphism.
\end{proposition}
\begin{proof}
We prove injectivity by counting dimensions. Let  $(\bb{T}^{\bfcdot})_\cal{I}$ denote the localization at the augmentation ideal, so that $\bb{T}^{\bfcdot}\otimes_{\Lambda^{\bfcdot}}\Q_p=(\bb{T}^{\bfcdot})_\cal{I}/\cal{I}(\bb{T}^{\bfcdot})_\cal{I}$. The Eichler--Shimura isomorphism 
\[
M_1^{\bfcdot}\otimes_{\Lambda_1^{\bfcdot}}\Q_p\cong  \big(\bb{T}_{1}^{\bfcdot}\otimes_{\Lambda_1^{\bfcdot}}\Q_p\big)^{\oplus2},
\]
together with Lemma \ref{classic Hida} and Nakayama's lemma, imply the existence of a surjective map
\[
\alpha\colon(\bb{T}^{\bfcdot})_\cal{I}^{\oplus2}\twoheadlongrightarrow(M^{\bfcdot})_\cal{I}
\]
of  $(\bb{T}^{\bfcdot})_\cal{I}$-modules. Since $(M^{\bfcdot})_\cal{I}$ is free over $(\Lambda^{\bfcdot})_\cal{I}$, Nakayama's lemma implies that $\ker(\alpha)$ is trivial.
In particular, $(\bb{T}^{\bfcdot})_\cal{I}$ is $(\Lambda^{\bfcdot})_\cal{I}$-projective, hence free. Putting all together,
\[\begin{split}
2\cdot \mrm{dim}_{\Q_p}\big(\bb{T}^{\bfcdot}\otimes_{\Lambda^{\bfcdot}}\Q_p\big)&=2\cdot \mrm{rk}_{(\Lambda^{\bfcdot})_\cal{I}}(\bb{T}^{\bfcdot})_\cal{I}\\
&=\mrm{rk}_{(\Lambda^{\bfcdot})_\cal{I}}(M^{\bfcdot})_\cal{I}\\
&=\mrm{dim}_{\Q_p}\big(M_1^{\bfcdot}\otimes_{\Lambda_1^{\bfcdot}}\Q_p\big)\\
&=2\cdot \mrm{dim}_{\Q_p}\big(\bb{T}_{1}^{\bfcdot}\otimes_{\Lambda_1^{\bfcdot}}\Q_p\big).
\end{split}
\]
\end{proof}

\begin{proposition}\label{Etalness at classical points}
 Suppose Assumption \ref{for semisimplicity} holds, then  the  Hecke algebra $\bb{T}^{\bfcdot}$ is finite \'etale over the augmentation ideal $\cal{I}\subset\Lambda^{\bfcdot}$.
\end{proposition}
\begin{proof}
  The algebra $\Lambda^{\bfcdot}$ is a normal domain being isomorphic to a power series ring with coefficients in $\Z_p$. In particular, the localization at any of its prime ideals is geometrically unibranch. Moreover, the injective structure map $\Lambda^{\bfcdot}\hookrightarrow \bb{T}^{\bfcdot}$ is finite and torsion-free by Corollary \ref{cor: T finite torsion-free}.
  
 By Proposition \ref{prop: exact control} and the Jacquet--Langlands correspondence, $\bb{T}^{\bfcdot}\otimes_{\Lambda^{\bfcdot}}\Q_p$ is identified with the Hecke algebra acting on Hilbert cuspforms of weight $2$, level $\frak{f}_A$, trivial character, which are new at the primes in $\Sigma^-\setminus\cal{S}$ and nearly ordinary at the primes in $\cal{S}$. Then, Assumption \ref{for semisimplicity} implies its semi-simplicity as in \cite[Theorem 4.2]{Coleman-Edixhoven} using \cite[Proposition 2]{Chiriac}.
  It follows that $\Lambda^{\bfcdot}\hookrightarrow \bb{T}^{\bfcdot}$ is unramified at every prime above $\cal{I}$. 
  
  Finally, (\cite[\href{https://stacks.math.columbia.edu/tag/0GSC}{Tag 0GSC}]{stacks-project}, Lemma 15.108.2) shows that, for every $\cal{P}\in \mrm{Spec}(\bb{T}^{\bfcdot})$ over $\cal{I}$, the map induced between the localizations $(\Lambda^{\bfcdot})_{\cal{I}}\hookrightarrow (\bb{T}^{\bfcdot})_{\cal{P}}$ is finite \'etale.
\end{proof}

\begin{corollary}\label{isolocalizations}
    For $\cal{P}\in \mrm{Spec}(\bb{T}^{\bfcdot})$ above the augmentation ideal $\cal{I}\subset\Lambda^{\bfcdot}$ with residue field $\Q_p$, the map between localizations
    \[
    (\Lambda^{\bfcdot})_{\cal{I}}\longhookrightarrow(\bb{T}^{\bfcdot})_\cal{P}
    \]
    is an isomorphism.
\end{corollary}
\begin{proof}
   By Proposition \ref{Etalness at classical points} the map $(\Lambda^{\bfcdot})_{\cal{I}}\hookrightarrow(\bb{T}^{\bfcdot})_\cal{P}$ is finite \'etale. Therefore, 
    \[
    (\bb{T}^{\bfcdot})_{\cal{P}}\otimes_{(\Lambda^{\bfcdot})_{\cal{I}}}\Q_p=(\bb{T}^{\bfcdot})_\cal{P}/\cal{P}(\bb{T}^{\bfcdot})_\cal{P}
    \]
    since the fiber is reduced and local. As $(\bb{T}^{\bfcdot})_\cal{P}/\cal{P}(\bb{T}^{\bfcdot})_\cal{P}=\Q_p$, Nakayama's lemma implies the required surjectivity.
\end{proof}

\subsection{Interlude on arithmetic $\cal{L}$-invariants}\label{Interlude on arithmetic L-invariants}
    

Given a ring $A$ and index set $I$, we denote by $A\llbracket\underline{X}_I\rrbracket$ the ring of formal power series in the variables $\{X_i\}_{i\in I}$ and coefficients in $A$.
The construction of the ring of power series is functorial both in the datum of the index set $I$ and of the coefficient ring $A$: given a subset $I\subseteq J$ and  a ring homomorphism $A\rightarrow B$, there is a canonical ring map $A\llbracket\underline{X}_I\rrbracket\rightarrow B\llbracket\underline{X}_J\rrbracket$.
Let $\p\in\cal{S}$. The inclusion $\{\p\}\subseteq \cal{S}$ and the homomorphism $E_\p\hookrightarrow E_{\otimes,\cal{S}}$, $e_\p\mapsto e_\p\otimes(\otimes_{\q\not=\p}1)$ induce the ring map
\begin{equation}\label{ringmap}
E_\p\llbracket X_\p\rrbracket\longhookrightarrow E_{\otimes,\cal{S}}\llbracket\underline{X}_\cal{S}\rrbracket.
\end{equation}
Let $\log_p\colon \cal{O}_\p^\times\to \p\cal{O}_\p$ denote the $p$-adic logarithm. The group homomorphism
\[
\cal{O}_\p^\times\longrightarrow E_\p\llbracket X_\p\rrbracket^\times,\qquad a_\p\mapsto  a_\p ^{X_\p}:=\sum_{n\ge0}\frac{\log_p(a_\p)^n}{n!}X_\p^n,
\]
considered for every $\p\in\cal{S}$, together with the morphism \eqref{ringmap}, gives us the homomorphism
\[
\cal{O}_\cal{S}^\times\longrightarrow E_{\otimes,\cal{S}}\llbracket\underline{X}_\cal{S}\rrbracket^\times,\qquad a_\cal{S}=(a_\p)_\p\mapsto a_\cal{S}^{\underline{X}_\cal{S}}:=\prod_{\p\in\cal{S}}a_\p ^{X_\p}.
\]
This defines the $\Z_p$-algebra homomorphism $\mrm{pw}_\otimes\colon\Lambda_\cal{S}\hookrightarrow E_{\otimes,\cal{S}}\llbracket\underline{X}_\cal{S}\rrbracket$.
Since $\cal{I}=\ker(\Lambda_\cal{S}\to\Z_p)$ is the preimage of the kernel of  $\mrm{ev}_{\underline{0}}\colon E_{\otimes,\cal{S}}\llbracket\underline{X}_\cal{S}\rrbracket\twoheadrightarrow  E_{\otimes,\cal{S}}$, $f(\underline{X}_\cal{S})\mapsto f(\underline{0})$, the morphism extends to the localization at $\cal{I}$:
\[
\mrm{pw}_\otimes\colon(\Lambda_\cal{S})_\cal{I}\longhookrightarrow E_{\otimes,\cal{S}}\llbracket\underline{X}_\cal{S}\rrbracket.
\]
\begin{remark}\label{one explicit deriviative}
   Fix $\p\in\cal{S}$. Formal differentiation in the variable $X_\p$ together with evalutation at $\underline{X}_\cal{S}=\underline{0}$ induces a $E_{\otimes,\cal{S}}$-linear map $\partial_{\p}(-)\big\lvert_{\underline{0}}\colon E_{\otimes,\cal{S}}\llbracket\underline{X}_\cal{S}\rrbracket\rightarrow E_{\otimes,\cal{S}}$. From the definitions, we find
    \[
\partial_{\p}\big(a_\cal{S}^{\underline{X}_\cal{S}}\big)\Big\lvert_{\underline{0}}=\log_p(a_\p)\otimes(\otimes_{\q\not=\p}1)\ \in\ E_{\otimes,\cal{S}}.
    \]
\end{remark}

 Let $f_A\colon\bb{T}^{\bfcdot}\twoheadrightarrow \Z_p$ be the homomorphism associated to our elliptic curve $A_{/F}$, we denote by $\cal{P}\in\mrm{Spec}(\bb{T}^{\bfcdot})$ its kernel which lay above the augmentation ideal $\cal{I}\in\mrm{Spec}(\Lambda^{\bfcdot})$. Write  $\bb{I}:=(\bb{T}^{\bfcdot})_\cal{P}$ for the associated localization and $ \scr{F}\colon \bb{T}^{\bfcdot}\longrightarrow\bb{I}$
 for the localization map.
By Corollary \ref{isolocalizations} the natural map $(\Lambda^{\bfcdot})_\cal{I}\overset{\sim}{\to} \bb{I}$ is an isomorphism. Hence, we obtain 
\[
\mrm{pw}_\otimes\colon\bb{I}\longhookrightarrow E_{\otimes,\cal{S}}\llbracket\underline{X}_\cal{S}\rrbracket.
\]
To simplify the notation, we set
\[
T^\otimes:=\mrm{pw}_\otimes(T)\qquad \forall\ T\in \bb{I}.
\]

\begin{proposition}\label{logarithmic derivative U_p}
  Let $\p\in\cal{S}$ and $\varpi_\p\in \p\cal{O}_\p$ a uniformizer. Set $\alpha_{\varpi_\p}=\scr{F}(\mbf{U}_{\varpi_\p})$ and denote by $q_\p\in \p\cal{O}_\p$ the Tate parameter of $A_{/F_\p}$. We have
  \[
  \frac{\partial_{\p}\big(\alpha_{\varpi_\p}^\otimes\big)}{\alpha_{\varpi_\p}^\otimes}\bigg\lvert_{\underline{0}}=-\frac{\log_{\varpi_\p}(q_\p)}{\mrm{ord}_\p(q_\p)}\otimes(\otimes_{\q\not=\p}1).
  \]
\end{proposition}
\begin{proof}
   We follow Greenberg--Stevens' strategy to compute arithmetic $\cal{L}$-invariants (for example, \cite[Section 2]{Greenberg-Stevens}). By (\cite{HidaGalois}, Theorem 1 $\&$ Proposition 2.3), the restriction  at a decomposition group $D_\p$ for $\p\in\cal{S}$ of the Galois representation $\varrho_\scr{F}\colon G_F\to \mrm{GL}_2(\bb{I})$ associated to the homomorphism $\scr{F}\colon \bb{T}^{\bfcdot}\to\bb{I}$ has the form 
    \[
    (\varrho_\scr{F})\big\lvert_{D_\p}\sim \begin{pmatrix}
        \cal{N}_F\cdot \delta_\p^{-1}&*\\ 
        0&\delta_\p
    \end{pmatrix}
    \]
    where $\cal{N}_F\colon D_\p\to\Z_p^\times$ is the $p$-adic cyclotomic character, and $\delta_\p\colon D_\p\to \bb{I}^\times$ is the character corresponding under local class field theory to the homomorphism 
    \[
    \xi_\p\colon F_\p^\times\longrightarrow \bb{I}^\times,\qquad u_\p\cdot \varpi_\p^{m}\mapsto [u_\p]\cdot \alpha_{\varpi_\p}^m
    \]
    for $u_\p\in\cal{O}_\p^\times, m\in\Z$. Consider the ring 
    \[
    \cal{R}_\p:=E_{\otimes,\cal{S}}\llbracket\underline{X}_\cal{S}\rrbracket/(X_\p^2,\{X_\q\}_{\q\not=\p})
    \]
    and the  homomorphism $\bb{I}\to \cal{R}_\p$ obtained by composing $\mrm{pw}_\otimes\colon \bb{I}\to E_{\otimes,\cal{S}}\llbracket\underline{X}_\cal{S}\rrbracket$ with the natural quotient map. Then, there exists a free $\cal{R}_\p$-module $V$ with a continuous $D_\p$-such that
    \begin{equation}\label{extension GS}
    \xymatrix{
    0\ar[r]& \cal{R}_\p(\cal{N}_{F})\ar[r]&V\ar[r]&\cal{R}_\p(\delta_\p^2)\ar[r]&0.
    }\end{equation}
    As in \cite[Theorem 2.3.4 ($\textbf{c}\hspace{-1mm}\implies\hspace{-1mm}\textbf
    {a}$)]{Greenberg-Stevens}, the existence of the extension \eqref{extension GS} implies that
    \[
q_\p\otimes1\in \mrm{H}^1\big(F_\p, E_{\otimes,\cal{S}}(\cal{N}_F)\big),\qquad \partial_\p(\xi_\p^2)\in \mrm{H}^1\big(F_\p, E_{\otimes,\cal{S}}\big)
    \]
    are orthogonal to each other under Tate's pairing 
    \[
    \langle-,-\rangle\colon\mrm{H}^1\big(F_\p, E_{\otimes,\cal{S}}(\cal{N}_F)\big)\times \mrm{H}^1\big(F_\p, E_{\otimes,\cal{S}}\big)\longrightarrow \mrm{H}^2\big(F_\p, E_{\otimes,\cal{S}}(\cal{N}_F)\big).
    \]
    The explicit description of Tate's paring (\cite{Greenberg-Stevens}, equation (2.3.3)) implies that
    \[
\partial_\p\Big(\xi_\p^2(q_\p)\Big)\Big\lvert_{\underline{0}}=0.
    \]
Now the claim follows by a direct calculation. Writing $q_\p=\langle q_\p\rangle\cdot \varpi_\p^{\mrm{ord}_\p(q_\p)}$, we find that
    \[
\partial_\p\Big(\xi_\p^2(q_\p)\Big)\Big\lvert_{\underline{0}}=\partial_\p\Big([\langle q_\p\rangle]^2\cdot (\alpha_{\varpi_\p}^\otimes)^{2\mrm{ord}_\p(q_\p)}\Big)\Big\lvert_{\underline{0}}.
\]
     On the one hand, using Remark \ref{one explicit deriviative} and $[\langle q_\p\rangle]\big\lvert_{\underline{0}}=1$,  we compute  that
    \[\begin{split}
    \partial_\p\Big([\langle q_\p\rangle]^2\Big)\Big\lvert_{\underline{0}}&=2\cdot \partial_\p\Big([\langle q_\p\rangle]\Big)\Big\lvert_{\underline{0}}\cdot [\langle q_\p\rangle]\Big\lvert_{\underline{0}}\\
    &=2\log_{\varpi_\p}(q_\p)\otimes(\otimes_{\q\not=\p}1).
    \end{split}\]
    On the other hand, as $ \big(\alpha^\otimes_{\varpi_\p}\big)^{2\mrm{ord}_\p(q_\p)}\Big\lvert_{\underline{0}}=1$ we see that
    \[\begin{split}
    \partial_\p\Big((\alpha^\otimes_{\varpi_\p})^{2\mrm{ord}_\p(q_\p)}\Big)\Big\lvert_{\underline{0}}&=2\mrm{ord}_\p(q_\p)\cdot \partial_\p\big(\alpha^\otimes_{\varpi_\p}\big)\Big\lvert_{\underline{0}}\cdot (\alpha^\otimes_{\varpi_\p})^{2\mrm{ord}_\p(q_\p)-1}\Big\lvert_{\underline{0}}\\
    &=2\mrm{ord}_\p(q_\p)\cdot \frac{\partial_\p\big(\alpha^\otimes_{\varpi_\p}\big)}{ \alpha_{\varpi_\p}^\otimes}\bigg\lvert_{\underline{0}}.
    \end{split}\]
\end{proof}

\section{Automorphic functions and Hecke theory}
 Let $\Sigma\subseteq \Sigma_p$ be a subset of $p$-adic places, $w\not\in\Sigma$ another place of $F$, and $K=\prod_vK_v\subseteq G(\A^{w,\infty})$ a compact open subgroup. We consider the $\Z_p$-module $\mrm{Hom}_{\Z_p}(\Z_p[\Xi], N)$ where
 \begin{itemize}
     \item [$\bfcdot$]  $\Xi$ is a set endowed with a left action of the image of $G(F)\to G_w$, $\delta\mapsto \delta_w$, and
      \item [$\bfcdot$]  $N$ is a left $\Z_p[K_\Sigma]$-module.
 \end{itemize}
 We equip $\mrm{Hom}_{\Z_p}(\Z_p[\Xi], N)$ with the induced \emph{left action} of $K_\Sigma$ and the \emph{right action} of $G(F)$:
 \[
 (k_\Sigma.\hspace{0.1mm}f.\hspace{0.1mm}\delta_w)(-):=k_\Sigma.\hspace{0.1mm}f(\hspace{0.1mm}\delta_w.-)
 \]
for all $k_\Sigma\in K_\Sigma$, $\delta\in G(F)$, and  $f\in \mrm{Hom}_{\Z_p}(\Z_p[\Xi], N)$. 

\begin{definition}
    The set $\cal{A}_{w,\Sigma}(K;\Xi, N)$ consists of those functions
    \[
\theta\colon G(\A^{w,\infty})\longrightarrow \mrm{Hom}_{\Z_p}(\Z_p[\Xi], N)
    \]
    such that 
    \[
   \theta(\delta^w\cdot  b\cdot k)=k_\Sigma^{-1}.\theta(b).\delta_w^{-1}\qquad\forall\ b\in G(\A^{w,\infty}),\  k\in K,\ \delta\in G(F).
    \]
\end{definition}

\subsubsection{Hecke theory.}\label{Hecke on aut forms}
The set $\cal{A}_{w,\Sigma}(K;\Xi, N)$ of automorphic functions always admits a left action of Hecke operators away from $\{w\}\cup\Sigma$. Indeed, for any $g\in G(\bb{A}^{\{w\}\cup\Sigma\cup\Sigma_\infty})$ let  $\{g_j\}_{j\in I}$ be left-coset representatives for $T_g$, i.e., $KgK=\coprod_{j\in I}g_jK$. As $(g_j)_\Sigma\in K_\Sigma$ for every $j\in I$, the expression 
  \[
  (T_g\theta)(b)=\sum_{j\in I}(g_j)_\Sigma.\theta(bg_j)\qquad\text{where}\qquad \theta\in \cal{A}_{w,\Sigma}(K;\Xi, N),\quad b\in G(\bb{A}^{w,\infty}),
  \]
 does not depend on the choice of coset representatives. To obtain an action of the Hecke operator $T_g$ for some $g\in G_\Sigma$, one has to endow the $\Z_p[K_\Sigma]$-module $N$ with a compatible action of the elements of the double coset $K g K$.

\subsection{Automorphic functions from the uniformization of Shimura curves}\label{Main examples}
As in Section \ref{Intro: pHc}, we have fixed an incoherent set $\cal{S}$ for the triple $(A_{/F},\hspace{.5mm} E/F,\hspace{.5mm}p)$ and a coherent set  $S=\cal{S}\cup\{w\}$ for the quadruple $(A_{/F},\hspace{.5mm} E/F,\hspace{.5mm}p,\hspace{.5mm}\cal{S})$. Let us also fix a subset 
\[
\Sigma\subseteq\cal{S}.
\]
Let $m\ge1$ and set $I_{m}:=\ker(D_\Sigma\to D_{\Sigma,m})$, so that the space of invariants $(J_\Sigma^\Box)^{I_{m}}$ is a right $\bb{T}_\Sigma\otimes_{\Lambda_\Sigma}\Lambda_{\Sigma,m}$-module. The $\Z_p$-module $(J_\Sigma^\Box)^{I_{m}}$ has a natural the left action $(K_m^\circ)_\Sigma$-action given by 
\begin{equation}\label{def: module structure on J}
k_\Sigma^{-1}.(-):=(-).\langle k_\Sigma\rangle_m\qquad\forall\ k\in (K_m^\circ)_\Sigma,
\end{equation}
while the set $\cal{H}_w^\Box$ is endowed with its natural $G(F)$-action twisted by the character \eqref{twisting character}.

\smallskip
\noindent The $w$-adic uniformization of $E_w$-points \eqref{def:Shimura curve} and $\mrm{CM}_E^w$ points \eqref{uniformization CM points} of the Shimura curve $X_m$ together with the morphisms 
\[
\varphi_m\colon X_m\longrightarrow J_m\otimes\bb{Z}_p,\qquad j_m\colon J_m^\Box\otimes\bb{Z}_p\longrightarrow J_\Sigma^\Box
\]
of Section \ref{Further observations} can be parleyed into two automorphic functions   
\[
\phi_m^\Box\in\cal{A}_{w,\Sigma}\big(K_m^\circ; \cal{H}_w^\Box, (J_\Sigma^\Box)^{I_m}\big)\qquad \text{for}\qquad \Box\in\{\text{loc}, \text{glo}\}.
\]
More precisely, we have:
\begin{lemma}
For $\Box\in\{\text{loc}, \text{glo}\}$, $m\ge1$, the function $\phi_m^\Box\colon G(\A^{w,\infty})\rightarrow \mrm{Hom}_{\Z_p}(\Z_p[\cal{H}_w^\Box], (J_\Sigma^\Box)^{I_m})$
described by the formula
\[\phi_m^\Box(b)\colon [z]\mapsto j_m\circ\varphi_m\big([z,b]_{m}\big)\qquad \forall\ b\in G(\A^{w,\infty}),\ z\in \cal{H}_w^\Box,
\]
belongs to $\cal{A}_{w,\Sigma}(K_m^\circ; \cal{H}_w^\Box, (J_\Sigma^\Box)^{I_m})$. Moreover, $\phi_m^\text{loc}, \phi_m^\text{glo}$ coincide in  $\cal{A}_{w,\Sigma}(K_m^\circ; \cal{H}_w^\text{glo}, (J_\Sigma^\text{loc})^{I_m})$.
\end{lemma}
\begin{proof}
This a direct consequence of the definitions. For additional clarity -- noting equation \eqref{def: module structure on J} --  we compute that
\[\begin{split}
\phi_m^\Box(\delta^w\cdot b\cdot k)(z)&=j_m\circ\varphi_m\big([z,\delta^w\cdot b\cdot k]_{m}\big)\\
&=j_m\circ\varphi_m\big([\delta_w^{-1}. z,b\cdot k]_{m}\big)\\
&=(k^{-1}_\Sigma.\phi_m^\Box.\delta_w^{-1}) (b)(z)
\end{split}\]
for any  $k\in K_m^\circ$, $\delta\in G(F)$, and $b\in G(\A^{w,\infty})$.
\end{proof}

\begin{lemma}\label{intro: trace map}
The formula
\[
\mrm{Tr}^{m+1}_{m}(\theta)(b):=\sum_{\xi\in K_{m}/K_{{m+1}}}\theta(b \xi)\qquad \forall\ b\in G(\A^{w,\infty})
\] 
defines a map
 \[
\mrm{Tr}^{m+1}_{m}\colon \cal{A}_{w,\Sigma}\big(K_{m+1}^\circ; \cal{H}_w^\Box, (J_\Sigma^\Box)^{I_{m+1}}\big)\longrightarrow \cal{A}_{w,\Sigma}\big(K_m^\circ; \cal{H}_w^\Box, (J_\Sigma^\Box)^{I_m}\big),
\]
\end{lemma}
\begin{proof}
   For any $k\in K_m$ 
   \begin{equation}\label{K_m-invariance}
   \begin{split}
    \mrm{Tr}^{m+1}_{m}(\theta)(bk)&=\sum_{\xi\in K_{m}/K_{{m+1}}}\theta(b\cdot k \xi)\\
&=\sum_{\xi'\in K_{m}/K_{{m+1}}}\theta(b\cdot \xi')=\mrm{Tr}^{m+1}_{m}(\theta)(b).\\
\end{split}
   \end{equation}
Now, if $k\in K_m^\circ$, there exists $a\in\cal{O}_\Sigma^\times$ such that  $kK_m=\gamma_aK_m$ -- see \eqref{explict description diamond} and \eqref{def: gamma_a}. Note that $\gamma_a\in K_{m+1}^\circ$ . Using \eqref{K_m-invariance} we compute
\[\begin{split}
\mrm{Tr}^{m+1}_{m}(\theta)(bk)&=\mrm{Tr}^{m+1}_{m}(\theta)(b\gamma_a)\\
&=\sum_{\xi\in K_{m}/K_{{m+1}}}\theta(b\cdot (\gamma_a\xi\gamma_a^{-1})\cdot \gamma_a)=k^{-1}.\mrm{Tr}^{m+1}_{m}(\theta)(b).\\
\end{split}\] 
Hence, $\mrm{Tr}^{m+1}_{m}(\theta)\in \cal{A}_{w,\Sigma}\big(K_m^\circ; \cal{H}_w^\Box, (J_\Sigma^\Box)^{I_{m}}\big)$.
\end{proof}

Interestingly, relation \eqref{compatibility} implies that the collection $\{\phi_m^\Box\}_{m\ge1}$ of automorphic functions is compatible with respect to the trace maps
giving rise to an element $\phi_\infty^\Box$ in the projective limit $\varprojlim_m \cal{A}_{w,\Sigma}(K_m^\circ; \cal{H}_w^\Box, (J_\Sigma^\Box)^{I_m})$. 

\subsubsection{Hecke module structures.}
The space $\cal{A}_{w,\Sigma}(K_m^\circ; \cal{H}_w^\Box, (J_\Sigma^\Box)^{I_m})$ inherits a right $\bb{T}_\Sigma\otimes_{\Lambda_\Sigma}\Lambda_{\Sigma,m}$-module structure from $(J_\Sigma^\Box)^{I_m}$. On top of that, we can endow it with a left action of $\frak{H}_{\Sigma,m}$ in the sense of Section \ref{Hecke on aut forms} as follows: 
    
    for $g\in G(\A^{w,\infty})$ write $K_m g K_m=\coprod_{j\in I}g_j K_m$, then
    \[
  (T_g\theta)(-)=\sum_{j\in I}\theta(-g_j)\qquad\forall\ \theta\in \cal{A}_{w,\Sigma}(K_m^\circ; \cal{H}_w^\Box, (J_\Sigma^\Box)^{I_m}).
  \]

Recall the the morphism $\frak{h}_{\Sigma,m+1}\to\frak{h}_{\Sigma,m}$ of Remark \ref{abstract compatibility}.
\begin{lemma}
    For all $m\ge1$ the trace map
\[
\mrm{Tr}^{m+1}_{m}\colon \cal{A}_{w,\Sigma}\big(K_{{m+1}}^\circ;\cal{H}_w^\Box,(J^\Box_\Sigma)^{I_{m+1}}\big)\longrightarrow \cal{A}_{w,\Sigma}\big(K_{m}^\circ;\cal{H}_w^\Box,(J^\Box_\Sigma)^{I_m}\big)
\]  
intertwines the $\frak{h}_{\Sigma,m+1}$-structure on the domain with the $\frak{h}_{\Sigma,m}$-structure on the codomain.
\end{lemma}
\begin{proof}
The same ideas appearing in the proof of Proposition \ref{Hecke-action on J} work here. In particular, the claim is clear for the Hecke operators away from $\Sigma$ and the diamond operators (see also Lemma \ref{intro: trace map}). To prove that, for $\p\in\Sigma$ and $\varpi\in\cal{O}_\p$ a uniformizer,
\[
\mrm{Tr}^{m+1}_m\circ ({_{m+1}}U_\varpi)=({_{m}}U_\varpi)\circ \mrm{Tr}^{m+1}_m\qquad \text{for}\ m\ge1
\]
write $\mrm{Tr}=\mrm{Tr}^{m+1}_m$ and factor it  as $\mrm{Tr}=\mrm{Tr}^2\circ\mrm{Tr}^1$ where $\mrm{Tr}^1$ averages over $(K_{m}\cap K_{m+1}^\circ)/K_{{m+1}}$ and $\mrm{Tr}^2$ averages over $K_{m}/(K_{m}\cap K_{m+1}^\circ)$.
The explicit coset representatives given in the proof of Proposition \ref{Hecke-action on J} allows one to check that 
\[
\mrm{Tr}^1\circ ({_{m+1}}U_\varpi)=({_{m+1}}U_\varpi)\circ \mrm{Tr}^1,\qquad \mrm{Tr}^2\circ ({_{m+1}}U_\varpi)=({_{m}}U_\varpi)\circ \mrm{Tr}^2.
\]
\end{proof}

\begin{corollary}\label{Hecke-action on limit}
The projective limit
\[
\varprojlim_m\ \cal{A}_{w,\Sigma}\big(K_{m}^\circ;\cal{H}_w^\Box,(J^\Box_\Sigma)^{I_m}\big)
\]
inherits the action of $\frak{h}_\Sigma=\varprojlim_m\frak{h}_{\Sigma,m}$.
\end{corollary}

 Let $\lambda\colon\frak{h}_{\Sigma,m}\to \bb{T}_\Sigma\otimes_{\Lambda_\Sigma}\Lambda_{\Sigma,m}$  be the homomorphism of $\Lambda_{\Sigma,m}$-algebras described by
 \[\lambda(T)=\begin{cases}
     \mbf{T}_\q \text{ (resp. }\mbf{U}_\q\text{) }&\text{if }\  T=T_\q \text{  (resp. }U_\q\text{)},\\
    \mbf{U}_\varpi&\text{if }\ T={_{m}}U_\varpi.\\
\end{cases}
\]

  \begin{lemma}
   The action of $\frak{h}_{\Sigma,m}$ on $\phi_m^\Box$ is via the homomorphism $\lambda\colon\frak{h}_{\Sigma,m}\to \bb{T}_\Sigma\otimes_{\Lambda_\Sigma}\Lambda_{\Sigma,m}$.
  \end{lemma}
  \begin{proof}
     A direct inspection gives the equalities
  \begin{equation}\label{explict Hecke action}
T_\q.\phi_m^\Box=\phi_m^\Box.\mbf{T}_\q,\qquad U_\q.\phi_m^\Box=\phi_m^\Box.\mbf{U}_\q,\qquad{_{m}}U_\varpi.\phi_m^\Box=\phi_m^\Box.\mbf{U}_\varpi.
    \end{equation}
    Furthermore, one sees that
    \[
    \langle a\rangle_m.\phi_m^\Box=\phi_m^\Box.\langle a\rangle_m\qquad \forall\ a\in\cal{O}_\Sigma^\times.
    \]
  \end{proof}

\subsection{Measure-valued automorphic functions}

Shapiro's lemma suggests to interpret the projective limit appearing in Corollary \ref{Hecke-action on limit} in terms of measure-valued automorphic functions. Consider the $p$-adic domain
\[
\cal{X}_\Sigma:=\varprojlim_m K_0/K_m.
\]
The group $K_0$ acts on $\cal{X}_\Sigma$ by left multiplication. Moreover, $\cal{X}_\Sigma$ admits a right $\cal{O}_\Sigma^\times$-action by 
\begin{equation}\label{Lambda-structure on measures}
[\gamma]\hspace{0.7mm}\bfcdot\hspace{1mm} a:=[\gamma\cdot\gamma_a]\qquad\gamma_a=\begin{pmatrix}
    1&0\\0&a
\end{pmatrix}.
\end{equation}
It follows that the space of $J_\Sigma^\Box$-valued measures on $\cal{X}_\Sigma$ has two $\Lambda_\Sigma$-structures: one inherited from the coefficients and the other arising from the $\cal{O}_\Sigma^\times$-action \eqref{Lambda-structure on measures}. For applications to families of modular forms it is necessary to consider only those measures for which the two $\Lambda_\Sigma$-structures coincide.
\begin{definition}
  Let $\cal{M}^\Box_\Sigma:=\cal{M}(\cal{X}_\Sigma, J_\Sigma^\Box)$ be the space of $J_\Sigma^\Box$-valued measures $\mu$ on $\cal{X}_\Sigma$ satisfying 
  \begin{equation}\label{def: equiv prop}
\mu(U)\llangle a\rrangle=\mu(U\hspace{0.5mm}\bfcdot\hspace{1mm} a)\qquad\forall\ U\subseteq\cal{X}_\Sigma\hspace{3mm}\text{compact open},\quad\forall \ a \in\cal{O}_\Sigma^\times.
\end{equation}
The $\Z_p$-module $\cal{M}^\Box_\Sigma$ is endowed with the left $(K_0)_{\Sigma}$-action given by the formula
\[
(k_\Sigma.\mu)(-):=\mu(k_\Sigma^{-1}-)\qquad \forall\ k_\Sigma\in (K_0)_{\Sigma}.
\]  
\end{definition}

\smallskip
For $m\ge1$ let $\cal{X}_{m}\subseteq \cal{X}_\Sigma$ represent the image of  $K_m$ in $\cal{X}_\Sigma$. The morphism of $\bb{T}_\Sigma$-modules
\[
\mrm{sp}_m\colon \cal{M}^\Box_\Sigma\longrightarrow (J_\Sigma^\Box)^{I_m},\qquad \mrm{sp}_m(\mu)=\mu(\cal{X}_{m})
\]
 is $(K_m^\circ)_{\Sigma}$-equivariant. Hence, it induces a map between spaces of automorphic functions
\[
(\mrm{sp}_m)_*\colon\cal{A}_{w,\Sigma}\big(K_0; \cal{H}_w^\Box, \cal{M}^\Box_\Sigma\big)\longrightarrow \cal{A}_{w,\Sigma}\big(K_m^\circ; \cal{H}_w^\Box, (J_\Sigma^\Box)^{I_m}\big).
\]

\begin{proposition}\label{lemma: iso-specialization}
    The maps $\{(\mrm{sp}_m)_*\}_{m\ge1}$ determine an isomorphism of $\bb{T}_\Sigma$-modules
    \[
    \mrm{sp}_\infty\colon \cal{A}_{w,\Sigma}\big(K_0; \cal{H}_w^\Box, \cal{M}^\Box_\Sigma\big)\overset{\sim}{\longrightarrow} \varprojlim_m\ \cal{A}_{w,\Sigma}\big(K_m^\circ; \cal{H}_w^\Box, (J_\Sigma^\Box)^{I_m}\big).
    \]
\end{proposition}
\begin{proof}
Consider the $(K_0)_{\Sigma}$-equivariant injection
\[
\cal{M}^\Box_\Sigma\longhookrightarrow\varprojlim_m\ J_\Sigma^\Box\big[K_{0}/K_{m}\big],\qquad\mu\mapsto \left\{\sum_{\xi\in K_{0}/K_{m}}\mu\big(\xi\cdot\cal{X}_{m}\big)\cdot[\xi]\right\}_{m\ge1}
\]
where the projective limit is computed with respect to the quotient maps $K_0/K_{m+1}\to K_0/K_{m}$, and the $(K_0)_{\Sigma}$-action on group rings is given by the left multiplication of group-like elements. It follows that any element $\theta\in \cal{A}_{w,\Sigma}(K_0; \cal{H}_w^\Box, \cal{M}^\Box_\Sigma)$  is determined by the collection 
\[
\left\{\theta(b)(z)\big(\xi\cdot\cal{X}_{m}\big)\ \bigg\lvert\ b\in G(\A^{w,\infty}),\ z\in\cal{H}_w^\Box,\ \xi\in K_{0},\ m\ge1\right\}\subseteq J_\Sigma^\Box.
\]
The injectivity of $\mrm{sp}_\infty$ follows from the equality
\[
[(\mrm{sp}_m)_*\theta](b\xi)(z)=\theta(b)(z)\big(\xi\cdot \cal{X}_{m}\big),
\]
a direct consequence of the fact that $\theta$ has level $K_0$. 

We prove surjectivity by explicitly constructing the preimage of a trace compatible family $\{\theta_m\}_{m\ge1}$ of automorphic functions. We claim that the expression
\begin{equation}\label{expression-reconstruction}
\theta(b)(z)(\xi\cal{X}_m):=\theta_m(b\xi)(z)
\end{equation}
describes an element $\theta\in \cal{A}_{w,\Sigma}\big(K_0; \cal{H}_w^\Box, \cal{M}^\Box_\Sigma\big)$ such that $(\mrm{sp}_m)_*\theta=\theta_m$ for every $m\ge1$. The trace compatibility of the family $\{\theta_m\}_{m\ge1}$ implies that for every fixed $b\in G(\A^{w,\infty})$, $z\in\cal{H}_w^\Box$, the function $\theta(b)(z)(-)$ -- defined in equation \eqref{expression-reconstruction} -- determines a $J_\Sigma^\Box$-valued measure on $\cal{X}_\Sigma$. Note that the  equality of the sets 
\[
\gamma_a\cal{X}_m=\cal{X}_m\hspace{1mm}\bfcdot\hspace{1mm} a.
\]
holds for all $m\ge1$, and all  $a\in\cal{O}_\Sigma^\times$. Then, one deduces $\theta(b)(z)(-)\in\cal{M}^\Box_\Sigma$ by computing
\[\begin{split}
\theta(b)(z)(\xi\cal{X}_m\hspace{1mm}\bfcdot\hspace{1mm} a)&=\gamma_a^{-1}\theta_m(b\xi)(z)\\
&=\theta_m(b\xi)(z)\llangle a\rrangle\\
&=\theta(b)(z)(\xi\cal{X}_m)\llangle a\rrangle.
\end{split}\]
Finally, the fact that
\[
\theta(\delta^w\cdot b\cdot k)=k^{-1}_\Sigma.\theta(b).\delta_w^{-1}\qquad\forall\ \delta\in G(F),\ b\in G(\A^{w,\infty}),\ k\in K_{0},
\]
is a direct consequence of the analogous property of each $\theta_m$.
\end{proof}

\begin{corollary}\label{distinguished autfun}
    The automorphic function $\Phi_\Sigma^\Box\in \cal{A}_{w,\Sigma}(K_0; \cal{H}_w^\Box, \cal{M}^\Box_\Sigma)$ corresponding to the trace-compatible collection $\{\phi_m^\Box\}_{m\ge1}$ is explicitly given by 
    \[
    \Phi_\Sigma^\Box(b)(z)(\xi\cdot\cal{X}_m)=j_m\circ\varphi_m\big([z, b\xi]_{K_m}\big)\qquad\forall\ b\in G(\A^{w,\infty}),\ z\in\cal{H}_w^\Box,\ \xi\in K_0.
    \]
\end{corollary}

\subsubsection{Interlude on compatibilities.}
Let $\Sigma\subseteq\Sigma'$ be two subsets of $\cal{S}$. Since $\cal{X}_{\Sigma'}=\cal{X}_{\Sigma'\setminus\Sigma}\times\cal{X}_\Sigma$, the projection map $\pi_{\Sigma'/\Sigma}\colon \cal{X}_{\Sigma'}\longrightarrow\cal{X}_\Sigma$ has compact fibers. Moreover, the compatibility of the level structures provide a $D_\Sigma$-equivariant map $\iota_{\Sigma/{\Sigma'}}\colon J_\Sigma^\Box\longrightarrow  J_{\Sigma'}^\Box$. The morphisms of left $\Z_p[(K_0)_\Sigma]$-modules
\[
(\pi_{\Sigma'/\Sigma})_*\colon\cal{M}_{\Sigma'}^\Box\to\cal{M}(\cal{X}_\Sigma,J_{\Sigma'}^\Box),\qquad (\iota_{\Sigma/{\Sigma'}})_*\colon\cal{M}_{\Sigma}^\Box\to\cal{M}(\cal{X}_\Sigma,J_{\Sigma'}^\Box)
\]
allows us to consider
\[\xymatrix{
\cal{A}_{w,\Sigma'}\big(K_0; \cal{H}_w^\Box,\cal{M}_{\Sigma'}^\Box\big)\ar[rr]^-{(\pi_{\Sigma'/\Sigma})_*}&& \cal{A}_{w,\Sigma}\big(K_0; \cal{H}_w^\Box,\cal{M}(\cal{X}_\Sigma,J_{\Sigma'}^\Box)\big)\\
&&\cal{A}_{w,\Sigma}\big(K_0; \cal{H}_w^\Box,\cal{M}_{\Sigma}^\Box\big)\ar[u]^-{(\iota_{\Sigma/{\Sigma'}})_*}.
}\]
From the definitions it is not hard to see that
    \begin{equation}\label{inter-compatibilities}
   (\pi_{\Sigma'/\Sigma})_*\Phi_{\Sigma'}^\Box= (\iota_{\Sigma/{\Sigma'}})_*\Phi_\Sigma^\Box.
    \end{equation}

\subsection{Semigroup action on $p$-adic domains}\label{Semigroup action on p-adic domains}
The goal of this section and the next is the definition of an action of Hecke operators at the primes in  $\Sigma$ on the space of automorphic functions 
\[
\cal{A}_{w,\Sigma}\big(K_0;\cal{H}_w^\Box,\cal{M}^\Box_\Sigma\big).
\]
We follow the approach of Ash--Stevens \cite{Ash-Stevens}, similarly to \cite{LongoBala}, with particular care taken to track the dependencies arising from the choice of uniformizers.    Recall the $p$-adic space $\cal{X}_\Sigma=\varprojlim_m K_0/K_m$. By definition 
\[
K_m=K_0^\Sigma\cdot \prod_{\p\in\Sigma}K_{\p^m}\qquad\forall\ m\ge0,
\]
therefore there is a product decomposition
\begin{equation}\label{product decomposition}
\cal{X}_\Sigma=\prod_{\p\in\Sigma}\cal{X}_{\p}\qquad \text{where}\qquad\cal{X}_{\p}=K_{\p^0}/K_{\p^\infty}\quad\text{and} \quad K_{\p^\infty}=\cap_mK_{\p^m}.
\end{equation}
If we denote by $(\cal{O}_\p^2)'=(\cal{O}_\p^2)\setminus\p(\cal{O}_\p^2)$ the set of primitive vectors, the space $\cal{X}_{\p}$ can be described more concretely through the identification
    \begin{equation}\label{iso1}
        \cal{X}_{\p}\overset{\sim}{\longrightarrow}\cal{O}_\p^\times\backslash\Big[(\cal{O}_\p^2)'\times\cal{O}_\p^\times\Big],\qquad \gamma\mapsto \left[\gamma\begin{pmatrix}
        1\\0
    \end{pmatrix},\det\gamma\right].
    \end{equation}
\begin{remark}\label{explict-right-action}
    Under \eqref{iso1}, the right $\cal{O}_\p^\times$-action on $\cal{X}_\p$ introduced in equation \eqref{Lambda-structure on measures} is given by the formula
    \[
    \left[\begin{pmatrix}
        x\\y
    \end{pmatrix},z\right]\hspace{0.3mm}\bfcdot\hspace{1mm} a=\left[\begin{pmatrix}
        x\\y
    \end{pmatrix},z\cdot a\right]\qquad\forall\ a\in\cal{O}_\p^\times.
    \]
\end{remark}
    Let $N_\p\subset B_\p\subset  G_\p$ denote the inclusion of the subgroup $N_\p$ generated by upper unipotent matrices into the upper triangular Borel subgroup $B_\p$ of $G_\p$. We consider 
    \[
    \cal{Y}_\p:=G_\p/N_\p\qquad\text{and}\qquad \bb{P}^1(F_\p)\cong G_\p/B_\p.
    \]
    There is a natural surjection $ \cal{Y}_\p\twoheadrightarrow \bb{P}^1(F_\p)$ and an inclusion $\cal{X}_{\p}\subseteq \cal{Y}_\p$. If $(F_\p^2)'=F_\p^2\setminus \{(0,0)\}$, the following concrete description 
    \begin{equation}\label{iso2}
\cal{Y}_\p\overset{\sim}{\longrightarrow}F_\p^\times\backslash\Big[(F_\p^2)'\times F_\p^\times\Big],\qquad \gamma\mapsto \left[\gamma\begin{pmatrix}
        1\\0
    \end{pmatrix},\det\gamma\right]
    \end{equation}
is compatible with \eqref{iso1}. We set 
\[
\cal{Y}_\Sigma=\prod_{\p\in\Sigma}\cal{Y}_\p.
\]

\subsubsection{Extending the $K_{\p^0}$-action.}\label{action-extension}
Fix a prime $\p\in\Sigma$. For a uniformizer $\varpi\in \cal{O}_\p$ we let the matrix $\eta_\varpi=\begin{pmatrix}
    1&0\\0&\varpi
\end{pmatrix}$ act on $\cal{Y}_\p$ by conjugation 
\begin{equation}\label{pi-action}
\eta_\varpi\cdot[\gamma]:=[\eta_\varpi\cdot \gamma \cdot\eta_\varpi^{-1}]\qquad \forall\ \gamma\in G_\p.
\end{equation}
Under the identification \eqref{iso2} the action is given by the formula
\begin{equation}\label{explict-action}
\eta_\varpi\cdot\left[\begin{pmatrix}
    x\\y
\end{pmatrix},z\right]=\left[\begin{pmatrix}
    x\\\varpi y
\end{pmatrix},z\right].
\end{equation}
\begin{remark}\label{stability}
   The set $\cal{X}_{\p}$ is not stable under the action of $\eta_\varpi$, i.e., $\eta_\varpi\cdot\cal{X}_{\p}\not\subseteq\cal{X}_\p$. However, its subset $\cal{X}_{\p^m}^\circ:=K_{\p^m}^\circ/K_{\p^\infty}$ satisfies 
   \[
\eta_\varpi\cdot\cal{X}^\circ_{\p^m}=\cal{X}_{\p^{m+1}}^\circ\subset\cal{X}^\circ_{\p^m}\qquad\forall\ m\ge1.
\]
\end{remark}
Consider the function $f_\p\colon \cal{Y}_\p\rightarrow \Z$ defined by 
    \[
    \left[\begin{pmatrix}
    x\\y
\end{pmatrix},z\right]\mapsto \mrm{min}\left\{\mrm{ord}\left(\frac{x^2}{z}\right), \mrm{ord}\left(\frac{y^2}{z}\right)\right\}.
    \]
  \begin{lemma}\label{ord-function}
      The function $f_\p\colon \cal{Y}_\p\rightarrow \Z$ satisfies:
      \begin{itemize}
          \item[$\bfcdot$] $f_\p^{-1}(0)=\cal{X}_{\p}$,
          \item[$\bfcdot$] $f_\p(\eta_\varpi\cdot\gamma)\ge f(\gamma)$ for all $\gamma\in\cal{Y}_\p$,
          \item[$\bfcdot$]  and $f_\p(\kappa \cdot\gamma)=f_\p(\gamma)$ for all $\kappa\in K_{\p^0}$ and $\gamma\in\cal{Y}_\p$.
      \end{itemize} 
  \end{lemma}
\begin{proof}
The first two claims are direct consequences of the definition of the function $f_\p\colon \cal{Y}_\p\rightarrow \Z$ and the description of the $\eta_\varpi$-action given in \eqref{explict-action}. For the last claim, it is easy to check that 
\begin{equation}\label{easy-inequality}
f_\p(\kappa \cdot\gamma)\ge f_\p(\gamma)\qquad \forall\ \kappa\in K_{\p^0},\ \forall\ \gamma\in\cal{Y}_\p,
\end{equation}
and the reversed inequality is obtained from \eqref{easy-inequality} as follows 
\[
f_\p(\gamma)=f_\p(\kappa^{-1}\cdot(\kappa\cdot \gamma))\ge f_\p(\kappa\cdot\gamma).
\]
\end{proof}

Let $\Delta^\varpi_\p$ denote \emph{free semigroup}  generated by the set  $K_{\p^0}\cup\{\eta_\varpi\}$ and consider the morphism of semigroups $\mrm{deg}\colon \Delta^\varpi_\p\to\bb{N}$ defined by
\[
\deg(\eta_\varpi)=1,\qquad \deg(k)=0\qquad \forall\ k\in K_{\p^0}.
\]
We let $\Delta^\varpi_\p$ act on $\cal{Y}_\p$ by declaring that $K_{\p^0}$ acts by left multiplication and $\eta_\varpi$ acts by \eqref{pi-action}. If $g\in G_\p$ denotes the image of $\widetilde{g}\in \Delta_\p^\varpi$, the action is explicitly given by the formula
\begin{equation}\label{explicit-action-semigroup}
\widetilde{g}\cdot[\gamma]=[g\cdot\gamma\cdot\eta_\varpi^{-\mrm{deg}(\widetilde{g})}]\qquad\forall\ [\gamma]\in\cal{Y}_\p,
\end{equation}
which -- through the identification \eqref{iso2} -- becomes
\begin{equation}\label{explicit-action-semigroup-in-coordinates}
\widetilde{g}\cdot\left[\begin{pmatrix}
    x_\gamma\\y_\gamma
\end{pmatrix},z_\gamma\right]=\left[g\begin{pmatrix}
    x_\gamma\\ y_\gamma
\end{pmatrix},\ \frac{\det(g)}{\varpi^{\mrm{deg}(\widetilde{g})}}\cdot z_\gamma\right].
\end{equation}
 \begin{remark}\label{rmk: observation on action}
The $\Delta^\varpi_\p$-action on $\cal{Y}_\p$ does not induce an action of the image of $\Delta^\varpi_\p$ in $G_\p$. A counter-example is found by looking at the two sides of the equality
    \[
    \varpi\begin{pmatrix}
        0&1\\1&0
    \end{pmatrix}=\begin{pmatrix}
        1&0\\0&\varpi
    \end{pmatrix}\begin{pmatrix}
        0&1\\1&0
    \end{pmatrix}\begin{pmatrix}
        1&0\\0&\varpi
    \end{pmatrix}.
    \]
    However, from the explicit description of the objects involved it is easy to see that the natural projection $\cal{Y}_\p\twoheadrightarrow \bb{P}^1(F_\p)$ intertwines the action of $\Delta^\varpi_\p$ on $\cal{Y}_\p$ with the action of its image in $G_\p$ on $\bb{P}^1(F_\p)$.
\end{remark}

\begin{lemma}\label{when-actions-coincide}
   Let $\widetilde{g}_1,\widetilde{g}_2\in S_\p^\varpi$. If $\mrm{deg}(\widetilde{g}_1)=\mrm{deg}(\widetilde{g}_2)$ and $\widetilde{g}_1,\widetilde{g}_2$ have the same image in $G_\p$, then 
   \[
   \widetilde{g}_1\cdot[\gamma]=\widetilde{g}_2\cdot[\gamma]\qquad\forall\ [\gamma]\in\cal{Y}_\p.
   \]
\end{lemma}
\begin{proof}
    It follows directly from \eqref{explicit-action-semigroup}.
\end{proof}
\begin{remark}\label{observation-on-action}
   It is possible to define an action of any  $x\in \big[K_{\p^0}\cdot \eta_\varpi \cdot K_{\p^0}\big]$ on $\cal{Y}_\p$ by setting 
    \[
    x\cdot_\varpi[\gamma]:=[x\cdot \gamma\cdot\eta_\varpi^{-1}]\qquad \forall\ [\gamma]\in\cal{Y}_\p. 
    \]
    Note that the double coset $\big[K_{\p^0}\cdot \eta_\varpi \cdot K_{\p^0}\big]$ does not depend on the uniformizer $\varpi$ and that the different actions are related by the following formula 
    \begin{equation}\label{dependence on uniformizer of the action}
     x\cdot_\varpi[\gamma]= x\cdot_{a\varpi}[\gamma]\hspace{1mm}\bfcdot\hspace{1mm} a\qquad \forall\ a\in\cal{O}_\p^\times.
    \end{equation}
    \end{remark}

\begin{lemma}\label{Ash-Stevens}
    Let $\widetilde{\cal{Y}}_\p\subseteq \cal{Y}_\p$ denote the smallest subset containing $\cal{X}_{\p}$ and stable under $\Delta^\varpi_\p$. Then $\widetilde{\cal{Y}}_\p\setminus\cal{X}_{\p}$ is also stable under $\Delta^\varpi_\p$. 
\end{lemma}
\begin{proof}
    This is \cite[Lemma 3.1]{Ash-Stevens}. 
By Lemma \ref{ord-function} and the definition of $\widetilde{\cal{Y}}_\p$ we have 
    \[
    f_\p(\gamma)\ge0 \quad \forall\ \gamma\in\widetilde{\cal{Y}}_\p,\qquad\text{and}\qquad
    \widetilde{\cal{Y}}_\p\setminus\cal{X}_{\p}=\big\{\gamma\in\widetilde{\cal{Y}}_\p\ \big\lvert\ f_\p(\gamma)>0\big\}.
    \]
    The stability of $\widetilde{\cal{Y}}_\p\setminus\cal{X}_{\p}$ under $\Delta^\varpi_\p$ is then clear. 
\end{proof}

\subsection{Semigroup action on measures}\label{action on distribution}
Fix a uniformizer $\varpi_\p\in \cal{O}_\p$ for every $\p\in\Sigma$. Lemma \ref{Ash-Stevens} provides sets $\{\widetilde{\cal{Y}}_\p\}_{\p\in\Sigma}$ which are stable under the action of the free semigroups $\{\Delta^{\varpi_\p}_\p\}_{\p\in\Sigma}$. Moreover, the product $\widetilde{\cal{Y}}_\Sigma=\prod_{\p\in\Sigma}\widetilde{\cal{Y}}_\p$ satisfies $\cal{X}_\Sigma\subset \widetilde{\cal{Y}}_\Sigma\subset \cal{Y}_\Sigma$ and is stable under the natural action of the product $\Delta^\varpi_\Sigma:=\prod_{\p\in\Sigma}\Delta_\p^{\varpi_\p}$. \begin{definition}
  Let $\widetilde{\cal{M}}^\Box_\Sigma:=\cal{M}(\widetilde{\cal{Y}}_\Sigma, J_\Sigma^\Box)$ be the space of $J_\Sigma^\Box$-valued measures $\mu$ on $\widetilde{\cal{Y}}_\Sigma$ satisfying 
  \[
\mu(V)\llangle a\rrangle=\mu(V\hspace{0.5mm}\bfcdot\hspace{1mm} a)\qquad\forall\ V\subseteq\widetilde{\cal{Y}}_\Sigma\hspace{3mm}\text{compact open},\quad\forall \ a \in\cal{O}_\Sigma^\times.
\]
The $\Z_p$-module $\widetilde{\cal{M}}^\Box_\Sigma$ is endowed with the left $(K_0)_{\Sigma}$-action given by the formula
\[
(k_\Sigma.\mu)(-):=\mu(k_\Sigma^{-1}-)\qquad \forall\ k_\Sigma\in (K_0)_{\Sigma}.
\]  
\end{definition}

\bigskip
The spaces of measures $\cal{M}^\Box_\Sigma$ and $\widetilde{\cal{M}}^\Box_\Sigma$ are related by the $(K_0)_{\Sigma}$-equivariant morphisms
\[\begin{split}
&\iota\colon \cal{M}^\Box_\Sigma\hookrightarrow \widetilde{\cal{M}}^\Box_\Sigma,\qquad (\iota\mu)(-):=\mu(-\cap\cal{X}_\Sigma), \\
&p\colon\widetilde{\cal{M}}^\Box_\Sigma\to\cal{M}^\Box_\Sigma,\qquad (p\mu)(-):=\mu(-)
\end{split}\]
which satisfy $p\circ \iota= \mrm{id}_{\cal{M}^\Box_\Sigma}$.
Furthermore, the semigroup action of $\Delta^\varpi_\Sigma$ on  $\widetilde{\cal{M}}^\Box_\Sigma$ is described by: 

if $\widetilde{g}\in \Delta^\varpi_\Sigma$ and $\mu\in \widetilde{\cal{M}}^\Box_\Sigma$, then
\begin{equation}\label{def: action distributions}
(\widetilde{g}.\mu)(V):=\mu\big(\widetilde{g}^{-1}V\cap \widetilde{\cal{Y}}_\Sigma\big)\qquad \forall\ V\subseteq\widetilde{\cal{Y}}_\Sigma\hspace{3mm}\text{compact open}.
\end{equation}

\begin{remark}\label{dependence on uniformizers action on distributions}
  Fix $\p\in\Sigma$, $\varpi\in\cal{O}_\p$ a uniformizer, and let $x\in \big[K_{\p^0}\cdot \eta_\varpi \cdot K_{\p^0}\big]$.  We have
  \[
  (x\cdot_{a\varpi}\mu)(-)=(x\cdot_{\varpi}\mu)(-)\llangle a\rrangle\qquad\forall\ a \in\cal{O}_\p^\times
  \]
  because equation \eqref{dependence on uniformizer of the action} implies that for any open compact $V\subseteq \widetilde{\cal{Y}}_\Sigma$ we have 
  \[
 (x^{-1}\cdot_{a\varpi}V)\cap\widetilde{\cal{Y}}_\Sigma=(x^{-1}\cdot_{\varpi}V\hspace{1mm}\bfcdot \hspace{1mm}a)\cap\widetilde{\cal{Y}}_\Sigma=\big((x^{-1}\cdot_{\varpi}V)\cap\widetilde{\cal{Y}}_\Sigma\big)\hspace{1mm}\bfcdot \hspace{1mm}a.
\]
\end{remark}

\begin{lemma}\label{stable-kernel}
    The kernel of $p\colon\widetilde{\cal{M}}^\Box_\Sigma\to\cal{M}^\Box_\Sigma$ is stable under the action of $\Delta^\varpi_\Sigma$.
\end{lemma}
\begin{proof}
    Let $\mu\in\widetilde{\cal{M}}^\Box_\Sigma$ such that $p(\mu)=0$, that is 
    \[
    \mu(U)=0\qquad \forall\ U\subseteq \cal{X}_\Sigma\hspace{3mm} \text{open compact}.
    \]
    We want to prove that $p(\widetilde{g}.\mu)=0$ for every $\widetilde{g}\in \Delta^\varpi_\Sigma$.  Let's consider a subset $U\subseteq \cal{X}_\Sigma$. Since
    \[
    \widetilde{\cal{Y}}_\Sigma=\prod_{\p\in\Sigma}\ \left(\cal{X}_{\p}\ \coprod\ (\widetilde{\cal{Y}}_\p\setminus\cal{X}_{\p})\right)
    \]
    and since Lemma \ref{Ash-Stevens} provides the inclusion
    \[
\widetilde{g}_\p\cdot(\widetilde{\cal{Y}}_\p\setminus\cal{X}_{\p})\subseteq (\widetilde{\cal{Y}}_\p\setminus\cal{X}_{\p})\qquad \forall\ \widetilde{g}_\p\in \Delta^{\varpi_\p}_\p,
    \]
    it follows that
    $U\cap \big(\widetilde{g}\cdot \widetilde{\cal{Y}}_\Sigma\big)\subseteq \widetilde{g}\cdot \cal{X}_\Sigma$, or equivalently that $\big(\widetilde{g}^{-1}\cdot U\big)\cap \widetilde{\cal{Y}}_\Sigma\subseteq\cal{X}_\Sigma$.  We deduce that $p(\widetilde{g}.\mu)=0$ as required.
\end{proof}

\begin{corollary}\label{cor: action on Distributions}
   There is an action of the semigroup $\Delta^\varpi_\Sigma$ on $\cal{M}^\Box_\Sigma$ described by 
\[
\widetilde{g}\star\mu:=p(\widetilde{g}.\iota(\mu))\qquad \forall\ \widetilde{g}\in \Delta^\varpi_\Sigma,\ \ \forall\ \mu\in\cal{M}^\Box_\Sigma.
\]
Moreover, $\widetilde{g}\star\mu$ is supported on $\widetilde{g}(\cal{X}_\Sigma)\cap\cal{X}_\Sigma$.
\end{corollary}
\begin{proof}
   Let $\mu\in\cal{M}^\Box_\Sigma$.  To prove the first claim we need to check that 
    \[
    (\widetilde{g}\cdot \widetilde{h})\star\mu=\widetilde{g}\star(\widetilde{h}\star\mu)\qquad \forall\ \widetilde{g},\widetilde{h}\in \Delta^\varpi_\Sigma.
    \]
Thanks to Lemma \ref{stable-kernel} it is enough to verify that
    \[
   p\big(\widetilde{h}.\iota(\mu)-\iota(\widetilde{h}\star\mu)\big)=\widetilde{h}\star\mu-(p\circ \iota)(\widetilde{h}\star\mu)=0.
    \]
  The second claim follows directly from the definitions.
\end{proof}

\begin{corollary}\label{Hecke-action of families}
  The space of automorphic functions $\cal{A}_{w,\Sigma}(K_0;\cal{H}_w^\Box,\cal{M}^\Box_\Sigma)$
admits an action of the Hecke operators 
\[
\big\{T_\q\ :\ \q\nmid\frak{f}_A\big\}\cup \big\{U_\q\ :\ \q\mid \n^{\mbox{\tiny $+$}}\cdot p_{\cal{S}\setminus\Sigma}\big\}\cup\big\{T_{\varpi}\ :\ \p\in\Sigma,\ \varpi\in\cal{O}_\p\ \text{uniformizer} \big\}.
\] 
\end{corollary}
\begin{proof}
The only claim that needs explaining is the existence of the action of $T_{\varpi}$ for $\varpi\in\cal{O}_\p$ a  uniformizer and $\p\in\Sigma$.  As noted in Section \ref{Hecke on aut forms}, it suffices to describe an action of elements of $[K_0\cdot \eta_\varpi \cdot K_0]$ on $\cal{M}^\Box_\Sigma$. For any fixed uniformizer, Remark \ref{observation-on-action} and Corollary \ref{cor: action on Distributions} allows us to finish the proof.
\end{proof}

\begin{remark}
Since the double coset $[K_0\cdot \eta_\varpi \cdot K_0]$ is independent of the choice of $\varpi\in\cal{O}_\p$, the dependence of the operator $T_\varpi$ on the uniformizer arises solely from the action \eqref{pi-action}. Remark \ref{dependence on uniformizers action on distributions} then implies
    \[
    T_{a\varpi}(-)=T_{\varpi}(-)\llangle a\rrangle\qquad \forall\ a\in\cal{O}_\p^\times.
    \]
\end{remark}

\subsection{Intertwining of Hecke operators}
The objective of this section is show that the isomorphism of $\bb{T}_\Sigma$-modules
 \[
    \mrm{sp}_\infty\colon \cal{A}_{w,\Sigma}\big(K_0; \cal{H}_w^\Box, \cal{M}^\Box_\Sigma\big)\overset{\sim}{\longrightarrow} \varprojlim_m\ \cal{A}_{w,\Sigma}\big(K_m^\circ; \cal{H}_w^\Box, (J^\Box_\Sigma)^{I_m}\big).
    \]
arising from Proposition \ref{lemma: iso-specialization} intertwines the Hecke action on its domain (Corollary \ref{Hecke-action of families}) with that on its codomain (Corollary \ref{Hecke-action on limit}). As a corollary we will promote \eqref{explict Hecke action} to a description of the action of the Hecke operators on the distinguished automorphic function (see Corollary \ref{distinguished autfun})
\[
\Phi_\Sigma^\Box\in \cal{A}_{w,\Sigma}(K_0;\cal{H}_w^\Box,\cal{M}^\Box_\Sigma).
\]

\subsubsection{Set up.}
For $m\ge0$ consider the subsets $\cal{X}_m\subseteq \cal{X}_m^\circ$ of $\cal{X}_\Sigma$ given by
\[
\cal{X}_m:=K_m/K_{\infty},\qquad \cal{X}_{m}^\circ:=K_{m}^\circ/K_{\infty}
\]
where $K_\infty=\cap_mK_m.$ Note that $\cal{X}_0=\cal{X}_0^\circ=\cal{X}_\Sigma$. 
\begin{definition}
  For any $m\ge0$ let $\cal{M}^\Box_{\Sigma,m}:=\cal{M}(\cal{X}_m^\circ, J_\Sigma^\Box)$ denote the space of $J_\Sigma^\Box$-valued measures $\mu$ on $\cal{X}_m^\circ$ satisfying 
  \[
\mu(U)\llangle a\rrangle=\mu(U\hspace{0.5mm}\bfcdot\hspace{1mm} a)\qquad\forall\ U\subseteq\cal{X}_m^\circ\hspace{3mm}\text{compact open},\quad\forall \ a \in\cal{O}_\Sigma^\times.
\]
The $\Z_p$-module $\cal{M}^\Box_{\Sigma,m}$ is endowed with the left $(K_m^\circ)_{\Sigma}$-action given by the formula
\[
(k_\Sigma.\mu)(-):=\mu(k_\Sigma^{-1}-)\qquad \forall\ k_\Sigma\in (K_m^\circ)_{\Sigma}.
\]  
\end{definition}
By Remark \ref{stability}, for every $m\ge1$ the space of automorphic functions
\[
\cal{A}_{w,\Sigma}(K_{m}^\circ;\cal{H}_w^{\Box},\cal{M}_{\Sigma,m}^{\Box})
\]
admits an action of the Hecke operators 
\[
\big\{T_\q\ :\ \q\nmid\frak{f}_A\big\}\cup\big\{U_\q\ :\ \frak{q}\mid \n^{\mbox{\tiny $+$}}\cdot p_{\cal{S}\setminus\Sigma}\big\}\cup\big\{{_{m}}U_{\varpi}\ :\ \p\in\Sigma,\ \varpi\in\cal{O}_\p\ \text{uniformizer} \big\}.
\]
To simplify the comparison with other references,
for any $\Phi\in \cal{A}_{w,\Sigma}(K_{m}^\circ;\cal{H}_w^{\Box},\cal{M}_{\Sigma,m}^{\Box})$ and some fixed $\p\in\Sigma$ we describe the automorphic function ${_{m}}U_{\varpi}\Phi$  by the formula
    \begin{equation}\label{formula m_U_p-action}
    \int_{\cal{X}_{m}^\circ} f(t)\hspace{1.5mm}\mrm{d}({_{m}}U_\varpi.\Phi(b)[z])(t)=\sum_{a\in\kappa_\p}\int_{\cal{X}_{m}^\circ}f(\widehat{\sigma}_{m,a}\cdot t)\hspace{1.5mm}\mrm{d}(\Phi(b\widehat{\sigma}_{m,a})[z])(t)
    \end{equation}
valid for any $b\in G(\A^{w,\infty}), z\in\cal{H}_w^{\Box}$, and any locally constant function $f\in \mrm{LC}(\cal{X}_m^\circ,\Z_p)$. The key observation to obtain \eqref{formula m_U_p-action} is the equality 
\[
\widehat{\sigma}_{m,a}^{-1}\big(\cal{X}_{m}^\circ\big)\cap \cal{X}_{m}^\circ=\cal{X}_{m}^\circ
\]
which holds for any $m\ge1$ -- see equation \eqref{cosets-decompositions} for the definition of $\widehat{\sigma}_{m,a}\in G_\p$.

\bigskip
\medskip
For any $0\le j\le k$ the inclusion $\cal{X}_k^\circ\subseteq\cal{X}_j^\circ$ determines the $(K_{k}^\circ)_\Sigma$-equivariant restriction of measures $\mrm{res}^j_k\colon \cal{M}_{\Sigma,j}^{\Box}\rightarrow \cal{M}_{\Sigma,k}^{\Box}$ and the associated function
\[
    (\mrm{res}^j_{k})_*\colon \cal{A}_{w,\Sigma}\big(K_{j}^\circ;\cal{H}_w^{\Box},\cal{M}_{\Sigma,j}^{\Box}\big)\longrightarrow\cal{A}_{w,\Sigma}\big(K_{k}^\circ;\cal{H}_w^{\Box},\cal{M}_{\Sigma,k}^{\Box}\big).
    \]
To study how the restriction maps $(\mrm{res}^j_{k})_*$ intertwine the Hecke action it is convenient to note that -- similarly to \eqref{product decomposition} -- there is a product decomposition
\begin{equation}\label{product decomposition level m}
\cal{X}_{m}^\circ=\prod_{\p\in\Sigma}\cal{X}_{\p^m}^\circ\qquad \text{where}\qquad \cal{X}_{\p^m}^\circ=K_{\p^m}^\circ/K_{\p^\infty}.
\end{equation}

\begin{proposition}\label{restriction - intertwining}
  For every $m\ge0$, the map induced by restriction of measures
    \[
    (\mrm{res}^m_{m+1})_*\colon \cal{A}_{w,\Sigma}\big(K_{m}^\circ;\cal{H}_w^{\Box},\cal{M}_{\Sigma,m}^{\Box}\big)\overset{{\sim}}{\longrightarrow}\cal{A}_{w,\Sigma}\big(K_{m+1}^\circ;\cal{H}_w^{\Box},\cal{M}_{\Sigma,m+1}^{\Box}\big)
    \]
    is an isomorphism commuting with the Hecke operators away from $\Sigma$. Moreover, for every $\p\in\Sigma$ and uniformizer $\varpi\in\cal{O}_\p$, it intertwines $T_\varpi$ with ${_{1}}U_\varpi$ when $m=0$, and  ${_{m}}U_\varpi$ with ${_{m+1}}U_\varpi$ when $m\ge1$.
\end{proposition}
\begin{proof}
The product structures \eqref{product decomposition level m} and the corresponding factorization of measures \eqref{factorization distributions} allows one to work one prime at the time and essentially reduces the proof to the content of \cite[Appendix II]{MokHeegner}. We sketch the main steps of the argument for the convenience of the reader. 

Let $m\ge0$. To see that $(\mrm{res}^m_{m+1})_*$ is an isomorphism, take $\Phi\in \cal{A}_{w,\Sigma}(K_{m}^\circ;\cal{H}_w^{\Box},\cal{M}_{\Sigma,m}^{\Box})$ and write 
\[
\cal{X}_{m}^\circ=\coprod_{g\in K_m^\circ/K^\circ_{m+1}}\big(g\cdot \cal{X}_{{m+1}}^\circ\big).
\]
For any $b\in G(\A^{w,\infty})$, $z\in\cal{H}_w^{\Box}$ then  one computes that
\begin{equation}\label{the formula}
\int_{\cal{X}_{m}^\circ} f(t)\hspace{1.5mm}\mrm{d}(\Phi(b)[z])(t)=\sum_{g\in K_m^\circ/K^\circ_{m+1}}\int_{\cal{X}_{m+1}^\circ}f(g\cdot t)\hspace{1.5mm}\mrm{d}(\Phi(bg)[z])(t)
\end{equation}
using \eqref{def: action distributions} and the defining properties of $\Phi$. It follows that $(\mrm{res}^m_{m+1})_*$ is injective because $\Phi$ turns out to be completely determined by $ (\mrm{res}^m_{m+1})_*\Phi$. Moreover, surjectivity is obtained by observing that the expression \eqref{the formula} can be used to construct a lift $\widetilde{\hspace{1mm}\Phi'}$ for any $\Phi'\in \cal{A}_{w,\Sigma}(K_{m+1}^\circ;\cal{H}_w^{\Box},\cal{M}_{\Sigma,m+1}^{\Box})$.

It is clear that $ (\mrm{res}^m_{m+1})_*$ commutes with the Hecke operators away from $\Sigma$. To prove that $ (\mrm{res}^m_{m+1})_*$ intertwines the action of the Hecke operators at $\p\in\Sigma$, one sets $\Phi':=(\mrm{res}^m_{m+1})_*\Phi$ and checks that for any $\p\in\Sigma$
\[
T_\varpi.\Phi=\widetilde{{_{1}}U_\varpi.\Phi'}\quad\text{if}\hspace{2mm}  m=0,\qquad\text{and}\qquad {_{m}}U_\varpi.\Phi=\widetilde{{_{m+1}}U_\varpi.\Phi'}\quad\text{if}\hspace{2mm}  m\ge1.
\]
 As mentioned before, one is allowed to focus only on distributions over the local pieces $\cal{X}_{\p^{m+1}}^\circ\subset\cal{X}_{\p^m}^\circ$, so that the details of the computation can be extracted from the proof of \cite[Proposition 8.2]{MokHeegner}. The calculations rely on three main ingredients:
 \begin{itemize}
     \item [$\bfcdot$] explicit coset representatives
 \[
  \cal{X}_{\p^0}\setminus \cal{X}_{\p^1}^\circ=\coprod_{a\in k_\p}\begin{pmatrix}
        a&1\\-1&0
\end{pmatrix}\cal{X}_{\p^1}^\circ,\qquad  \cal{X}_{\p^m}^\circ\setminus \cal{X}_{\p^{m+1}}^\circ=\coprod_{a\in k_\p}\begin{pmatrix}
        1&0\\a\varpi^m&1
    \end{pmatrix}\cal{X}_{\p^{m+1}}^\circ\quad\text{for}\hspace{2mm} m\ge1,
 \]
      \item [$\bfcdot$] the inclusions
\[
\sigma_\infty(\cal{X}_{\p^0})\cap\cal{X}_{\p^0}=\cal{X}_{\p^1}^\circ\qquad\text{and}\qquad\sigma_a(\cal{X}_{\p^0})\cap\cal{X}_{\p^0}\subseteq\cal{X}_{\p^0}\setminus \cal{X}_{\p^1}^\circ\qquad\forall\ a\not=\infty,
\]
where $\{\sigma_a\ \lvert \ \bar{a}\in\bb{P}^1(\kappa_\p)\}$ are the coset representatives for $T_\p$ given in $\eqref{cosets-decompositions}$,
       \item [$\bfcdot$] and the equality
\begin{equation}\label{decomposition domain of integration}
 \cal{X}_{\p^m}^\circ=\coprod_{a\in k_\p}\widehat{\sigma}_{m,a}\big(\cal{X}_{\p^{m}}^\circ\big)\qquad \text{for}\hspace{2mm} m\ge1,
\end{equation}
where $\{\widehat{\sigma}_{m,a}\ \lvert \ \bar{a}\in \kappa_\p\}$ are the coset representatives for ${_{m}}U_\varpi$ given in $\eqref{cosets-decompositions}$. 
 \end{itemize}
\end{proof}

Consider the $K_m^\circ$-equivariant function 
\[
\mrm{ev}_{m}\colon\cal{M}_{\Sigma,m}^{\Box}\longrightarrow (J_\Sigma^{\Box})^{I_m},\qquad \mu\mapsto \mu({\cal{X}_{m}}).
\]

\begin{lemma}\label{specialization - intertwining}
    For $m\ge1$ the morphism
    \[
    (\mrm{ev}_{m})_*\colon \cal{A}_{w,\Sigma}\big(K_{m}^\circ;\cal{H}_w^{\Box},\cal{M}_{\Sigma,m}^{\Box}\big)\longrightarrow \cal{A}_{w,\Sigma}\big(K_{m}^\circ;\cal{H}_w^{\Box},(J_\Sigma^{\Box})^{I_m}\big)
    \]
    commutes with the Hecke operators away from $\Sigma$. Moreover, for every $\p\in\Sigma$ and every uniformizer $\varpi\in\cal{O}_\p$ it intertwines ${_{m}}U_\varpi$ with ${_{m}}U_{\varpi}$.
\end{lemma}
\begin{proof}
   Reasoning as in the proof of Proposition \ref{restriction - intertwining}, the claim reduces to the equality
    \[
\widehat{\sigma}_{m,a}^{-1}\big(\cal{X}_{\p^m}\big)\cap\cal{X}_{\p^m}^\circ=\cal{X}_{\p^m}\qquad\forall\ a\in\kappa_\p
    \]
    where $\{\widehat{\sigma}_{m,a}\ \lvert \ \bar{a}\in \kappa_\p\}$ are the coset representatives for ${_{m}}U_\varpi$ given in $\eqref{cosets-decompositions}$.
\end{proof}

\subsection{Upshot}
\begin{corollary}\label{distributions-projlim}
    The  isomorphism of $\bb{T}_\Sigma$-modules
    \[
    \mrm{sp}_\infty\colon \cal{A}_{w,\Sigma}\big(K_{0};\cal{H}_w^{\Box},\cal{M}^{\Box}_\Sigma\big)\overset{\sim}{\longrightarrow} \varprojlim_m\ \cal{A}_{w,\Sigma}\big(K_{m}^\circ;\cal{H}_w^{\Box},(J_\Sigma^{\Box})^{I_m}\big)
    \]
    commutes with the Hecke operators away from $\Sigma$. Moreover, for every $\p\in\Sigma$ and uniformizer $\varpi\in\cal{O}_\p$ it intertwines $T_\varpi$ with $U_{\varpi}$.
\end{corollary}
\begin{proof}
Since $\mrm{sp}_m=\mrm{ev}_m\circ\mrm{res}_m^0$, Proposition \ref{restriction - intertwining} and Lemma \ref{specialization - intertwining} imply that 
for $m\ge1$ the specialization map
    \[
    (\mrm{sp}_m)_*\colon \cal{A}_{w,\Sigma}\big(K_{0};\cal{H}_w^{\Box},\cal{M}^{\Box}_\Sigma\big)\longrightarrow \cal{A}_{w,\Sigma}\big(K_{m}^\circ;\cal{H}_w^{\Box}, (J_\Sigma^{\Box})^{I_m}\big)
    \]
   commutes with the Hecke operators away from $\Sigma$. Moreover, for every $\p\in\Sigma$ and uniformizer $\varpi\in\cal{O}_\p$ it intertwines $T_\varpi$ with ${_{m}}U_{\varpi}$. The claim then follows from Proposition \ref{lemma: iso-specialization} and Corollary \ref{Hecke-action on limit}.
\end{proof}

\begin{corollary}\label{Hecke action on special form}
    We have 
    \[
    T_\q.\Phi_\Sigma^\Box=\Phi_\Sigma^\Box.\mbf{T}_\q,\qquad U_\q.\Phi_\Sigma^\Box=\Phi_\Sigma^\Box.\mbf{U}_\q,\qquad T_\varpi.\Phi_\Sigma^\Box=\Phi_\Sigma^\Box.\mbf{U}_{\varpi}.
    \]
\end{corollary}
\begin{proof}
    It follows from Corollary \ref{distributions-projlim}, Corollary \ref{Hecke-action on limit}, and equation \eqref{explict Hecke action}.
\end{proof}


\section{Construction}\label{Section: construction}
As in Section \ref{Intro: pHc}, we have fixed an incoherent set $\cal{S}$ for the triple $(A_{/F},\hspace{.5mm} E/F,\hspace{.5mm}p)$, a coherent set  $S=\cal{S}\cup\{w\}$ for the quadruple $(A_{/F},\hspace{.5mm} E/F,\hspace{.5mm}p,\hspace{.5mm}\cal{S})$, and an embedding $\psi\colon T\hookrightarrow G$ of $F$-algebraic groups. 
 For every $v\in S$, we consider the $G_v$-space $\cal{H}_v=\bb{P}^1(E_v)\setminus\bb{P}^1(F_v)$ with twisted $G_v$-action (see \eqref{twisting character}).
We are interested in the $G(F)$-module $\Z_p\big[\cal{H}_S^\Box\big]$ where 
\[
\cal{H}_{S}^\Box:=\cal{H}_\cal{S}\times \cal{H}_w^\Box,\qquad \cal{H}_\cal{S}:=\prod_{\p\in\cal{S}}\cal{H}_{\p},
\]
and denote elements of $\cal{H}_S^\Box$ by $\tau_{S}=(\tau_{\cal{S}},\tau_w)$. For any subset $\Sigma\subseteq S$, consider the homomorphism 
\begin{equation}\label{def: twist Sigma character}
\chi_\Sigma\colon G_\Sigma\longrightarrow G(E_\Sigma/F_\Sigma),\qquad \chi_\Sigma(g)=(\chi_v(g))_{v\in \Sigma},
\end{equation}
where $G(E_\Sigma/F_\Sigma)=\prod_{v\in \Sigma}\mrm{Gal}(E_v/F_v)$.
\begin{remark}
   It is useful to have an explicit notation for the twisted action on $\cal{H}_\Sigma$. We write
   \[
   g.\tau=g(\tau)^{\chi_\Sigma(g)}\qquad\forall\ g\in G_\Sigma,\ \tau\in\cal{H}_\Sigma,
   \] 
   where $g(\tau)=\frac{a\tau+b}{c\tau+d}$ if $g=\begin{pmatrix}
       a&b\\c&d
   \end{pmatrix}$.
\end{remark}

The main contribution of this section is the construction an element 
\[
\varphi^\Box_A\hspace{1mm}\in\hspace{1mm} \mrm{H}^0\left(X_G^{S}(\frak{f});\ \Z_p[\cal{H}_{S}^\Box],\ A(E_{\otimes,\cal{S}})\otimes_{\Q_p}\mrm{H}^1(E^\Box, V_p(A))\right)_{\pi_A}
\]
using the automorphic function $\Phi_\cal{S}^\Box\in \cal{A}_{w,\cal{S}}(K_0;\cal{H}_w^\Box,\cal{M}^\Box_\cal{S})$. Here the subscript $\pi_A$ serves to denote the $\pi_A$-isotypic subspace with respect to the ``good'' Hecke operators.

\subsection{Data associated to points of $p$-adic symmetric spaces}\label{Associated data}
  Fix a  uniformizer $\varpi_\p\in \cal{O}_\p$ for every $\p\in\cal{S}$. Cartan's decomposition
  \[
  G_\p=\coprod_{m\in\bb{N}}K_{\p^0}\cdot \eta_\varpi^m\cdot K_{\p^0}
  \]
  corresponds to the partition of the vertices $\cal{V}_\p$ of the Bruhat-Tits tree $\cal{T}_\frak{p}$ for $G_\p$ obtained by organizing them in terms of their distance from the standard vertex $\mrm{v}_\circ=[\cal{O}_\p^2]$.
  Consider a point $\tau\in\cal{H}_\p$. The associated vertex $\mrm{v}=\mrm{red}(\tau)\in\cal{V}_\p$ corresponds to an element
\[
\widetilde{g}_{\mrm{v}}\in \Big(K_{\p^0}\cdot \eta_\varpi^{m_{\mrm{v}}}\cdot K_{\p^0}\Big)/K_{\p^0}
\]
where $m_{\mrm{v}}=\mrm{dist}(\mrm{v},\mrm{v}_\circ)$.
 By a small abuse of notation, we also write $\widetilde{g}_{\mrm{v}}\in K_{\p^0}\cdot \eta_\varpi^{m_{\mrm{v}}}\cdot K_{\p^0}$ for a lift  and denote by $g_{\mrm{v}}$ its image in $G_\p$. Define  
\[
\bb{X}_{\mrm{v}}:=\widetilde{g}_{\mrm{v}}\cdot  \cal{X}_{\p^{0}},
\]
a compact open subset of $\widetilde{\cal{Y}}_\p$ (see Lemma \ref{Ash-Stevens} for the definition of $\widetilde{\cal{Y}}_\p$).  
\begin{remark}
    Following \eqref{dependence on uniformizer of the action}, for any double-coset element $x\in K_{\p^0}\cdot \eta_\varpi^m\cdot K_{\p^0}$ one finds that
    \begin{equation}\label{rmk: different acts}
    x\cdot_{\varpi}V=x\cdot_{a\varpi} V\bfcdot\hspace{1mm} a^m.
    \end{equation}
    As $\cal{X}_{\p^{0}}$ is stable under the action of $\cal{O}_\p^\times$, the set $\bb{X}_{\mrm{v}}$ does not depend on the choice of $\varpi\in \p\cal{O}_\p$.
\end{remark}
It is convenient to describe $\bb{X}_{\mrm{v}}$ in terms of a lattice in $F_\p^2$. If  $L_{\mrm{v}}:=g_{\mrm{v}}\cdot \cal{O}_\p^2$, and $L_{\mrm{v}}'=L_{\mrm{v}}\setminus \p L_{\mrm{v}}$,  the identification \eqref{iso2} gives
    \begin{equation}\label{domain in terms of lattice}
    \bb{X}_{\mrm{v}}\overset{\sim}{\longrightarrow}\cal{O}_\p^\times\backslash\big[L_{\mrm{v}}'\times\cal{O}_\p^\times\big].
    \end{equation}

\begin{definition}\label{Domain of integration}{(Domain of integration)}
    For $\tau_\cal{S}=(\tau_\p)_\p\in \cal{H}_\cal{S}$ let $\mrm{v}_\cal{S}=(\mrm{red}(\tau_\p))_\p$ the corresponding tuple of vertices. We define a compact open subset of $\widetilde{\cal{Y}}=\prod_{\p\in\cal{S}}\widetilde{\cal{Y}}_\p$ by setting
    \[
\bb{X}_{\mrm{v}_\cal{S}}:=\prod_{\p\in\cal{S}}\bb{X}_{\mrm{v}_\p}.
    \]
\end{definition}

Recall the identification \eqref{iso2} 
    \[
\cal{Y}_\p\overset{\sim}{\longrightarrow}F_\p^\times\backslash\Big[(F_\p^2)'\times F_\p^\times\Big],\qquad \gamma\mapsto \left[\begin{pmatrix}
        x_\gamma\\y_\gamma
    \end{pmatrix},z_\gamma\right]
\]
where $(F_\p^2)'=F_\p^2\setminus \{(0,0)\}$. Then, for any point  $\tau\in\cal{H}_\p$ we consider the continuous function $f_{\tau,\varpi}\colon \cal{Y}_\p\to E_\p$ described by the formula
\[
f_{\tau,\varpi}(\gamma)=\frac{1}{2}\log_{\varpi}\left(\frac{(x_\gamma-\tau y_\gamma)^2}{z_\gamma}\right)
\]
where $\log_{\varpi}$ denotes the $p$-adic logarithm which satisfies $\log_{\varpi}(\varpi)=0$. 
\begin{remark}\label{rmk: log as derivative}
  Any $x\in E_\p^\times$ can be written uniquely as
   \[
   x=\varpi^{\mrm{ord}_\p(x)}\cdot \zeta_x\cdot \langle x\rangle_{\varpi}
   \]
   where $\zeta_x$ is a root of unity and $\langle x\rangle_{\varpi}\equiv 1\pmod{\varpi}$. Thus,
   we can write  $\log_{\varpi}(x)=\log_p (\langle x\rangle_{\varpi})$ for any $x\in E_\p^\times$.
  For later purposes, it is useful to note that the function
   \[
   F_{\tau,\varpi}\colon\cal{Y}_\p\longrightarrow E_\p\llbracket X_\p\rrbracket,\qquad     \gamma\mapsto \frac{1}{2}\left\langle \frac{(x_\gamma-\tau y_\gamma)^2}{z_\gamma}\right\rangle_\varpi^{X_\p}
   \]
  satisfies (see Remark \ref{explict-right-action})
   \[
    F_{\tau,\varpi}(-\hspace{0.5mm}\bfcdot\hspace{1mm} a_\p^{-1})=a_\p^{X_\p}\cdot F_{\tau,\varpi}(-)\qquad\forall\ a_\p\in\cal{O}_\p^\times.
   \]
Moreover, the function $f_{\tau,\varpi}$ can be expressed as
    \[
    f_{\tau,\varpi}=\partial_{\p}F_{\tau,\varpi}\big\lvert_{0}.
    \]
\end{remark}

\begin{definition}\label{Integrand function}{(Integrand function)}
    For any $\tau_\cal{S}=(\tau_\p)_\p\in \cal{H}_\cal{S}$ we define the continuous function
    \[
    f_{\tau_\cal{S},\varpi}\colon \cal{Y}_\cal{S}\longrightarrow E_{\otimes,\cal{S}},\qquad \gamma=(\gamma_\p)_\p\mapsto\otimes_{\p\in\cal{S}}f_{\tau_\p,\varpi}(\gamma_\p).
    \]
\end{definition}
\begin{remark}\label{rmk: log as derivative II}
Consider the finite dimensional $\Q_p$-algebra
\[
R_{\otimes,\cal{S}}:=\frac{E_{\otimes,\cal{S}}\llbracket\underline{X}_\cal{S}\rrbracket}{(X_\p^2\ \lvert\ \p\in\cal{S})}.
\]
    By considering the inclusions $E_\p\llbracket X_\p\rrbracket\hookrightarrow E_{\otimes,\cal{S}}\llbracket \underline{X}_\cal{S}\rrbracket$ we can form the function
    \[
   F_{\tau_\cal{S},\varpi}\colon\cal{Y}_\cal{S}\longrightarrow R_{\otimes,\cal{S}},\qquad     \gamma\mapsto \prod_{\p\in\cal{S}}F_{\tau_\p,\varpi}(\gamma_\p).
   \]
   which satisfies
    \[
    F_{\tau_\cal{S},\varpi}(-\hspace{0.5mm}\bfcdot\hspace{1mm} a^{-1})=a^{\underline{X}_\cal{S}}\cdot F_{\tau_\cal{S},\varpi}(-)\qquad\forall\ a\in\cal{O}_\cal{S}^\times.
   \]
   and, setting $\partial_\cal{S}=\prod_{\p\in\cal{S}}\partial_\p$,
   \[
   f_{\tau_\cal{S},\varpi}=\partial_{\cal{S}}F_{\tau_\cal{S},\varpi}\big\lvert_{\underline{0}}.
   \]
\end{remark}

For $\p\in\cal{S}$ we denote by  $\cal{Y}^\p$ the product $\prod_{\q\in\cal{S}\setminus\{\p\}}\cal{Y}_\q$ and by $\gamma\mapsto\gamma^\p$ projection map $\cal{Y}_\cal{S}\to\cal{Y}^\p$. For later application we highlight the following property: 
\begin{lemma}\label{constant components}
    Let $\lambda_\cal{S}\in\cal{H}_\cal{S}$, $\delta\in G(F)$, and $\widetilde{g},\widetilde{h}\in \Delta^\varpi_\cal{S}$ such that $g=\delta h$ in $G_\cal{S}$. Choose an ordering $\{\p_1,\dots,\p_{\lvert\cal{S}\rvert}\}$ of $\cal{S}$, then the difference 
    \[
    f_{\delta(\lambda_\cal{S}),\varpi}\big(\widetilde{g}\cdot \gamma\big)-f_{\lambda_\cal{S},\varpi}\big(\widetilde{h}\cdot\gamma\big)
    \]
    can be written as 
    \[
    \sum_{j=1}^{\lvert\cal{S}\rvert}\Big(\underset{k<j}{\otimes}f_{\lambda_{\p_k},\varpi}\big(\widetilde{h}_{\p_k}\cdot \gamma_{\p_k}\big)\Big)
    \otimes C_j \otimes\Big(\underset{k>j}{\otimes}f_{\delta(\lambda_{\p_k}),\varpi}\big(\widetilde{g}_{\p_k}\cdot \gamma_{\p_k}\big)\Big)
    \]
    where $C_j=C_j\big(\delta, \varpi_{\p_j}, \lambda_{\p_j}, \widetilde{g}_{\p_j},\widetilde{h}_{\p_{j}}\big)$ is a constant.
\end{lemma}
\begin{proof}
For notational convenience we denote the action of $\Delta_\cal{S}^\varpi$ on functions by $(f\lvert \widetilde{g})(-):=f(\widetilde{g}-)$. Through a telescopic argument the claim reduces to checking that for every $\p\in\cal{S}$ the difference
    \[
    \log_{\varpi}\left(\frac{(x-\delta(\lambda) y)^2}{z}\right)\bigg\lvert \widetilde{g}-\log_\varpi\left(\frac{(x-\lambda y)^2}{z}\right)\bigg\lvert \widetilde{h}
    \]
    is a constant independent of $(x,y)\in (F_\p^2)'$, $z\in F_\p^\times$. Using the explicit description in \eqref{explicit-action-semigroup-in-coordinates}, we find that
    \begin{equation}\label{eq: formula action for function}
        \big(x-\delta(\lambda) y\big)\big\lvert \widetilde{g}=\begin{pmatrix}
        1&-\delta(\lambda)
\end{pmatrix}g\begin{pmatrix}
        x\\y
    \end{pmatrix}\qquad\text{and}\qquad z\lvert \widetilde{g}=\frac{\det(g)}{\varpi^{\mrm{deg}(\widetilde{g})}}\cdot z.
    \end{equation}
   Moreover, the equality
    \[
    \begin{pmatrix}
        1&-\delta(\lambda)
\end{pmatrix}=\frac{1}{(c_\delta\lambda+d_\delta)}\begin{pmatrix}
    1&-\lambda
\end{pmatrix}\mrm{adj}(\delta),\qquad\text{for}\qquad
    \mrm{adj}(\delta)\delta=\det(\delta)\begin{pmatrix}
    1&0\\0&1
\end{pmatrix},
    \]
    together with $g=\delta h$ implies that
    \[\big(x-\delta(\lambda) y\big)\big\lvert \widetilde{g}=\frac{\det(\delta)}{(c_\delta\lambda+d_\delta)}\big(x-\lambda y\big)\big\lvert h.
    \]
    We conclude that
    \[
     \left(\frac{(x-\delta(\lambda) y)^2}{z}\right)\bigg\lvert \widetilde{g}=\frac{\det(\delta)}{(c_\delta\lambda+d_\delta)^2}\cdot\varpi^{\mrm{deg}(\widetilde{g})-\mrm{deg}(\widetilde{h})}\cdot \left(\frac{(x-\lambda y)^2}{z}\right)\bigg\lvert \widetilde{h}.
    \]
   The claim follows by taking the logarithm of both sides of the equality.
\end{proof}


 \begin{definition}\label{Distribution}{(Measure)}
    Consider  $\tau_{S}=(\tau_{\cal{S}},\tau_w)\in \cal{H}_S^\Box$, $h\in G(\A^{S,\infty})/(K_0)^S$. We define the associated measure by
\[
\mu^{\varpi,\Box}_{(\tau_{S},h)}:= \Big(\widetilde{g}_{\mrm{v}_\cal{S}}.\Phi_\cal{S}^\Box(g_{\mrm{v}_\cal{S}}h)(\tau_w)\Big)\Big[\chi(g_{\mrm{v}_{\cal{S}}})\Big]\ \in\ (\widetilde{\cal{M}}^\Box_{\cal{S}})_{\overline{\varrho}}\Big[G(E_\cal{S}/F_\cal{S})\Big]
\]
where $\widetilde{g}_{\mrm{v}_\cal{S}}\in \Delta_\cal{S}^\varpi$ denotes the element   obtained from $\mrm{v}_{\cal{S}}=\mrm{red}(\tau_\cal{S})$ as described at the beginning of Section \ref{Associated data}, and $g_{\mrm{v}_\cal{S}}\in G_\cal{S}$ is the image of $\widetilde{g}_{\mrm{v}_\cal{S}}$ in $G_\cal{S}$.
\end{definition}
Equation \eqref{def: action distributions} shows that $\mu^{\varpi,\Box}_{(\tau_{S},h)}$ is supported on 
\begin{equation}\label{eq: support}
\bb{X}_{\mrm{v}_\cal{S}}\overset{\sim}{\longrightarrow} \prod_{\p\in\cal{S}}\cal{O}_\p^\times\backslash\big[L_{\mrm{v}_\p}'\times\cal{O}_\p^\times\big].
\end{equation}
Moreover,  the property \eqref{def: equiv prop} and the equality \eqref{rmk: different acts} explain the dependence on the choice of uniformizers:
\begin{equation}\label{measure: dep on unif}
\mu^{a\varpi,\Box}_{(\tau_{S},h)}=\mu^{\varpi,\Box}_{(\tau_{S},h)}\cdot \llangle a\rrangle^{m_{\mrm{v}_\cal{S}}}.
\end{equation}
where $\llangle a\rrangle^{m_{\mrm{v}_\cal{S}}}:=\prod_{\p\in\cal{S}}\llangle a_\p\rrangle^{m_{\mrm{v}_\p}}$ is a product of diamond operators.
 \begin{remark}
   The notation $\mu^{\varpi,\Box}_{(\tau_{S},h)}$ is convenient but slightly misleading. The measure depends on the choice of $\{\varpi_\p\}_{\p\in\cal{S}}$, the vertices $\mrm{v}_{\cal{S}}=\mrm{red}(\tau_\cal{S})$, the point $\tau_w$, and the element $h$. 
 \end{remark}

 Let $N$ be a finite dimensional $\Q_p$-vector space endowed with an action of $G(E_\cal{S}/F_\cal{S})$. We define integration functional 
 \[
 \int_{\bb{X}_{\mrm{v}_\cal{S}}}(-)\hspace{1mm}\mrm{d}\mu^{\varpi,\Box}_{(\tau_{S},h)}\colon \cal{C}(\bb{X}_{\mrm{v}_\cal{S}},N)\longrightarrow N\otimes_{\Z_p} (J^\Box_\cal{S})^\wedge_{\overline{\varrho}}
 \]
 by setting
 \[
\int_{\bb{X}_{\mrm{v}_\cal{S}}}f(\gamma)\hspace{1mm}\mrm{d}\mu^{\varpi,\Box}_{(\tau_{S},h)}(\gamma):=\int_{\bb{X}_{\mrm{v}_\cal{S}}}f(\gamma)^{\chi(g_{\mrm{v}_{\cal{S}}})}\hspace{1mm}\mrm{d}\Big(\widetilde{g}_{\mrm{v}_\cal{S}}.\Phi_\cal{S}^\Box(g_{\mrm{v}_\cal{S}}h)(\tau_w)\Big)(\gamma)
 \]
for any $f\in \scr{C}(\bb{X}_{\mrm{v}_\cal{S}},N)$.

\subsection{Plectic Abel--Jacobi map}\label{Enhanced modular parametrization}
 We combine the objects introduced in the previous section to define the function
\[
\varphi^\Box_\varpi\colon G(\A^{S,\infty})/K_0^S\longrightarrow \mrm{Hom}_{\Z_p}\big(\Z_p[\cal{H}_{S}^\Box],\ E_{\otimes,\cal{S}}\otimes_{\Z_p}(J^\Box_\cal{S})^\wedge_{\overline{\varrho}}\big)
\]
given  by 
\begin{equation}\label{parametrization}
\varphi^\Box_\varpi(h)(\tau_{S})=\int_{\bb{X}_{\mrm{v}_\cal{S}}}f_{\tau_\cal{S},\varpi}(\gamma)\hspace{1mm}\mrm{d}\mu^{\varpi,\Box}_{(\tau_{S},h)}(\gamma).
\end{equation}
Unwinding the definitions one sees that 
\begin{equation}\label{eq: unwinding}
\varphi^\Box_\varpi(h)(\tau_{S})=\int_{\cal{X}_\cal{S}}\hspace{1.5mm}f_{\tau_\cal{S}^{\chi(g_{\mrm{v}})},\varpi}(\widetilde{g}_{\mrm{v}_\cal{S}}\cdot \gamma)\hspace{1.5mm}\mrm{d}\big(\Phi^\Box_\cal{S}(g_{\mrm{v}_\cal{S}}h)(\tau_w)\big)(\gamma)
\end{equation}
where we shortened $\tau_\cal{S}^{\chi(g_{\mrm{v}})}:=\tau_\cal{S}^{\chi_{\cal{S}}(g_{\mrm{v}_{\cal{S}}})}$.

\begin{remark}\label{rmk: exotic AJ as a derivative}
    It is useful to think of $\varphi^\Box_\varpi$ as a derivative. Remarks \ref{rmk: log as derivative} $\&$ \ref{rmk: log as derivative II} imply that
    \begin{equation}\label{def: in terms of derivatives}
    \varphi^\Box_\varpi(h)(\tau_S)=\partial_\cal{S}\left(\int_{\bb{X}_{\mrm{v}_\cal{S}}}F_{\tau_\cal{S},\varpi}(\gamma)\hspace{1mm}\mrm{d}\mu^{\varpi,\Box}_{(\tau_S,h)}(\gamma)\right)\big\lvert_{\underline{0}}.
    \end{equation}
     Furthermore, let us write  $F_{\tau_\cal{S},\varpi}(\gamma)$ in the standard basis $F_{\tau_\cal{S},\varpi}(\gamma)=\sum_{\Sigma\subseteq\cal{S}}F^\Sigma_{\tau_\cal{S},\varpi}(\gamma)\cdot \underline{X}_\Sigma$ where $F^\Sigma_{\tau_\cal{S},\varpi}:=\partial_\Sigma F_{\tau_\cal{S},\varpi}\lvert_{\underline{0}}\hspace{1mm}\in\cal{C}(\cal{Y},E_{\otimes,\cal{S}})$ for every $\Sigma\subseteq\cal{S}$. In particular, $F^\cal{S}_{\tau_\cal{S},\varpi}=f_{\tau_\cal{S},\varpi}$. Then, the integral 
\begin{equation}\label{integral with parameters}
\int_{\bb{X}_{\mrm{v}_\cal{S}}}F_{\tau_\cal{S},\varpi}(\gamma)\hspace{1mm}\mrm{d}\mu^{\varpi,\Box}_{(\tau_S,h)}(\gamma)
\ \in\ R_{\otimes,\cal{S}}\otimes_{\Z_p}(J_\infty^\Box)^\wedge,
\end{equation}
where $R_{\otimes,\cal{S}}=\frac{E_{\otimes,\cal{S}}\llbracket\underline{X}_\cal{S}\rrbracket}{(X_\p^2\ \lvert\ \p\in\cal{S})}$, appearing in the RHS of \eqref{def: in terms of derivatives}, becomes
\begin{equation}\label{explicit integral in coordinates}
\int_{\bb{X}_{\mrm{v}_\cal{S}}}F_{\tau_\cal{S},\varpi}(\gamma)\hspace{1mm}\mrm{d}\mu^{\varpi,\Box}_{(\tau_S,h)}(\gamma)=\sum_{\Sigma\subseteq\cal{S}}\left(\int_{\bb{X}_{\mrm{v}_\cal{S}}}F^\Sigma_{\tau_\cal{S},\varpi}(\gamma)\hspace{1mm}\mrm{d}\mu^{\varpi,\Box}_{(\tau_S,h)}(\gamma)\right)\cdot\underline{X}_\Sigma.
\end{equation}
From this point of view, it might not be too surprising that the function $\varphi^\Box_\varpi$ is well-behaved only when it corresponds to the leading term of \eqref{integral with parameters}.
\end{remark}
 \subsubsection{Invariance under diamond operators.}\label{subsub: Invariance under diamond operators} To ensure that the image of $\varphi^\Box_\varpi$ consists of elements fixed by the action of diamond operators, we localize the target of $\varphi^\Box_\varpi$ at  $\cal{P}\in\mrm{Spec}(\bb{T}_{\overline{\varrho}})$, the kernel of  the homomorphism $f_A\colon\bb{T}_{\overline{\varrho}}\twoheadrightarrow \Z_p$ associated to the modular elliptic curve $A_{/F}$:
 
 we let $\varphi^\Box_\varpi\otimes1$ be the function
  \begin{equation}\label{varphi localized}
  (\varphi^\Box_\varpi\otimes1)(h)(\tau_S):=\left(\int_{\bb{X}_{\mrm{v}_\cal{S}}}f_{\tau_\cal{S},_\varpi}(\gamma)\hspace{1mm}\mrm{d}\mu^{\varpi,\Box}_{(\tau_S,h)}(\gamma)\right)\otimes 1
  \end{equation}
 taking values in
 \[
 \Big(E_{\otimes,\cal{S}}\otimes_{\Z_p}(J^\Box_\cal{S})^\wedge_{\overline{\varrho}}\Big)\otimes_{\bb{T}_{\overline{\varrho}}}\bb{I},
 \]
 where $\bb{I}$ denotes the localization of $\bb{T}_{\overline{\varrho}}$ at $\cal{P}$.
\begin{remark}\label{rmk: stress on when localize}
    It is important to localize after evaluating the integral, otherwise we would loose too much information. Later on, it will become clear that the function $\varphi^\Box_\varpi$ is already well-behaved when the local residual representations $\big\{\overline{\varrho}_{\lvert D_\p}\big\}_{\p\in\cal{S}}$ are not finite.
\end{remark}

The next lemma will allow us to invoke the results on descent proved in Section \ref{Descent of Galois cohomology classes}.
\begin{lemma}\label{identify-inv}
  Let $M$ be a $\bb{T}_{\overline{\varrho}}$-module, $\bb{T}_{\overline{\varrho}}\to B$ a flat ring map, then the natural inclusion
  \[
  M^{D_\Sigma}\otimes_{\bb{T}_{\overline{\varrho}}}B\overset{\sim}{\longrightarrow}(M\otimes_{\bb{T}_{\overline{\varrho}}}B)^{D_\Sigma}
  \]
  is an isomorphism.
\end{lemma}
\begin{proof}
Let $\cal{I}_\Sigma=\ker(\Lambda_\Sigma\to\Z_p)$ and write $\bb{T}=\bb{T}_{\overline{\varrho}}$ to simplify the notation. Note that
 \[
    M^{D_\Sigma}=\mrm{Hom}_{\bb{T}}\big(\bb{T}/\cal{I}_\Sigma\bb{T}, M\big),\qquad  (M\otimes_\bb{T}B)^{D_\Sigma}=\mrm{Hom}_{B}\big(B/\cal{I}_\Sigma B, M\otimes_\bb{T}B\big),
    \]
    and that $\bb{T}$ is Noetherian. Applying the left exact functors 
    \[
    \mrm{Hom}_{\bb{T}}(-,M)\otimes_{\bb{T}}B,\qquad \text{and}\qquad \mrm{Hom}_{B}(-\otimes_{\bb{T}}B,M\otimes_{\bb{T}}B)
    \]
    to a finite presentation $\bb{T}^{\oplus t}\to \bb{T}\to \bb{T}/\cal{I}_\Sigma\bb{T}\to 0$ of the $\bb{T}$-module $\bb{T}/\cal{I}_\Sigma\bb{T}$, we obtain the following commuting diagram 
    \[\xymatrix{
    0\ar[r] & M^{D_\Sigma}\otimes_\bb{T}B\ar[r]\ar[d] &M\otimes_\bb{T}B\ar[r]\ar[d]^-\sim &M^{\oplus t}\otimes_\bb{T}B\ar[d]^-\sim\\
    0\ar[r] & (M\otimes_\bb{T}B)^{D_\Sigma}\ar[r] &M\otimes_\bb{T}B\ar[r] &(M\otimes_\bb{T}B)^{\oplus t}\\
    }\]
    with exact rows and where the two rightmost vertical maps are isomorphisms. A simple diagram chase implies that the leftmost vertical map is also an isomorphism.
\end{proof}

Let $h\in G(\A^{S,\infty})$, $\tau_S\in \cal{H}_S^\Box$.  For any subset $\Sigma\subseteq\cal{S}$ and any $\lambda_\Sigma\in\cal{H}_\Sigma$, set $f_{\lambda_\Sigma,\varpi}=\otimes_{\q\in\Sigma}f_{\lambda_\q,\varpi}$. We consider sets $V_\cal{S}=\prod_{\p\in\cal{S}} V_\p$ where  $V_\p\subseteq \widetilde{\cal{Y}}_\p$ is a compact open subset which is stable under the action of $\cal{O}_\p^\times$.  The following is the key vanishing result of this article.  
\begin{theorem}\label{vanishing implies invariance}
    If $\Sigma\subset\cal{S}$ is a proper subset  and there exists $\q\in\Sigma^c$ such that $V_{\q}=\bb{X}_{\mrm{v}_{\q}}$, then 
    \[
    \left(\int_{V_\cal{S}}\mathbbm{1}_{\Sigma^c}\otimes f_{\lambda_\Sigma,\varpi}(\gamma)\hspace{1mm}\mrm{d}\mu^{\varpi,\Box}_{(\tau_S,h)}(\gamma)\right)\otimes1=0.
    \]
\end{theorem}
\begin{proof}
Let us begin by observing that we understand the action of diamond operators on  the measure $\mu^{\varpi,\Box}_{(\tau_S,h)}$ because it belongs to $ (\widetilde{\cal{M}}^\Box_\cal{S})_{\overline{\varrho}}$. Thus, for $\p\in\cal{S}$ and $a_\p\in\cal{O}_\p^\times$ we compute that
    \[
\left(\int_{V_\cal{S}}\mathbbm{1}_{\Sigma^c}\otimes f_{\lambda_\Sigma,\varpi}(\gamma)\hspace{1mm}\mrm{d}\mu^{\varpi,\Box}_{(\tau_S,h)}(\gamma)\right)\otimes \llangle a_\p\rrangle
   =\left(\int_{V_\cal{S}}\mathbbm{1}_{\Sigma^c}\otimes f_{\lambda_\Sigma,\varpi}(\gamma\hspace{.3mm}\bfcdot \hspace{.3mm}a_\p^{-1})\hspace{1mm}\mrm{d}\mu^{\varpi,\Box}_{(\tau_S,h)}(\gamma)\right)\otimes1.
    \]
In particular, $\cal{O}_{\Sigma^c}^\times$ acts trivially. Moreover, the equality $f_{\lambda_\p,\varpi}(\gamma_\p\hspace{.3mm}\bfcdot \hspace{.3mm}a_\p^{-1})=f_{\lambda_\p,\varpi}(\gamma_\p)+\frac{1}{2}\log_p(a_\p)$ allows us to deduce that the invariance under the $\cal{O}_\p^\times$-action, for $\p\in\Sigma$, depends on the vanishing of the integral of a function whose components in $\Sigma^c\cup\{\p\}$ are constant.

\smallskip
\noindent Note that unwinding the definitions as in \eqref{eq: unwinding}, and using the compatibility \eqref{inter-compatibilities} we obtain
\begin{equation}\label{eq: computation with compatibilities}
\begin{split}
\int_{V_\cal{S}}\mathbbm{1}_{\Sigma^c}\otimes f_{\lambda_\Sigma,\varpi}(\gamma)\hspace{1mm}&\mrm{d}\mu^{\varpi,\Box}_{(\tau_S,h)}(\gamma)\\
&=\int_{\cal{X}_{\q}\times V_{\q^c}^*}\hspace{1.5mm}\mathbbm{1}_{\Sigma^c}\otimes f_{\lambda_\Sigma^{\chi(g_{\mrm{v}})},\varpi}(\widetilde{g}_{\mrm{v}_\Sigma}\cdot \gamma)\hspace{1.5mm}\mrm{d}\big(\Phi^\Box_\cal{S}(g_{\mrm{v}_\cal{S}}h)(\tau_w)\big)(\gamma)\\
&=1_{\q}\otimes\int_{V_{\q^c}^*}\hspace{1.5mm} \mathbbm{1}_{\Sigma^c\setminus\{\q\}}\otimes f_{\lambda_\Sigma^{\chi(g_{\mrm{v}})},\varpi}(\widetilde{g}_{\mrm{v}_\Sigma}\cdot \gamma)\hspace{1.5mm}\mrm{d}\big(\iota_*\Phi^\Box_{\q^c}(g_{\mrm{v}_\cal{S}}h)(\tau_w)\big)(\gamma)
\end{split}\end{equation}
where  $\q^c:=\cal{S}\setminus\{\q\}$, $V_{\q^c}^*:=\widetilde{g}^{-1}_{\mrm{v}_{\q^c}}.V_{\q^c}$,  and  $\iota_*\Phi^\Box_{\q^c}=(\iota_{\cal{S}/\q^c})_*\Phi^\Box_{\q^c}$.

\smallskip
 The proof is now by induction on the cardinality of $\Sigma$.
Let $n\le\lvert\cal{S}\rvert$ be a positive integer and suppose the claimed vanishing holds for all $\Sigma'$ with $\lvert\Sigma'\rvert<n$. If $\lvert\Sigma\rvert=n$, then the quantity
\begin{equation}\label{eq: quantity to vanish}
\left(\int_{V_\cal{S}}\mathbbm{1}_{\Sigma^c}\otimes f_{\lambda_{\Sigma},\varpi}(\gamma)\hspace{1mm}\mrm{d}\mu^{\varpi,\Box}_{(\tau_S,h)}(\gamma)\right)\otimes1
\end{equation}
is fixed by all diamond operators. Hence, equation \eqref{eq: computation with compatibilities}, Lemma \ref{identify-inv}, and Corollary \ref{cor: descent} imply that \eqref{eq: quantity to vanish} belongs to the image of 
\begin{equation}\label{eq: prop vanishing}
E_{\otimes,\cal{S}}\otimes_{\Z_p}\mrm{H}^1\big(E^\Box,\hspace{.7mm} T_p(J_1^\circ(\q^c))_{\overline{\varrho}}\big)\otimes_{\bb{T}_{\overline{\varrho}}}\bb{I}\longrightarrow E_{\otimes,\cal{S}}\otimes_{\Z_p}\mrm{H}^1\big(E^\Box,\hspace{.7mm}  T_p(J_1^\circ(\cal{S}))_{\overline{\varrho}}\big)\otimes_{\bb{T}_{\overline{\varrho}}}\bb{I}.
\end{equation}
Finally, the domain of \eqref{eq: prop vanishing} is trivial because the elliptic curve $A_{/F}$ has multiplicative reduction at every prime in $\cal{S}$.
\end{proof}

We deduce the following corollaries:

\begin{corollary}\label{cor: exotic AJ as leading term}
    The following equality holds in $R_{\otimes,\cal{S}}\otimes_{\Z_p}(J^\Box_\cal{S})^\wedge_{\overline{\varrho}}\otimes_{\bb{T}_{\overline{\varrho}}}\bb{I}$
    \[\begin{split}
    \Bigg(\int_{ V_{\Sigma^c}\times  \bb{X}_{\mrm{v}_{\Sigma}}}&F_{\lambda_\cal{S},\varpi}(\gamma)\hspace{1mm}\mrm{d}\mu^{\varpi,\Box}_{(\tau_S,h)}(\gamma)\Bigg)\otimes1\\
    = &\sum_{J\supseteq\Sigma}\hspace{2.5mm}\frac{1}{2^{\lvert J^c\rvert}}\cdot \underline{X}_J\cdot \Bigg(\int_{ V_{\Sigma^c}\times  \bb{X}_{\mrm{v}_{\Sigma}}}\mathbbm{1}_{J^c}\otimes f_{\lambda_{J},\varpi}(\gamma)\hspace{1mm}\mrm{d}\mu^{\varpi,\Box}_{(\tau_S,h)}(\gamma)\Bigg)\otimes1.
    \end{split}\]
    In particular, we find that
    \[
 \left(\int_{\bb{X}_{\mrm{v}_\cal{S}}}F_{\tau_\cal{S},\varpi}(\gamma)\hspace{1mm}\mrm{d}\mu^{\varpi,\Box}_{(\tau_S,h)}(\gamma)\right)\otimes1=\underline{X}_\cal{S}\cdot (\varphi^\Box_\varpi\otimes1)(h)(\tau_S).
\]
\end{corollary}
\begin{proof}
     It follows from equation \eqref{explicit integral in coordinates} and Theorem \ref{vanishing implies invariance}.
\end{proof}

\begin{corollary}\label{inv-diamond}
The image of $\varphi^\Box_\varpi\otimes1$ is fixed by the action of diamond operators, i.e.,
    \[
 \llangle a\rrangle (\varphi^\Box_\varpi\otimes1)(-)(-)=  (\varphi^\Box_\varpi\otimes1)(-)(-)\qquad \forall\ a\in\cal{O}_\cal{S}^\times.
\]
\end{corollary}
\begin{proof}
   By Remark \ref{rmk: log as derivative II} we know that $F_{\tau_\cal{S},\varpi}(-\hspace{0.5mm}\bfcdot\hspace{1mm} a^{-1})=a^{\underline{X}_\cal{S}}\cdot F_{\tau_\cal{S},\varpi}(-)$ for all $a\in\cal{O}_\cal{S}^\times$, hence for $\mu=\mu^{\varpi,\Box}_{(\tau_S,h)}$ we have that
    \[
\bigg(\int_{\bb{X}_{\mrm{v}_\cal{S}}}F_{\tau_\cal{S},\varpi}(\gamma)\hspace{1mm}\mrm{d}\mu(\gamma)\bigg)\otimes \llangle a\rrangle=a^{\underline{X}_\cal{S}}\cdot\bigg(\int_{\bb{X}_{\mrm{v}_\cal{S}}}F_{\tau_\cal{S},\varpi}(\gamma)\hspace{1mm}\mrm{d} \mu(\gamma)\bigg)\otimes1.
   \]
   The claim then follows from Corollary \ref{cor: exotic AJ as leading term}.
\end{proof}

\begin{corollary}\label{independence of choice of uniformizers}
    The function $\varphi^\Box_\varpi\otimes1$ does not depend on the choice of uniformizers $\{\varpi_\p\}_{\p\in\cal{S}}$.
\end{corollary}
\begin{proof}
Fix $\p\in\cal{S}$. The dependence of  $f_{\tau,\varpi}$ on the choice of uniformizer  is expressed by 
\[
f_{\tau,a\varpi}(\gamma)=f_{\tau,\varpi}(\gamma)-\frac{\log_p(a)}{2\mrm{ord}_\p(\varpi)}\mrm{ord}_\p\left(\frac{(x_\gamma-\tau y_\gamma)^2}{z_\gamma}\right)\qquad \forall\ a\in\cal{O}_\p^\times.
\]
As  $\mrm{ord}_\p\left(\frac{(x_\gamma-\tau y_\gamma)^2}{z_\gamma}\right)$ is constant on $\bb{X}_\mrm{v}$ for $\mrm{v}=\mrm{red}(\tau)$, the claim follows from Theorem \ref{vanishing implies invariance}. The dependence of the measure on the choice of uniformizers is given in \eqref{measure: dep on unif}, thus the claim follows from Corollary \ref{inv-diamond}.
\end{proof}

Corollary \ref{inv-diamond} and Lemma \ref{identify-inv} imply that the function $\varphi^\Box_\varpi\otimes1$ takes value in 
\[
\Big(E_{\otimes,\cal{S}}\otimes_{\Z_p}(J^\Box_\cal{S})^\wedge_{\overline{\varrho}}\otimes_{\bb{T}_{\overline{\varrho}}}\bb{I}\Big)^{D_\cal{S}}\cong E_{\otimes,\cal{S}}\otimes_{\Z_p}\mrm{H}^0\big(D_\cal{S},\hspace{1mm} (J^\Box_\cal{S})^\wedge_{\overline{\varrho}}\big)\otimes_{\bb{T}_{\overline{\varrho}}}\bb{I}.
\]
 Furthermore, Corollary \ref{cor: descent} ensures the existence of a natural Hecke equivariant map
\begin{equation}\label{eq: Hecke equivariant map for A}
\mrm{pr}_A\colon \mrm{H}^0\big(D_\cal{S},\hspace{1mm} (J^\Box_\cal{S})^\wedge_{\overline{\varrho}}\big)\otimes_{\bb{T}_{\overline{\varrho}}}\bb{I}\longrightarrow  \mrm{H}^1(E^\Box, V_p(A)).
\end{equation}
For every $\p\in\cal{S}$ let $q_\p\in\p\cal{O}_\p$ denote Tate's parameter for $A_{/F_\p}$, and consider the isomorphism
\[
\exp^\mrm{Tate}_{q_\p}\colon E_\p\overset{\sim}{\longrightarrow} A(E_\p)^\wedge\otimes_{\Z_p}\Q_p,\qquad \exp^\mrm{Tate}_{q_\p}=\phi^\mrm{Tate}_\p\circ(\log_{q_\p})^{-1}
\]
obtained as the composition of the inverse of $\log_{q_\p}\colon\big(E_\p^\times/q_{\p}^\Z \big)^\wedge\otimes_{\Z_p}\Q_p\overset{\sim}{\longrightarrow} E_\p$ with Tate's uniformization $\phi^\mrm{Tate}_\p\colon \big(E_\p^\times/q_{\p}^\Z\big)^\wedge\otimes_{\Z_p}\Q_p \overset{\sim}{\longrightarrow} A(E_\p)^\wedge\otimes_{\Z_p}\Q_p$.

We are now ready to make the following definition:
\begin{definition}{(Plectic Abel--Jacobi map)}\label{def: Exotic AJ map}
Let
\[
\varphi^\Box_A\colon G(\A^{S,\infty})/(K_0)^S\longrightarrow \mrm{Hom}_{\Z_p}\big(\Z_p[\cal{H}_{S}^\Box],\ A(E_{\otimes,\cal{S}})\otimes_{\Q_p}\mrm{H}^1(E^\Box, V_p(A))\big)
\]
be the function defined by 
\[
\varphi^\Box_A:=(\exp_{q_\cal{S}}^\mrm{Tate}\otimes\hspace{0.3mm}\mrm{pr}_A)(\varphi^\Box_\varpi\otimes1)
\]
where  $A(E_{\otimes,\cal{S}}):=\otimes_{\p\in\cal{S}}\big(A(E_\p)^\wedge\otimes_{\Z_p}\Q_p\big)$ and $\exp_{q_\cal{S}}^\mrm{Tate}=\otimes_{\p\in\cal{S}}\exp_{q_\p}^\mrm{Tate}$.
\end{definition}
As we will become clear in Section \ref{Comparison with previous constructions}, the use of $\exp_{q_\cal{S}}^\mrm{Tate}$ is quite natural -- contrary to what might appear at this stage. Moreover, by Corollary \ref{independence of choice of uniformizers}, the function $\varphi^\Box_A$ does not depend on the choice of uniformizers $\{\varpi_\p\}_{\p\in\cal{S}}$.

\subsubsection{Invariance under the $G(F)$-action.}\label{subsub: Invariance under the G(F)-action}
The group $G(F)$ acts on the function $\varphi^\Box_A$ according to the formula
\[
(\delta^{-1}.\varphi^\Box_A)(-)(-):=\varphi^\Box_A(\delta^S\cdot -)(\delta_S. -)\qquad \forall\ \delta\in G(F).
\]

\begin{proposition}\label{prop:G(F)-invariance}
The function $\varphi^\Box_A$  is $G(F)$-invariant, i.e.,
    \[
   \varphi^\Box_A(\delta^S\cdot -)(\delta_S.-)=\varphi^\Box_A(-)(-)\qquad \forall\ \delta\in G(F).
    \]
\end{proposition}
\begin{proof}
Recall that in equation $\eqref{eq: unwinding}$ we expressed $\varphi^\Box_\varpi(h)(\tau_S)$ as
\[
\varphi^\Box_\varpi(h)(\tau_S)=\int_{\cal{X}_\cal{S}}\hspace{1.5mm}f_{\tau_\cal{S}^{\chi(g_\mrm{v})},\varpi}(\widetilde{g}_{\mrm{v}_\cal{S}}\cdot \gamma)\hspace{1.5mm}\mrm{d}\big(\Phi^\Box_\cal{S}(g_{\mrm{v}_\cal{S}}h)(\tau_w)\big)(\gamma).
\]
Since the reduction map satisfies $\mrm{red}\big(\delta.\tau_\cal{S}\big)=\delta \mrm{red}\big(\tau_{\cal{S}}\big)$, there exists an element $k_\cal{S}\in (K_0)_\cal{S}$ such that $g_{\delta_\cal{S}\mrm{v}_\cal{S}}\cdot k_\cal{S}=\delta_\cal{S}\cdot  g_{\mrm{v}_\cal{S}}$. A direct computation using the properties of $\Phi^\Box_\cal{S}$ shows that
\[\begin{split}
 \mu^{\varpi,\Box}_{\delta(\tau_S,h)}&=\Big(\widetilde{g}_{\delta_\cal{S}\mrm{v}_\cal{S}}.\Phi_\cal{S}^\Box( g_{\delta_\cal{S}\mrm{v}_\cal{S}}\cdot \delta^Sh)(\delta_w.\tau_w)\Big)\Big[\chi(g_{\delta_\cal{S}\mrm{v}_{\cal{S}}})\Big]\\
 &= \Big(\widetilde{g}_{\delta_\cal{S}\mrm{v}_\cal{S}}.\Phi_\cal{S}^\Box( \delta^wg_{\mrm{v}_\cal{S}}k_\cal{S}^{-1}\cdot h)(\delta_w.\tau_w)\Big)\Big[\chi(\delta_\cal{S})\cdot \chi(g_{\mrm{v}_{\cal{S}}})\Big]\\
 &= \Big((\widetilde{g}_{\delta_\cal{S}\mrm{v}_\cal{S}}\cdot k_\cal{S}).\Phi_\cal{S}^\Box( g_{\mrm{v}_\cal{S}}h)(\tau_w)\Big)\Big[\chi(\delta_\cal{S})\cdot \chi(g_{\mrm{v}_{\cal{S}}})\Big],
\end{split}\]
hence
\[
\varphi^\Box_\varpi(\delta^Sh)(\delta_S.\tau_S)=\int_{\cal{X}_\cal{S}}\hspace{1.5mm}f_{\delta_\cal{S}(\tau_\cal{S}^{ \chi(g_{\mrm{v}})}),\varpi}\big((\widetilde{g}_{\delta_\cal{S}\mrm{v}_\cal{S}}\cdot k_\cal{S}\big)\cdot\gamma\big)\hspace{1.5mm}\mrm{d}\big(\Phi_\cal{S}^\Box(g_{\mrm{v}_\cal{S}}h)(\tau_w)\big)(\gamma).
\]
Then, Lemma \ref{constant components} together with Theorem \ref{vanishing implies invariance}  imply that
\[
\varphi^\Box_A(\delta^Sh)(\delta_S.\tau_S)=\varphi^\Box_A(h)(\tau_S).
\]
\end{proof}

\subsubsection{Upshot.}
\vspace{3mm}
For a $\Z_p$-module $N$ and an $\Z_p[G_\cal{S}]$-module $M$,
there is a $G(F)$-action on the $R$-module of continuous functions $\cal{C}\big(G(\A^{S,\infty})/K_0^S,\Hom_{\Z_p}(M,N)\big)$ given by
\[
(\gamma.\Phi)(g)(m)=\Phi(\gamma^{-1}g)(\gamma^{-1}m),
\]
where $G(F)$ acts on $M$ via the diagonal embedding $G(F)\into G_S$. We consider
\[
\mrm{H}^0\big(X_G^S(\frak{f}),M,N\big):=\mrm{H}^{0}\Big(G(F),\hspace{1mm}\cal{C}\big(G(\A^{S,\infty})/K_0^S,\hspace{0.5mm}\Hom_{\Z_p}(M,N)\big)\Big).
\]
Let $S_{\mrm{bad}}$ denote the set of prime divisors of $\frak{f}_A$, the cohomology group $\mrm{H}^0\big(X_G^S(\frak{f}),M,N\big)$ is endowed with the action of the spherical Hecke algebra
\[
\mathbb{T}^\mrm{sph}:=\cal{C}_c\Big(K_0^{S_{\mrm{bad}}}\backslash G\big(\A^{S_{\mrm{bad}},\infty}\big)/K_0^{S_{\mrm{bad}}},\Z\Big).
\]

\begin{corollary}\label{piA-isotypic subspace}
    The plectic Abel--Jacobi map belongs to
    \[
    \varphi^\Box_A\hspace{1mm}\in\hspace{1mm} \mrm{H}^0\left(X_G^{S}(\frak{f});\ \Z_p[\cal{H}_{S}^\Box],\ A(E_{\otimes,\cal{S}})\otimes_{\Q_p}\mrm{H}^1(E^\Box, V_p(A))\right)_{\pi_A}.
    \]
\end{corollary}
\begin{proof}
    Proposition \ref{prop:G(F)-invariance} means that 
    \[
     \varphi^\Box_A\hspace{1mm}\in\hspace{1mm} \mrm{H}^0\left(X_G^{S}(\frak{f});\ \Z_p[\cal{H}_{S}^\Box],\ A(E_{\otimes,\cal{S}})\otimes_{\Q_p}\mrm{H}^1(E^\Box, V_p(A))\right).
    \]
    The fact that $\varphi^\Box_A$ belongs to the $\pi_A$-isotypic subspace (with respect to the ``good'' Hecke operators) follows from the observation the Hecke action on $\varphi^\Box_A$ is via the Hecke action on the measure and hence on the automorphic function $\Phi_\cal{S}^\Box$.
\end{proof}
By considering the (mock) quaternionic Shimura variety 
\[
X_G^S(\frak{f})^\Box:=G(F)\backslash \big(\cal{H}_S^\Box\times G(\A^{S,\infty})/K_0^S\big)
\]
we can interpret Corollary \ref{piA-isotypic subspace} as giving an plectic Abel--Jacobi map
\[
\varphi^\Box_A\colon X_G^S(\frak{f})^\Box\longrightarrow A(E_{\otimes,\cal{S}})\otimes_{\Q_p}\mrm{H}^1(E^\Box, V_p(A)).
\]

\subsection{Invariance under Galois action}
For every $\p\in\cal{S}$ the embedding $\psi\colon T\hookrightarrow G$ of $F$-algebraic groups induces an action of $T(F)$ on $\cal{H}_\p$ with two fixed points $z_{\p}^\psi,\overline{z}_{\p}^\psi$. We set 
\[
z^\psi_{S}:=\big(\{z^\psi_\p\}_{\p\in\cal{S}},\hspace{1mm}z^\psi_w\big)\ \in\ \cal{H}_{S}^\mrm{glo}.
\]
\begin{remark}\label{rmk: reduction of fixed point}
    By assumption \eqref{choice of optimal embedding}, the fixed point $z^\psi_{\cal{S}}=(z^\psi_\p)_{\p\in\cal{S}}$ reduces to the tuple of standard vertices $\mrm{v}_{\circ}=\mrm{red}(z^\psi_{\cal{S}})\in\cal{V}_\cal{S}$. In particular, we can suppose that $g_{\mrm{v}^\psi_\cal{S}}=1$.
\end{remark}
\begin{definition}
    We let $\cal{H}_{S}^{\scalebox{0.6}{$\mrm{CM}$}}$ denote the $G(F)$-orbit of   $z^\psi_{S}\in\cal{H}_{S}^\mrm{glo}$.
\end{definition}

The aim of this section is to understand the Galois action on the classes 
\[
\varphi_A^\mrm{glo}(h)(\tau_S)\ \in \ A(E_{\otimes,\cal{S}})\otimes_{\Q_p}\mrm{H}^1({^{w}}E_{\mrm{ab}}, V_p(A))
\]
for $h\in G(\A^{S,\infty})$ and $\tau_S\in\cal{H}_{S}^{\scalebox{0.6}{$\mrm{CM}$}}$. Note that equation \eqref{reciprocityLAW} implies
 \begin{equation}\label{Galois action on exotic AJ map}
    \varphi^\mrm{glo}_A(h)(\tau_{S})^{\mrm{rec}_E(u)}=\varphi^\mrm{glo}_A(\psi(u)\cdot h)(\tau_{S})\qquad \forall\ u\in T(\A^{S,\infty}).
  \end{equation}
Moreover, by Proposition \ref{prop:G(F)-invariance} it suffices to understand the Galois action on $\varphi_A^\mrm{glo}(h)(z^\psi_{S})$. For any $h\in G(\A^{S,\infty})$, consider the CM point
\[
x_{\psi,h}:=\big[z^\psi_w,\hspace{1mm}  h\big]\in X_0(E_{\psi,h})
\]
defined over the finite abelian extension $E_{\psi,h}$ of $E$ characterized by
\[
\mrm{rec}_E\colon T(F)\backslash T(\bb{A}^\infty)/U_{\psi,h}\overset{\sim}{\longrightarrow}\mrm{Gal}(E_{\psi,h}/E)
\]
where $U_{\psi,h}:=\psi^{-1}(K_{\psi,h})$ for $K_{\psi,h}:=   h\cdot K_0\cdot   h^{-1}$.

\begin{theorem}\label{thm: Galois invariance}
We have
\[
\varphi^\mrm{glo}_A(h)(z^\psi_{S})\ \in\  A(E_{\otimes,\cal{S}})\otimes_{\Q_p}\mrm{H}^1\big(E_{\psi,h}, V_p(A)\big).
\]
\end{theorem}
\begin{proof}
Let $u\in U_{\psi,h}$. From the definition of $\varphi^\mrm{glo}_A$ and Remark \ref{rmk: reduction of fixed point} we know that
 \[
    \varphi^\mrm{glo}_\varpi(h)(z^\psi_{S})^{\mrm{rec}_E(u)}=
\int_{\cal{X}_\cal{S}}\hspace{1.5mm}f_{z^\psi_\cal{S},\varpi}(\gamma)\hspace{1.5mm}\mrm{d}\big(\Phi^\mrm{glo}_\cal{S}(\psi^w(u)\cdot h)(z^\psi_w)\big)(\gamma).
  \]
The definition of $U_{\psi,h}$ ensures the existence of $k_u\in K_0$ such that
\begin{equation}\label{def: k_a}
\psi^w(u)\cdot h=h\cdot k_u.
\end{equation}
The properties of $\Phi^\mrm{glo}_\cal{S}\in \cal{A}_{w,\cal{S}}(K_0;\cal{H}_w^\mrm{glo},\cal{M}^\mrm{glo}_\cal{S})$ then imply
\[
\Phi^\mrm{glo}_\cal{S}(\psi^w(u)\cdot h)=(k_u)_\cal{S}^{-1}.\Phi^\mrm{glo}_\cal{S}( h).
\]
In other words, we found that
  \[
   \varphi^\mrm{glo}_\varpi(h)(z^\psi_{S})^{ \mrm{rec}_E(u)}=
\int_{\cal{X}_\cal{S}}\hspace{1.5mm}f_{z^\psi_\cal{S},\varpi}((k_u)_\cal{S}^{-1}\cdot  \gamma)\hspace{1.5mm}\mrm{d}\big(\Phi^\mrm{glo}_\cal{S}(h)(z^\psi_w)\big)(\gamma).
  \]
  The idea of the proof is now to show that for every $\p\in\cal{S}$ the difference of the local terms
  \begin{equation}\label{additive equality}
f_{z^\psi_\p,\varpi}\big((k_u)_\p^{-1}\cdot  \gamma_\p\big)-f_{z^\psi_\p,\varpi}\big(  \gamma_\p\big)=\frac{1}{2}\log_{\varpi_\p}\bigg(\frac{(v\overline{z}_{\p}^\psi+w)^2}{\det(\psi(u)_\p^{-1})}\bigg)
  \end{equation}
 is a constant depending on $\psi(u)^{-1}_\p=\begin{pmatrix}
    x&y\\v&w
\end{pmatrix}$. One can then invoke Theorem \ref{vanishing implies invariance} to conclude
\[
  \varphi^\mrm{glo}_A(h)(z^\psi_{S})^{ \mrm{rec}_E(u)}=  \varphi^\mrm{glo}_A(h)(z^\psi_{S}).
\]
To establish \eqref{additive equality}, recall that  for every $\p\in\cal{S}$ we have
  \[f_{z^\psi_\p,\varpi}(\gamma_\p)=\log_{\varpi_\p}\bigg(\begin{pmatrix}
      1&-z^\psi_\p
  \end{pmatrix}\gamma_\p\begin{pmatrix}
      1\\0
  \end{pmatrix}\bigg)-\frac{1}{2}\log_{\varpi_\p}\Big(\det(\gamma_\p)\Big),
 \]
  and that we described the action of $S_\p^\varpi$ in \eqref{eq: formula action for function}. 
Now note that equation \eqref{def: k_a} implies
  \[
  \begin{pmatrix}
      1&-z^\psi_\p
  \end{pmatrix} (k_{u})_\p^{-1}\cdot \gamma_\p\begin{pmatrix}
      1\\0
  \end{pmatrix}=\begin{pmatrix}
      1&-z^\psi_\p
  \end{pmatrix}\psi(u)_\p^{-1}\cdot  \gamma_\p\begin{pmatrix}
      1\\0
  \end{pmatrix}
  \]
Finally, since $\psi(u)^{-1}_\p$ is diagonalizable with eigenvectors $\begin{pmatrix}
    z^\psi_\p\\1
\end{pmatrix}$, $\begin{pmatrix}
    \overline{z}^\psi_\p\\1
\end{pmatrix}$, and corresponding eigenvalues $(vz^\psi_\p+w),(v\overline{z}^\psi_\p+w)$ we find that
\[
\begin{pmatrix}
    1&-z^\psi_\p
\end{pmatrix}\psi(u)^{-1}_\p=(v\overline{z}^\psi_\p+w)\cdot \begin{pmatrix}
    1&-z^\psi_\p
\end{pmatrix}.
\]

\end{proof}

\subsection{Cap products}

For a non-zero ideal $\mathfrak{c}\subseteq \OO_F$ let $U(\frak{c})$ denote the image of $\widehat{\cal{O}}^\times_\mathfrak{c}$ in $ T(\A^\infty)$. We set
\[
\mrm{H}^0(X_T(\mathfrak{c}),N):=\mrm{H}^0\big(T(F),\cal{C}(T(\A^{\infty})/U(\frak{c}),N)\big),
\]
\[
\mrm{H}_0(X_T(\mathfrak{c}),N):=\mrm{H}_0\big(T(F),\cal{C}_c(T(\A^{\infty})/U(\frak{c}),N)\big)
\]
for any any abelian group $N$. More concretely, $\mrm{H}^0(X_T(\mathfrak{c}),N)$ is naturally the space of functions
\[
\cal{C}\big(T(F)\backslash T(\A^{\infty})/U(\frak{c}),\hspace{1mm} N\big),
\]
while $\mrm{H}_0(X_T(\mathfrak{c}),N)$ can be identified with
\[
\mrm{H}_0(X_T(\mathfrak{c}),N)\overset{\sim}{\longrightarrow} N\Big[T(\A^{\infty})/U(\frak{c})\Big]_{T(F)},\qquad \beta\mapsto \left[ \sum_{t\in T(\A^{\infty})/U(\frak{c})}\beta(t)\cdot [t] \right].
\]
There is a natural map
\begin{equation}\label{map: inv to cov}
\mrm{H}^0(X_T(\mathfrak{c}),N)\longrightarrow \mrm{H}_0(X_T(\mathfrak{c}),N),\qquad \alpha\mapsto \left[\sum_{t\in J(\frak{c})}\alpha(t)\cdot [t] \right]
\end{equation}
where $J(\frak{c})\subset T(\A^{\infty})/U(\frak{c})$ is any set of representatives of the finite quotient $T(F)\backslash T(\A^{\infty})/U(\frak{c})$,
and multiplication of functions induces the cap product pairing
\[
\cap\colon \mrm{H}^{0}(X_T(\mathfrak{c}),N)\times \mrm{H}_{0}(X_T(\cal{O}_F),\Z)\longrightarrow \mrm{H}_{0}(X_T(\mathfrak{c}),N),\qquad \alpha\cap [t]=\left[\sum_{u\in U(\cal{O}_F)/U(\frak{c})}\alpha(\tilde{t}u)\cdot [\tilde{t}u]\right]
\]
where $\tilde{t}\in T(\A^{\infty})/U(\frak{c})$ is any lift of $t\in T(\A^{\infty})/U(\cal{O}_F)$.
In this setting, the fundamental class $\vartheta\in \mrm{H}_{0}(X_T(\OO_F),\Z)$ of \cite[Definition 4.4]{plecticHeegner} is simply the image of the constant function $1\in \mrm{H}^{0}(X_T(\OO_F),\Z)$ under \eqref{map: inv to cov}. 

\bigskip
\noindent By assumption \eqref{choice of optimal embedding}, there is a non-zero ideal $\mathfrak{c}\subseteq \OO_F$ coprime to $\mathfrak{f}_A$ such that  $\psi^{-1}(K_0)^S=U(\frak{c})^S$. Recall that $L_\mathfrak{c}/E$ is the anticyclotomic abelian extension of $E$ of conductor $\mathfrak{c}$, with Galois group $\mathcal{G}_{\mathfrak{c}}:=\Gal(L_\mathfrak{c}/E)$, such that the Artin map induces an isomorphism
\[
\mrm{rec}_E\colon T(F)\backslash T(\A^{\infty})/U(\frak{c})\xlongrightarrow{\sim}\mathcal{G}_{\mathfrak{c}}.
\]
Any character $\chi\colon \mathcal{G}_{\mathfrak{c}}\to \overline{\bb{Q}}^{\times}$ can be thought of as an element in $\mrm{H}^{0}(X_T(\mathfrak{c}),\overline{\bb{Q}})$, and the $\chi$-twisted fundamental class 
\[\
\vartheta_{\chi}:=\chi\cap\vartheta\in \mrm{H}_{0}(X_T(\mathfrak{c}),\overline{\bb{Q}})
\]
of \cite[Definition 4.5]{plecticHeegner} turns out to be the image of $\chi\in \mrm{H}^{0}(X_T(\mathfrak{c}),\overline{\bb{Q}})$ under \eqref{map: inv to cov}.
Note that 
\[
T\big(\A^{\Sigma^{\mbox{\tiny $-$}}\cup\Sigma_\infty}\big)/U(\mathfrak{c})^{\Sigma^{\mbox{\tiny $-$}}}=\hspace{1mm}T\big(\A^{S,\infty}\big)/U(\mathfrak{c})^{S}\hspace{1mm}=\hspace{1mm}T(\A^{\infty})/U(\mathfrak{c})
\]
because every prime ideal in $\Sigma^{\mbox{\tiny $-$}}$ is inert in $E$ and $\mathfrak{c}$ is coprime to $\frak{f}_A$.
Then, the inclusion $\psi\colon T\hookrightarrow G$ together with the $T(F)$-equivariant homomorphism 
\[
\Psi_S\colon \Z_p\longrightarrow \Z_p[\cal{H}_{S}^\Box],\qquad 1\mapsto [z_S^\psi]
\]
induce a map
\[\xymatrix{
\mrm{H}^0\left(X_G^{S}(\frak{f});\ \Z_p[\cal{H}_{S}^\Box],\ A(E_{\otimes,\cal{S}})\otimes_{\Q_p}\mrm{H}^1(E^\Box, V_p(A))\right)\ar[d]_-{\Psi^*_S}&\\
\mrm{H}^0\left(X_T(\frak{c});\ A(E_{\otimes,\cal{S}})\otimes_{\Q_p}\mrm{H}^1(E^\Box, V_p(A))\right).
}\]
Finally, considering the cap product pairing
\[
\cap\colon \mrm{H}^{0}(X_T(\mathfrak{c}),N)\times \mrm{H}_{0}(X_T(\frak{c}),M)\longrightarrow N\otimes_\Z M,\qquad \alpha\cap m[t]=\alpha(t)\otimes m,
\]
we are ready for the following:
\begin{definition}
The \emph{plectic Heegner class} attached to  $(A_{/F},\cal{S},\chi,\Box)$ is 
\[
\kappa_{A,\cal{S}}^{\chi,\Box}:=(\Psi_S)^{\ast}( \varphi_{A}^\Box)\cap \vartheta_{\chi^{-1}}\hspace{1mm} \in\hspace{1mm}   A(E_{\otimes,\cal{S}})\otimes_{\Q_p}\mrm{H}^1(E^\Box, V_p(A))\otimes_{\Z}\overline{\Q}.
\]
\end{definition}

\begin{corollary}\label{field of definition of plectic Heegner class}
The plectic Heegner class belongs to
    \[
    \kappa_{A,\cal{S}}^{\chi,\mrm{glo}}\hspace{1mm}  \in\hspace{1mm}   A(E_{\otimes,\cal{S}})\otimes_{\Q_p}\mrm{H}^1(L_\frak{c}, V_p(A))^\chi
    \]
    and is independent of the choice of the auxiliary place $\{w\}=S\setminus\cal{S}$.
\end{corollary}
\begin{proof}
    Unraveling the definitions one finds that
\begin{equation}\label{explicit description plectic heegner class}
\kappa_{A,\cal{S}}^{\chi,\mrm{glo}}=\sum_{t\in T(F)\backslash T(\A^{\Sigma^{\mbox{\tiny $-$}}\cup\Sigma_\infty})/U(\mathfrak{c})^{\Sigma^{\mbox{\tiny $-$}}}}\varphi_{A}^\mrm{glo}(t)(z_S^\psi)\otimes \chi^{-1}(t)
\end{equation}
with the slight abuses of notation $\varphi_{A}^\Box(t):=\varphi_{A}^\Box(\psi(t))$ and $\chi^{-1}(t):=\chi^{-1}(\mrm{rec}_E(t))$. As $T\big(\A^{\Sigma^{\mbox{\tiny $-$}}\cup\Sigma_\infty}\big)$ is commutative and $\psi^{-1}(K_0)=U(\frak{c})$, Theorem \ref{thm: Galois invariance} implies that
\[
\varphi_{A}^\Box(t)(z_S^\psi)\hspace{1mm}\in\hspace{1mm}A(E_{\otimes,\cal{S}})\otimes_{\Q_p}\mrm{H}^1\big(L_\frak{c}, V_p(A)\big)\qquad \forall\ t\in T\big(\A^{\Sigma^{\mbox{\tiny $-$}}\cup\Sigma_\infty}\big).
\]
Moreover, combining \eqref{Galois action on exotic AJ map} with \eqref{explicit description plectic heegner class}, one finds that  
\[
\big(\kappa_{A,\cal{S}}^{\chi,\mrm{glo}}\big)^{g}=\chi(g)\cdot \kappa_{A,\cal{S}}^{\chi,\mrm{glo}}\qquad \forall\ g\in \cal{G}_\frak{c}.
\]
The independence of the choice of auxiliary place $w\in (\Sigma^{\mbox{\tiny $-$}}\cup\Sigma_\infty)\setminus\cal{S}$ is a consequence of the formula \eqref{explicit description plectic heegner class}, because it shows that the construction of $\kappa_{A,\cal{S}}^{\chi,\mrm{glo}}$ involves only CM points ``unramified at every $\Sigma^{\mbox{\tiny $-$}}\setminus\cal{S}$''.
\end{proof}

\section{Comparison with previous constructions}\label{Comparison with previous constructions}
The goal of this section is to show how to recover plectic Heegner points, and mock plectic invariants  from plectic Heegner classes.  Note that we continue to consider a coherent set $S=\cal{S}\cup\{w\}$ for the triple $(A_{/F}, E/F, p)$.

\subsubsection{Notation.}  We let $\Z_p[\cal{H}_\p]^0:=\ker(\Z_p[\cal{H}_\p]\to\Z_p)$ and consider the tensor products
\[
\Z_p[\cal{H}_{\cal{S}}]^\plectic:=\bigotimes_{\p\in\cal{S}}\Z_p[\cal{H}_\p]^0,\qquad \Z_p\big[\cal{H}_{S}^\Box\big]^\plectic:=\Z_p\big[\cal{H}_{\cal{S}}\big]^\plectic\otimes_{\Z_p}\Z_p\big[\cal{H}_w^\Box\big].
\]
Note that $\Z_p\big[\cal{H}_{S}^\Box\big]^\plectic\subset \Z_p[\cal{H}_S^\Box]$ is stable under the twisted $G(F)$-action. 

\bigskip
\noindent  Define a partition $\cal{S}=\cal{S}^+\amalg\cal{S}^-$  in terms of the reduction type of $A_{/F}$ by setting
\[
\cal{S}^\pm:=\big\{\p\in\cal{S}\ \lvert\ \varepsilon_\p=\pm1\big\}
\]
where $\varepsilon_\p=+1$ (resp. $-1$) if $A_{/F_\p}$ has split (resp. non-split) multiplicative reduction. 
For any $v\in S$, denote by  $\msf{H}_v$ the $G_v$-space $\bb{P}^1(E_v)\setminus\bb{P}^1(F_v)$ with its usual $G_v$-action via M\"obius transformations, and consider the $G_\cal{S}$-representations
\[
\mrm{St}_{\cal{S}^\pm}:=\bigotimes_{\p\in\cal{S}}\mrm{St}^{\varepsilon_\p}_\p,\qquad \Z_p[\msf{H}_{\cal{S}^\pm}]^\plectic:=\bigotimes_{\p\in\cal{S}}\Z_p^{\varepsilon_\p}[\msf{H}_\p]^0\
\]
where $\mrm{St}_\p^\pm$ (resp. $\Z_p^\pm[\msf{H}_\p]^0$) denotes the twist of the locally constant $\Z_p$-valued Steinberg representation of $G_\p$ (resp. of $\Z_p[\msf{H}_\p]^0)$ by a specialization of the twisting character \eqref{twisting character}
\[
\chi_\p^\pm\colon G_\p\longrightarrow \{\pm1\},\qquad g\mapsto (\pm1)^{\mrm{ord}_\p(\det(g))}.
\]
 Let $\mrm{St}^{\pm}_{\p}\big(\widehat{E}_{\p}^\times\big)^{\mrm{ct}}$ be the continuous $\widehat{E}_{\p}^\times$-valued Steinberg representation of $G_\p$. There are natural  $G_\p$-equivariant maps
\begin{equation}\label{notation: steinberg-divisors}
\widehat{E}_\p^\times\otimes_{\Z_p} \mrm{St}^\pm_\p\longrightarrow\mrm{St}^{\pm}_{\p}\big(\widehat{E}_{\p}^\times\big)^{\mrm{ct}},\qquad \Z_p^\pm[\msf{H}_{\p}]^0\longrightarrow \mrm{St}^{\pm}_{\p}\big(\widehat{E}_{\p}^\times\big)^{\mrm{ct}},
\end{equation}
 the first is induced by the inclusion of locally constant function in continuous function, while the second maps a divisor  $D_\p$ of degree zero to the class of functions on $\bb{P}^1(F_\p)$ with divisor $D_\p$.
 Finally, for any subset $\Sigma\subseteq\cal{S}$ we let $\chi_{\Sigma^\pm}\colon G_\Sigma\to \{\pm1\}$ denote the character 
 \begin{equation}\label{chi+-}
\chi_{\Sigma^\pm}:=\prod_{\p\in\Sigma}\chi_\p^{\varepsilon_\p},
 \end{equation}
  a specialization of the character $\chi_\Sigma\colon G_\Sigma\to G(E_\Sigma/F_\Sigma)$ defined in \eqref{def: twist Sigma character}.

\subsection{Review}\label{section: review}
We briefly recall the approach of \cite{plecticHeegner} and \cite{Iwasawa4mock}. In  \cite[Section 2.2.2]{Iwasawa4mock} it is explained how to use the non-constant $F$-rational morphism $J_1^\circ\to A$ arising from the modularity of $A_{/F}$, and the $w$-adic uniformization 
\[
\phi_w^\mrm{CD}\colon G(F)\backslash\Big(\cal{H}_w^\Box\times G(\A^{w,\infty})/K_1^\circ\Big)\longrightarrow X_1^\circ(E^\Box)
\] 
to produce an element
\[
\omega^\Box_{A,w}\hspace{1mm}\in\hspace{1mm} \mrm{H}^0\big(X_G^w(\frak{f});\ \Z_p[\cal{H}_w^\Box],\ A(E^\Box)\otimes_\Z\Z_p\big)_{\pi_A}.
\]

\subsubsection{Relation to \cite{plecticHeegner}.}\label{Relation to plecticHeegner}
  When $w\not\in\Sigma_\infty$  it is possible to define an element 
    \[
    c^\Box_{A,w}\hspace{1mm}\in\hspace{1mm} \mrm{H}^0\big(X_G^w(\frak{f});\ \Z_p^{\varepsilon_w}[\msf{H}_w^\Box],\ A(E^\Box)\otimes_\Z\Z_p\big)_{\pi_A}
    \]
    from $\omega^\Box_{A,w}$.  For any $b\in G(\A^{w,\infty})/K_1^\circ$ and any $\tau\in\msf{H}_w^\Box$, one sets 
    \[
   c^\Box_{A,w}(b)(\tau):=\omega^\Box_{A,w}(b)(\tau)^{\chi_w(g_\mrm{v})}
    \]
    for $\mrm{v}=\mrm{red}(\tau)$. In other words, we  claim that 
    \[
    \chi^{\varepsilon_w}_w(\delta)\cdot c^\Box_{A,w}(\delta^wb)(\delta_w(\tau))= c^\Box_{A,w}(b)(\tau)\qquad \forall \ \delta\in G(F).
    \]
    To verify the claim,  one computes that
    \[\begin{split}
    \chi^{\varepsilon_w}_w(\delta)\cdot c^\Box_{A,w}(\delta^wb)(\delta_w(\tau))&= \chi^{\varepsilon_w}_w(\delta)\cdot\omega^\Box_{A,w}(\delta^wb)(\delta_w(\tau))^{\chi_w(g_{\delta\mrm{v}})}\\
     &=\omega^\Box_{A,w}(\delta^w b)(\delta_w.\tau)^{\chi_w(g_{\mrm{v}})}\\
     &=\omega^\Box_{A,w}(b)(\tau)^{\chi_w(g_{\mrm{v}})}\\
     &= c^\Box_{A,w}(b)(\tau)
    \end{split}\]
    where the second equality is a consequence of $\chi_w(g_{\delta\mrm{v}})=\chi_w(\delta)\cdot \chi_w(g_{\mrm{v}})$ and the Galois equivariance property \eqref{Galois equivariance CD parametrization}, while the third equality follows from the properties of $\omega^\Box_{A,w}$. 

    Furthermore, when $w\in\Sigma_p$, the theory of Mumford curves implies that $c^\mrm{loc}_{A,w}$ can be described using multiplicative integrals. More precisely, the image of $c^\mrm{loc}_{A,w}$ under 
    \[
    \mrm{H}^0\big(X_G^w(\frak{f});\ \Z_p^{\varepsilon_w}[\msf{H}_w],\ A(E_w)\otimes_\Z\Z_p\big)\longrightarrow \mrm{H}^0\big(X_G^w(\frak{f});\  \Z_p^{\varepsilon_w}[\msf{H}_w]^0,\ A(E_w)\otimes_\Z\Z_p\big)
    \]
    equals the image of a generator of  $\mrm{H}^0\big(X_G^w(\frak{f});\hspace{1mm} \mrm{St}_w^{\varepsilon_w},\hspace{0.5mm}  \Z_p\big)_{\pi_A}$  under  the composition
    \[\xymatrix{
    \mrm{H}^0\big(X_G^w(\frak{f});\hspace{1mm}  \mrm{St}_w^{\varepsilon_w}, \hspace{0.5mm} \Z_p\big)\ar[r]\ar@{.>}[dr] & \mrm{H}^0\big(X_G^w(\frak{f}); \hspace{1mm} \mrm{St}^{{\varepsilon_w}}_w(\widehat{E}_w^\times)^{\mrm{ct}}, \hspace{0.5mm}  A(E_w)\otimes_\Z\Z_p\big)\ar[d]\\
    &
     \mrm{H}^0\big(X_G^w(\frak{f});\hspace{1mm}  \Z_p^{\varepsilon_w}[\msf{H}_w]^0,\hspace{0.5mm}  A(E_w)\otimes_\Z\Z_p\big).
    }\]
   We refer to \cite[Section 4.1]{plecticHeegner} and \cite[Section 3.2]{polyquadraticPlectic} for more details.

\subsubsection{The associated measure on $\bb{P}^1(F_\cal{S})$.}
For every $\p\in\cal{S}$, and depending on the chosen isomorphism $G_\p\cong \mrm{PGL}_2(F_\p)$,  there is a $G_\p$-equivariant way to associate to an oriented edge $e_\p\in\cal{E}_\p$ of the Bruhat-Tits tree an open compact subset $U_{e_\p}\subset\bb{P}^1(F_\p)$, such that the collection $\{U_{e_\p}\}_{e_\p\in\cal{E}_\p}$ form a basis for the $p$-adic topology of $\bb{P}^1(F_\p)$. Moreover, the $G_\p$-action of $\cal{E}_\p$ provides an identification $G_\p/K_{\p^1}^\circ\cong \cal{E}_\p$ singling out a distinguished oriented edge $e_{\p,\circ}$. 

As $\bb{P}^1(F_\cal{S})=\prod_{\p\in\cal{S}}\bb{P}^1(F_\p)$, one deduces the existence of  a $G_\cal{S}$-equivariant function associating to a tuple of oriented edges $e=(e_\p)_\p$ in $ \cal{E}_\cal{S}:=\prod_{\p\in\cal{S}}\cal{E}_\p$ an open compact subset $U_{e}\subset\bb{P}^1(F_\cal{S})$, such that the collection $\{U_{e}\}_{e\in\cal{E}_\cal{S}}$ form a basis for the $p$-adic topology of $\bb{P}^1(F_\cal{S})$. We denote by $e_\circ$ the distinguished tuple of oriented edges $(e_{\p,\circ})_\p$ and for any $e\in\cal{E}_\cal{S}$ we let $g_e\in G_\cal{S}$ denote an element such that $e=g_e(e_\circ)$. We refer to \cite[Section 3.2]{plecticHeegner} for more details.

\medskip
\noindent By induction, the arguments in  \cite[Section 2.2.2]{Iwasawa4mock} also show that $\omega^\Box_{A,w}$ vanishes in the cokernel of the injective map
\[
\mrm{H}^0\big(X_G^S(\frak{f});\ \mrm{St}_{\cal{S}^\pm}\otimes_{\Z_p} \Z_p[\cal{H}_w^\Box],\ A(E^\Box)\otimes_\Z\Z_p\big)_{\pi_A}\hooklongrightarrow \mrm{H}^0\big(X_G^w(\frak{f});\ \Z_p[\cal{H}_w^\Box],\ A(E^\Box)\otimes_\Z\Z_p\big)_{\pi_A}
\]
\begin{equation}\label{defining property}
\phi\mapsto \big[b\mapsto \phi(b^S)(b_\cal{S}.\mathbbm{1}_{U_{e_\circ}})\big],
\end{equation}
because the elliptic curve $A_{/F}$ has multiplicative reduction at the primes in $\cal{S}$ (see also \cite[Section 3.3]{plecticHeegner}). Therefore, it is possible to lift the the class $\omega_{A,w}^\Box$ to
\begin{equation}\label{lift of the class}
\omega_{A,S}^\Box\hspace{1mm}\in\hspace{1mm} \mrm{H}^0\big(X_G^S(\frak{f});\ \mrm{St}_{\cal{S}^\pm}\otimes_{\Z_p} \Z_p[\cal{H}_w^\Box],\ A(E^\Box)\otimes_\Z\Z_p\big)_{\pi_A}.
\end{equation} 

In light of the non-Archimedean uniformization of quaternionic Shimura varieties, it is convenient to modify slightly the class $\omega_{A,S}^\Box$. Regard the $\Z_p$-module $\widehat{E}_{\p}^\times$ as a $G_\p$-representation through the homomorphism $\chi_\p\colon G_\p\to G(E_\p/F_\p)$, and consider its tensor product with the Steinberg representation
\[
\mrm{St}^{\dagger}_{\p}(\widehat{E}_{\p}^\times):=\widehat{E}_{\p}^\times\otimes_{\Z_p}\mrm{St}_{\p}.
\]
Let $\mrm{St}^{\dagger}_{\p}\big(\widehat{E}_{\p}^\times\big)^{\mrm{ct}}$ denote the continuous $\widehat{E}_{\p}^\times$-valued Steinberg representation with $G_\p$-action twisted by the action of $\chi_\p$ on coefficients. Then, similarly to \eqref{notation: steinberg-divisors}, there are natural  $G_\p$-equivariant maps
\begin{equation}\label{notation: twisted steinberg-divisors}
\mrm{St}^{\dagger}_{\p}(\widehat{E}_{\p}^\times)\longrightarrow\mrm{St}^{\dagger}_{\p}\big(\widehat{E}_{\p}^\times\big)^{\mrm{ct}},\qquad \Z_p[\cal{H}_{\p}]^0\longrightarrow \mrm{St}^{\dagger}_{\p}\big(\widehat{E}_{\p}^\times\big)^{\mrm{ct}},
\end{equation}
induced respectively by the inclusion of locally constant function in continuous function, and by associating to a divisor  $D_\p$ of degree zero, the class of functions on $\bb{P}^1(F_\p)$ with divisor $D_\p$.
We consider the $G_\cal{S}$-module
\[
\mrm{St}^{\dagger}_{\cal{S}}(\widehat{E}_{\otimes,\cal{S}}^\times):=\bigotimes_{\p\in\cal{S}}\mrm{St}^{\dagger}_{\p}(\widehat{E}_{\p}^\times)
\]
with $G_\cal{S}$-action given by
\[
g.(\xi\otimes\mathbbm{1}_{U_e}):=\xi^{\chi_\cal{S}(g)}\otimes\mathbbm{1}_{U_{g e}}\qquad\forall\ g\in G_\cal{S}
\]
where $\xi\in \widehat{E}_{\otimes, \cal{S}}^\times$ and $e\in\cal{E}_\cal{S}$. 
Inspired by the discussion in Section \ref{Relation to plecticHeegner} we propose the following:
\begin{definition}\label{def: multiplicative variant}
 For all $h\in G(\A^{S,\infty})/(K_0)^S$, we let
\[
\omega_{A,S}^{\Box,\times}(h)\colon  \mrm{St}^\dagger_{\cal{S}}(\widehat{E}_{\otimes,\cal{S}}^\times)\otimes_{\Z_p}\Z_p[\cal{H}_w^\Box]\longrightarrow  \widehat{E}_{\otimes,\cal{S}}^\times\otimes_{\Z} A(E^\Box)
\]
be the $\Z_p$-linear function defined by 
\[
\omega_{A,S}^{\Box,\times}(h)(\xi\otimes\mathbbm{1}_{U_e}\otimes[\tau_w]):=\chi_{\cal{S}^\pm}(g_e)\cdot \xi^{\chi_\cal{S}(g_e)}\otimes\big(\omega_{A,S}^\Box(h)(\mathbbm{1}_{U_e}\otimes[\tau_w])\big)
\]
 for all $\xi\in \widehat{E}_{\otimes,\cal{S}}^\times$, $\mathbbm{1}_{U_e}\in\mrm{St}_\cal{S}$, and $\tau_w\in\cal{H}_w^\Box$.  
\end{definition}
\begin{lemma}
    The function $\omega_{A,S}^{\Box,\times}$  is $G(F)$-invariant, i.e.,
    \[
  \omega_{A,S}^{\Box,\times}(\delta^S\cdot -)(\delta_S.-)=\omega_{A,S}^{\Box,\times}(-)(-)\qquad \forall\ \delta\in G(F).
    \]
    In other words, 
    \[
\omega_{A,S}^{\Box,\times}\hspace{1mm}\in\hspace{1mm} \mrm{H}^0\big(X_G^S(\frak{f});\ \mrm{St}^\dagger_{\cal{S}}(\widehat{E}_{\otimes,\cal{S}}^\times)\otimes_{\Z_p} \Z_p[\cal{H}_w^\Box],\ \widehat{E}_{\otimes,\cal{S}}^\times\otimes_\Z A(E^\Box)\big)_{\pi_A}.
    \]
\end{lemma}
\begin{proof}
By definition 
 \[
 \delta_S.(\xi\otimes\mathbbm{1}_{U_e}\otimes[\tau_w])=\xi^{\chi_\cal{S}(\delta)}\otimes\mathbbm{1}_{U_{\delta e}}\otimes[\delta.\tau_w].
 \]
Thus, we compute
    \[\begin{split}
   \omega_{A,S}^{\Box,\times}(\delta^Sh)(\delta_S.(\xi\otimes\mathbbm{1}_{U_e}\otimes[\tau_w]))&=\chi_{\cal{S}^\pm}(g_{\delta e})\cdot (\xi^{\chi_\cal{S}(\delta)})^{ \chi_\cal{S}(g_{\delta e})}\otimes\big(\omega_{A,S}^\Box(\delta^Sh)(\mathbbm{1}_{U_{\delta e}}\otimes[\delta.\tau_w])\big)\\
     &=\chi_{\cal{S}^\pm}(g_e)\cdot \xi^{\chi_\cal{S}(g_{ e})}\otimes \chi_{\cal{S}^\pm}(\delta)\cdot\big(\omega_{A,S}^\Box(\delta^Sh)(\mathbbm{1}_{U_{\delta e}}\otimes[\delta.\tau_w])\big)\\
     &= \chi_{\cal{S}^\pm}(g_e)\cdot \xi^{\chi_\cal{S}(g_{ e})}\otimes\big(\omega_{A,S}^\Box(h)(\mathbbm{1}_{U_{ e}}\otimes[\tau_w])\big)\\
     &=  \omega_{A,S}^{\Box,\times}(h)(\xi\otimes\mathbbm{1}_{U_e}\otimes[\tau_w])
    \end{split}\]
    where the second equality is a consequence of $\chi_\cal{S}(g_{\delta e})=\chi_\cal{S}(\delta)\cdot \chi_\cal{S}(g_{e})$, while the third equality follows from the properties of $\omega^\Box_{A,S}$.  
\end{proof}

\subsubsection{Plectic Abel--Jacobi map for plectic zero-cycles.}
The morphisms \eqref{notation: twisted steinberg-divisors} give $G_\cal{S}$-equivariant maps
\[
\mrm{St}^{\dagger}_{\cal{S}}(\widehat{E}_{\otimes,\cal{S}}^\times)\longrightarrow \mrm{St}^{\dagger}_{\cal{S}}(\widehat{E}_{\otimes,\cal{S}}^\times)^{\mrm{ct}},\qquad \Z_p\big[\cal{H}_{\cal{S}}^\Box\big]^\plectic\longrightarrow \mrm{St}^{\dagger}_{\cal{S}}(\widehat{E}_{\otimes,\cal{S}}^\times)^{\mrm{ct}},
\]
which combine with Tate's uniformization $\phi^\mrm{Tate}_\cal{S}\colon\widehat{E}_{\otimes, \cal{S}}^\times\to A(E_{\otimes,\cal{S}})$ (see Definition \ref{def: Exotic AJ map}) to give 
\begin{equation}\label{integrating twisted plectic divisors}\xymatrix{
\mrm{H}^0\big(X_G^S(\frak{f});\ \mrm{St}^\dagger_{\cal{S}}(\widehat{E}_{\otimes,\cal{S}}^\times)\otimes_{\Z_p} \Z_p[\cal{H}_w^\Box],\ \widehat{E}_{\otimes,\cal{S}}^\times\otimes_\Z A(E^\Box)\big)_{\pi_A}\ar[d]&\\
\mrm{H}^0\big(X_G^S(\frak{f});\ \Z_p\big[\cal{H}_{S}^\Box\big]^\plectic,\ A(E_{\otimes,\cal{S}})\otimes_{\Z_p} A(E^\Box)^\wedge\big)_{\pi_A}.
}\end{equation}
We denote  by
\[
\varphi_A^{\circ,\Box}\hspace{1mm}\in\hspace{1mm}\mrm{H}^0\big(X_G^S(\frak{f});\  \Z_p[\cal{H}_{S}^\Box]^\plectic,\ A(E_{\otimes,\cal{S}})\otimes_{\Z_p} A(E^\Box)^\wedge\big)_{\pi_A}
\]
the image of $\omega_{A,S}^{\Box,\times}$ under \eqref{integrating twisted plectic divisors}.

\subsubsection{Alternative description.}
It is possible to give a slightly different description of $\varphi_A^{\circ,\Box}$ which will be used in the comparison with the function $\varphi_A^{\Box}$ of Corollary \ref{piA-isotypic subspace}:

\medskip
For any $D_\p\in \Z_p[\cal{H}_\p]^0$ and $u_\p\in\p\cal{O}_\p\setminus\{0\}$  we let $\ell_{D_\p, \omega_\p}\colon \bb{P}^1(F_\p)\to E_\p$ be the function defined by linearly extending the association 
\[
([\tau_\p]-[\tau_\p'])\mapsto \log_{u_\p}\left(\frac{t_\p-\tau_\p}{t_\p-\tau_\p'}\right).
\]
 For $D=\otimes_\p D_\p\in \Z_p[\cal{H}_{\cal{S}}]^\plectic$ and $u=(u_\p)_{\p\in\cal{S}}$ we then set
\[
\ell_{D,u}\colon\bb{P}^1(F_S)\longrightarrow E_{\otimes,\cal{S}}\qquad \ell_{D,u}(t):=\otimes_{\p\in\cal{S}} \ell_{D_\p, u_\p}(t_\p).
\]
We consider the following  ``additive'' variant of the class in Definition \ref{def: multiplicative variant}
\[
\omega_{A,S}^{\Box,+}\hspace{1mm}\in\hspace{1mm} \mrm{H}^0\big(X_G^S(\frak{f});\ \mrm{St}^\dagger_{\cal{S}}(E_{\otimes,\cal{S}})\otimes_{\Z_p} \Z_p[\cal{H}_w^\Box],\ E_{\otimes,\cal{S}}\otimes_\Z A(E^\Box)\big)_{\pi_A}
\]
where for all $h\in G(\A^{S,\infty})$, $\xi\in E_{\otimes,\cal{S}}$, $\mathbbm{1}_{U_e}\in\mrm{St}_\cal{S}$, and $\tau_w\in\cal{H}_w^\Box$ one sets
\[
\omega_{A,S}^{\Box,+}(h)(\xi\otimes\mathbbm{1}_{U_e}\otimes[\tau_w]):=\chi_{\cal{S}^\pm}(g_e)\cdot \xi^{\chi_\cal{S}(g_e)}\otimes\big(\omega_{A,S}^\Box(h)(\mathbbm{1}_{U_e}\otimes[\tau_w])\big).
\]
Then, letting $q_\p\in \p\cal{O}_\p$ denote the Tate parameter of $A_{/F_\p}$ and setting $q=(q_\p)_{\p\in\cal{S}}$, we find that 
\[
\varphi_A^{\circ,\Box}(h)(D\otimes[\tau_w])=(\exp_{q_\cal{S}}^\mrm{Tate}\otimes1)\left(\int_{\bb{P}^1(F_\cal{S})}\ell_{D,q}(t)\hspace{1mm}\mrm{d}\big(\omega_{A,S}^{\Box,+}(h)(\tau_w)\big)(t)\right)
\]
for all $h\in G(\A^{S,\infty})$ and $D\otimes[\tau_w]\in\Z_p[\cal{H}_{S}^\Box]^\plectic$.

\subsubsection{Mock plectic invariants.} Suppose $F=\bb{Q}$, $\cal{S}=\{p\}$ and $w=\infty$. Then, mock plectic invariants are directly obtained from the class $\varphi_A^{\circ,\mrm{glo}}$ by means of a cap product (\cite[Section 3.4]{DarmonFornea}, \cite[Definition 2.7]{Iwasawa4mock}).

\subsubsection{Plectic Heegner points.}
Suppose $w\in\Sigma_p$. The construction of plectic Heegner points requires an additional step which consists in showing that $\varphi_A^{\circ,\mrm{loc}}$ can be lifted with respect to the injection
\begin{equation}\label{pHP injection}
\xymatrix{\mrm{H}^0\big(X_G^S(\frak{f});\hspace{1mm}   \Z_p[\cal{H}_{S}],\hspace{0.5mm}  A(E_{\otimes,\cal{S}})\otimes_{\Z_p} A(E_w)^\wedge\big)_{\pi_A}\ar@{^{(}->}[d]_-{(\iota_\plectic)_*}\\
\mrm{H}^0\big(X_G^S(\frak{f});\hspace{1mm}   \Z_p[\cal{H}_{S}]^\plectic,\hspace{0.5mm}  A(E_{\otimes,\cal{S}})\otimes_{\Z_p} A(E_w)^\wedge\big)_{\pi_A}
}\end{equation}
induced by the natural inclusion $\iota_\plectic\colon \Z_p[\cal{H}_{S}]^\plectic\hookrightarrow \Z_p[\cal{H}_{S}]$. The existence of the lift 
\[
(\varphi_A^{\circ,\mrm{loc}})^\sim\hspace{1mm}\in \hspace{1mm}\mrm{H}^0\big(X_G^S(\frak{f});\hspace{1mm}   \Z_p[\cal{H}_{S}],\hspace{0.5mm}  A(E_{\otimes,\cal{S}})\otimes_{\Z_p} A(E_w)^\wedge\big)_{\pi_A}
\]
of $\varphi_A^{\circ,\mrm{loc}}$ is deduced in \cite[Theorem 4.10]{plecticHeegner} by showing that $\varphi_A^{\circ,\mrm{loc}}$ vanishes in the cokernel of \eqref{pHP injection}. Plectic Heegner points are then obtained from the class $(\varphi_A^{\circ,\mrm{loc}})^\sim$ by means of a cap product (\cite{plecticHeegner}, Definition 4.11).

\subsubsection{Strategy.}  To recover previous constructions from plectic Heegner classes, we claim that it suffices to prove the following key result.
\begin{theorem}\label{thm: comp circ and nocirc}
     The natural inclusion $\iota_\plectic\colon \Z_p[\cal{H}_{S}^\Box]^\plectic\hookrightarrow \Z_p[\cal{H}_{S}^\Box]$ together with the Kummer map $\cal{K}\colon A(E^\Box)^\wedge\hookrightarrow \mrm{H}^1(E^\Box,T_p(A))$ produce the identification
    \[
    (\iota_\plectic)_*\varphi_A^\Box=(\cal{K})_*\varphi_A^{\circ,\Box}
    \]
    in $\mrm{H}^0\left(X_G^{S}(\frak{f});\ \Z_p[\cal{H}_{S}^\Box]^\plectic,\ A(E_{\otimes,\cal{S}})\otimes_{\Q_p}\mrm{H}^1(E^\Box,V_p(A))\right)$.
\end{theorem}
\begin{proof}
The proof will be given in Section \ref{Proof of theorem comp circ and nocirc}. It is the culmination of a careful analysis of $(\iota_\plectic)_*\varphi_A^\Box$ obtained in Section \ref{sect: On the behavior of certain integrals}, and the reinterpretation of the class 
\[
\omega_{A,S}^\Box\hspace{1mm}\in\hspace{1mm}\mrm{H}^0\big(X_G^S(\frak{f});\ \mrm{St}_{\cal{S}^\pm}\otimes_{\Z_p} \Z_p[\cal{H}_w^\Box],\ A(E^\Box)\otimes_\Z\Z_p\big)_{\pi_A}
\]
in terms of towers of Shimura curves given in Section \ref{sect: Reinterpretation}.
\end{proof}
For mock plectic invariants the sufficiency is clear. For plectic Heegner points, it is a consequence of the following lemma.

\begin{lemma}
    Suppose that $w\in\Sigma_p$ and that Theorem \ref{thm: comp circ and nocirc} holds, then the equality
    \[
    \varphi^\mrm{loc}_A=(\varphi_A^{\circ,\mrm{loc}})^\sim
    \]
    is valid in $\mrm{H}^0\left(X_G^{S}(\n^{\mbox{\tiny $+$}});\ \Z_p[\cal{H}_{S}],\ A(E_{\otimes,\cal{S}})\otimes_{\Z_p}A(E_w)^\wedge\right)_{\pi_A}$.
\end{lemma}
\begin{proof}
    For any $\Q_p$-linear retraction $r\colon \mrm{H}^1(E_w, V_p(A))\to A(E_w)^\wedge\otimes_{\Z_p}\Q_p$ of the Kummer map, the equality $(\iota_\plectic)_*\varphi_A^\mrm{loc}=(\cal{K})_*\varphi_A^{\circ,\mrm{loc}}$ implies
    \[
    (\iota_\plectic)_* r_*(\varphi^\mrm{loc}_A)=\varphi_A^{\circ,\mrm{loc}}.
    \]
   Since $(\varphi_A^{\circ,\mrm{loc}})^\sim$ is the unique element  in $\mrm{H}^0\left(X_G^{S}(\n^{\mbox{\tiny $+$}});\ \Z_p[\cal{H}_{S}],\ A(E_{\otimes,\cal{S}})\otimes_{\Z_p}A(E_w)^\wedge\right)_{\pi_A}$ lifting $\varphi_A^{\circ,\mrm{loc}}$, we deduce that
    \[
    r_*(\varphi^\mrm{loc}_A)=(\varphi_A^{\circ,\mrm{loc}})^\sim.
    \]
    In particular, $r_*(\varphi^\mrm{loc}_A)$ is independent of the retraction. As  $A(E_w)^\wedge\otimes_{\Z_p}\Q_p\not=0$, we deduce that \[
\varphi^\mrm{loc}_A\hspace{1mm}\in\hspace{1mm}\mrm{H}^0\big(X_G^{S}(\frak{f});\ \Z_p[\cal{H}_{S}],\ A(E_{\otimes,\cal{S}})\otimes_{\Z_p}A(E_w)^\wedge\big)_{\pi_A}
    \]
    and the claim.
\end{proof}

\subsection{On the behavior of certain $p$-adic integrals}\label{sect: On the behavior of certain integrals}
Recall that  $f_A\colon\bb{T}_{\overline{\varrho}}\twoheadrightarrow \Z_p$ is the homomorphism associated to the modular elliptic curve $A_{/F}$, whose kernel $\cal{P}\in\mrm{Spec}(\bb{T}_{\overline{\varrho}})$  lay above the augmentation ideal $\cal{I}\in\mrm{Spec}(\Lambda^{\bfcdot})$.  Letting  $\scr{F}\colon \bb{T}_{\overline{\varrho}}\to \bb{I}$ denote the localization map at $\bb{I}=(\bb{T}_{\overline{\varrho}})_\cal{P}$, for $\p\in\cal{S}$ and $\varpi_\p\in\cal{O}_\p$ a uniformizer, we consider
 \[
 \alpha_{\varpi_\p}:=\scr{F}(\mbf{U}_{\varpi_\p}),
 \]
 and define $\varepsilon_\p:=f_A(\mbf{U}_{\varpi_\p})\in\{\pm1\}$ which is independent of the choice of uniformizer. 
As noted in Section \ref{Interlude on arithmetic L-invariants}, since the structure map $(\Lambda^{\bfcdot})_\cal{I}\to\bb{I}$ is an isomorphism (Corollary \ref{isolocalizations}), there exists  
\[
\mrm{pw}_\otimes\colon \bb{I}\to R_{\otimes,\cal{S}},\quad T\mapsto T^\otimes
\]
mapping a group-like element $[a_\cal{S}]$ for $a_\cal{S}\in\cal{O}_\cal{S}^\times$ to $ a_\cal{S}^{\underline{X}_\cal{S}}\in R_{\otimes,\cal{S}}$. Hence,
\[
(\alpha_{\varpi_\p}^\otimes)\big\lvert_{\underline{0}}=\varepsilon_\p.
\]
In Definition \ref{Distribution} we associated to    $\Phi_\cal{S}^\Box\in \cal{A}_{w,\cal{S}}(K_0;\cal{H}_w^\Box,\cal{M}^\Box_\cal{S})$ the collection of measures
\[
\mu^{\varpi,\Box}_{(\tau_S,h)}\in (\widetilde{\cal{M}}^\Box_{\cal{S}})_{\overline{\varrho}}\Big[G(E_\cal{S}/F_\cal{S})\Big]
\]
indexed by $\tau_S\in \cal{H}_{S}^\Box$ and $h\in G(\A^{S,\infty})$.
We are now interested in studying the integrals
\begin{equation}\label{eq: quantity to study}
\left(\int_{V}F_{\lambda_\cal{S},\varpi}(\gamma)\hspace{1mm}\mrm{d}\mu^{\varpi,\Box}_{(\tau_S,h)}(\gamma)\right)\otimes1\ \in\ R_{\otimes,\cal{S}}\otimes_{\Z_p} (J^\Box_\cal{S})^\wedge_{\overline{\varrho}}\otimes_{\bb{T}_{\overline{\varrho}}}\bb{I}
\end{equation}
for $\lambda_\cal{S}\in\cal{H}_\cal{S}$ and $V=\prod_{\p\in\cal{S}}V_\p$ with $V_\p\subset\widetilde{\cal{Y}}_\p$ compact open subset stable under the action of $\cal{O}_\p^\times$ for every $\p\in\cal{S}$. It is quite convenient to observe the following fact:

\begin{lemma}\label{lemma: T to Totimes}
For every $T\in \bb{I}$ we have   
\[
   \left(\int_{V}F_{\lambda_\cal{S},\varpi}(\gamma)\hspace{1mm}\mrm{d}\mu^{\varpi,\Box}_{(\tau_S,h)}(\gamma)\right)\otimes T=T^\otimes\cdot \left(\int_{V}F_{\lambda_\cal{S},\varpi}(\gamma)\hspace{1mm}\mrm{d}\mu^{\varpi,\Box}_{(\tau_S,h)}(\gamma)\right)\otimes1.
    \]
\end{lemma}
\begin{proof}
It suffices to verify the claim for diamond operators. By Remark \ref{rmk: log as derivative II} we know that $F_{\lambda_\cal{S},\varpi}(-\hspace{0.5mm}\bfcdot\hspace{1mm} a^{-1})=a^{\underline{X}_\cal{S}}\cdot F_{\lambda_\cal{S},\varpi}(-)$ for all $a\in\cal{O}_\cal{S}^\times$, hence for $\mu=\mu^{\varpi,\Box}_{(\tau_S,h)}$ we compute
    \[\begin{split}
      \left(\int_{V}F_{\lambda_\cal{S},\varpi}(\gamma)\hspace{1mm}\mrm{d}\mu(\gamma)\right)\otimes \llangle a\rrangle&=\left(\int_{V}F_{\lambda_\cal{S},\varpi}(\gamma\hspace{0.5mm}\bfcdot\hspace{1mm} a^{-1})\hspace{1mm}\mrm{d} \mu(\gamma)\right)\otimes1\\
&=a^{\underline{X}_\cal{S}}\cdot\left(\int_{V}F_{\lambda_\cal{S},\varpi}(\gamma)\hspace{1mm}\mrm{d} \mu(\gamma)\right)\otimes1.
    \end{split}\]
\end{proof}

Set $\bb{P}^1_\Sigma:=\bb{P}^1(F_\Sigma)$ and denote by $\pi_\Sigma\colon \cal{Y}_\cal{S}\twoheadrightarrow \bb{P}^1_\Sigma\times\cal{Y}_{\Sigma^c}$
the projection map associated to a subset $\Sigma\subseteq\cal{S}$.
As the restriction of $\pi_\Sigma$ to any compact open subset of $\cal{Y}_\cal{S}$ has compact fibers, the pushforward under $\pi_\Sigma$ of any measure in $\widetilde{\cal{M}}_\cal{S}^\Box$  is well-defined. We will also be interested in understanding the properties of integrals obtained from  measures of the form $\pi_{\Sigma,*}\Big(\mu^{\varpi,\Box}_{(\tau_S,h)}\Big)$. To that end, it is convenient to generalize slightly the definitions given in the previous section.
 For a subset $\Sigma\subseteq\cal{S}$ consider the element $\beta_\Sigma^\varpi:=\prod_{\p\in \Sigma}\beta_\p^\varpi\in \Delta_\cal{S}^\varpi$ obtained from 
\[
\beta_\p^\varpi:=\big(\eta_{\varpi_\p}\cdot c_\p\cdot \eta_{\varpi_\p}\cdot c_\p\big)\ \in\ \Delta_\p^\varpi\qquad \text{where}\qquad c_\p=\begin{pmatrix}
    0&1\\1&0
\end{pmatrix}\in K_{\p^0}.
\]
We define the  measure $\mu^{\varpi,\Box,\Sigma}_{(\tau_S,h)}$ by setting
\[
\mu^{\varpi,\Box,\Sigma}_{(\tau_S,h)}:=\beta_\Sigma^\varpi.\hspace{0.5mm}\mu^{\varpi,\Box}_{(\tau_S,h)}\ \in\ \widetilde{\cal{M}}^\Box_\cal{S}\Big[G(E_\cal{S}/F_\cal{S})\Big].
\]
It is supported on 
\[
\beta_\Sigma^\varpi.\hspace{0.5mm}\bb{X}_{\mrm{v}_\cal{S}}\overset{\sim}{\longrightarrow} \prod_{\p\in \Sigma}\cal{O}_\p^\times\backslash\big[(\p L_{\mrm{v}_\p})'\times\cal{O}_\p^\times\big]\times \prod_{\p\in \Sigma^c}\cal{O}_\p^\times\backslash\big[L_{\mrm{v}_\p}'\times\cal{O}_\p^\times\big].
\]
\begin{remark}
    The action of $\beta_\p^\varpi\in \Delta_\p^\varpi$ commutes with the action of any element in $\Delta_\p^\varpi$. This property follows from Lemma \ref{when-actions-coincide} and the fact that the image of $\beta_\p^\varpi$ in $G_\p$ is trivial.
\end{remark}

\begin{lemma}\label{translation+pushforward}
For any $\Sigma\subseteq\cal{S}$ the $\pi_{\Sigma}$-pushforward of the measures $\mu^{\varpi,\Box}_{(\tau_S,h)}$ and $\mu^{\varpi,\Box,\Sigma}_{(\tau_S,h)}$ coincide.
\end{lemma}
\begin{proof}
Fix $\p\in\cal{S}$. We noted in Remark \ref{rmk: observation on action} that the projection $\cal{Y}_\p\twoheadrightarrow \bb{P}^1(F_\p)$ intertwines the action of $\Delta^\varpi_\p$ on $\cal{Y}_\p$ with the action of its image in $G_\p$  on $\bb{P}^1(F_\p)$. The claim then follows by observing that the element $\beta_\p^\varpi$ becomes trivial in $G_\p$.
\end{proof}


\subsubsection{Comparison of measures.}\label{sect: Comparison of measures}
Fix $\p\in\cal{S}$. For any two points $\tau_S,\tau_S'\in\cal{H}_{S}$ we denote by $\mrm{v}_\p, \mrm{v}_\p'$ the vertices of $\cal{T}_\p$ associated to their $\p$-components. Assume that
\begin{equation}\label{conditions on points}
m_{\mrm{v}'_\p}=m_{\mrm{v}_\p}+1\qquad\text{and}\qquad (\tau_S)_v=(\tau_S')_v\qquad \forall\ v\in S\setminus\{\p\}.
\end{equation}
It follows that we can assume $\widetilde{g}_{\mrm{v}_\cal{S}'}=\widetilde{g}_{\mrm{v}_\cal{S}}\cdot\eta_{\varpi_\p}$. Then, using equation \eqref{iso1}, one checks that 
\[
\cal{X}_{\p^1}^\circ=\cal{X}_\p\cap\big(\eta_{\varpi_\p}.\cal{X}_\p\big),\qquad c_\p.\cal{X}_{\p^1}^\circ=\cal{X}_\p\cap\big((c_\p\cdot \eta_{\varpi_\p}\cdot c_\p).\cal{X}_\p\big)
\]
to deduce
\begin{equation}\label{description of intersection}
\widetilde{g}_{\mrm{v}_\cal{S}}.\big(\cal{X}_{\p^1}^\circ\times\cal{X}^\p\big)=\bb{X}_{\mrm{v}_\cal{S}}\cap\bb{X}_{\mrm{v}_\cal{S}'}\qquad \text{and}\qquad (\widetilde{g}_{\mrm{v}'_\cal{S}}\cdot c_\p).\big(\cal{X}_{\p^1}^\circ\times\cal{X}^\p\big)=\bb{X}_{\mrm{v}'_\cal{S}}\cap\big(\beta_\p.\bb{X}_{\mrm{v}_\cal{S}}\big).
\end{equation}

\begin{proposition}\label{prop: comparison of distributions}
Let $\sigma_\p\in \mrm{Gal}(E_\p/F_\p)$ be the generator. The following equalities hold
\[
\mrm{Res}_{\bb{X}_{\mrm{v}_\cal{S}}\cap\bb{X}_{\mrm{v}_\cal{S}'}}\big(\mu^{\varpi,\Box}_{(\tau_S',h)}\big)\big[\sigma_\p\big]=\mrm{Res}_{\bb{X}_{\mrm{v}_\cal{S}}\cap\bb{X}_{\mrm{v}_\cal{S}'}}\big(\mu^{\varpi,\Box}_{(\tau_S,h)}\big)\otimes \mbf{U}_{\varpi_\p}
\]
and
\[
\mrm{Res}_{\bb{X}_{\mrm{v}_\cal{S}'}\cap(\beta_\p.\bb{X}_{\mrm{v}_\cal{S}})}\big(\mu^{\varpi,\Box,\{\p\}}_{(\tau_S,h)}\big)\big[\sigma_\p\big]=\mrm{Res}_{\bb{X}_{\mrm{v}_\cal{S}'}\cap(\beta_\p.\bb{X}_{\mrm{v}_\cal{S}})}\big(\mu^{\varpi,\Box}_{(\tau_S',h)}\big)\otimes \mbf{U}_{\varpi_\p}.
\]
\end{proposition}
\begin{proof}
  The reasoning of \cite[Lemma 2.9]{BDHidaRational}  extends to our setting.  We present only the proof of the first claim, since the proof of the second is analogous. We will show that 
  \begin{equation}\label{prop: claimed identity}
    \int_{\bb{X}_{\mrm{v}_\cal{S}}\cap\bb{X}_{\mrm{v}_\cal{S}'}} \phi( t)\hspace{1.5mm}\mrm{d}\mu^{\varpi,\Box}_{(\tau_S,h)}\otimes\mbf{U}_{\varpi_\p}=  \int_{\bb{X}_{\mrm{v}_\cal{S}}\cap\bb{X}_{\mrm{v}_\cal{S}'}}\phi(t)^{\sigma_\p}\hspace{1.5mm}\mrm{d}\mu^{\varpi,\Box}_{(\tau_S',h)}
  \end{equation}
  for every continuous function $\phi\in \cal{C}(\bb{X}_{\mrm{v}_\cal{S}}\cap\bb{X}_{\mrm{v}_\cal{S}'},N)$ valued in a finite dimensional $\Q_p$-vector space with $G(E_\cal{S}/F_\cal{S})$-action. Equality \eqref{description of intersection} suggests to consider the automorphic function
 $\Psi$  arising from $\Phi_\cal{S}^\Box\in \cal{A}_{w,\cal{S}}(K_0;\cal{H}_w^\Box,\cal{M}^\Box_\cal{S})$ via the map induced by restriction to $\cal{X}^\circ_{\p^1}\times\cal{X}^\p\subset \cal{X}_\cal{S}$.
We note that since $T_{\varpi_\p}.\Phi_\cal{S}^\Box=\Phi_\cal{S}^\Box.\mbf{U}_{\varpi_\p}$, we have $U_{\varpi_\p}.\Psi=\Psi.\mbf{U}_{\varpi_\p}$ -- see the proof Proposition \ref{restriction - intertwining}. As in equation \eqref{formula m_U_p-action}, we can write
   \begin{equation}\label{equality of integrals}
    \int_{\cal{X}^\circ_{\p^1}\times\cal{X}^\p} \phi'(t)\hspace{1.5mm}\mrm{d}(U_{\varpi_\p}.\Psi(b)[z])=\sum_{a\in\kappa_\p}\int_{\cal{X}^\circ_{\p^1}\times\cal{X}^\p}\phi'(\widehat{\sigma}_{\p,a}\cdot t)\hspace{1.5mm}\mrm{d}(\Psi(b\widehat{\sigma}_{\p,a})[z])
    \end{equation}
 for any $b\in G(\A^{w,\infty}), z\in\cal{H}_w$, and any continuous function $\phi'\in \cal{C}(\cal{X}^\circ_{\p^1}\times\cal{X}^\p,N)$. 

 To establish \eqref{prop: claimed identity}, we substitute in equation \eqref{equality of integrals} the following quantities: $b=g_{\mrm{v}_\cal{S}}h$ for $h\in G(\A^{S,\infty})$, $z=\tau_w$, and  $\phi'(t)=\phi(\widetilde{g}_{\mrm{v}_\cal{S}}\cdot t)^{\chi(g_{\mrm{v}_\cal{S}})}$ where $\phi\in \cal{C}(\bb{X}_{\mrm{v}_\cal{S}}\cap\bb{X}_{\mrm{v}_\cal{S}'},N)$. Then, the LHS of \eqref{equality of integrals} satisfies
  \[\begin{split}
    \int_{\cal{X}^\circ_{\p^1}\times\cal{X}^\p} \phi(\widetilde{g}_{\mrm{v}_\cal{S}}\cdot t)\hspace{1.5mm}\mrm{d}(U_{\varpi_\p}.\Psi(g_{\mrm{v}_\cal{S}}h)[\tau_w])&= \int_{\widetilde{g}_{\mrm{v}_\cal{S}}.(\cal{X}^\circ_{\p^1}\times\cal{X}^\p)} \phi(t)^{\chi(g_{\mrm{v}_\cal{S}})}\hspace{1.5mm}\mrm{d}(\widetilde{g}_{\mrm{v}_\cal{S}}.\Psi(g_{\mrm{v}_\cal{S}}h)[\tau_w])\otimes\mbf{U}_{\varpi_\p}\\
    &=  \int_{\widetilde{g}_{\mrm{v}_\cal{S}}.(\cal{X}^\circ_{\p^1}\times\cal{X}^\p)} \phi(t)^{\chi(g_{\mrm{v}_\cal{S}})}\hspace{1.5mm}\mrm{d}(\widetilde{g}_{\mrm{v}_\cal{S}}.\Phi_\cal{S}^\Box(g_{\mrm{v}_\cal{S}}h)[\tau_w])\otimes\mbf{U}_{\varpi_\p}\\
   &=  \int_{\bb{X}_{\mrm{v}_\cal{S}}\cap\bb{X}_{\mrm{v}_\cal{S}'}} \phi(t)\hspace{1.5mm}\mrm{d}\mu^{\varpi,\Box}_{(\tau_S,h)}\otimes\mbf{U}_{\varpi_\p}.
   \end{split}\]
 Similarly, the RHS of \eqref{equality of integrals} can be written as
\[\begin{split}
\sum_{a\in\kappa_\p}\int_{\cal{X}^\circ_{\p^1}\times\cal{X}^\p}\phi(\widetilde{g}_{\mrm{v}_\cal{S}}\cdot\widehat{\sigma}_{\p,a}\cdot t)^{\chi(g_{\mrm{v}_\cal{S}})}&\hspace{1.5mm}\mrm{d}(\Phi_\cal{S}^\Box(g_{\mrm{v}_\cal{S}}\cdot \widehat{\sigma}_{\p,a}\cdot h)[\tau_w])\\
=& \sum_{a\in\kappa_\p}\int_{(\widetilde{g}_{\mrm{v}_\cal{S}}\cdot\widehat{\sigma}_{\p,a}).(\cal{X}^\circ_{\p^1}\times\cal{X}^\p)}\ \phi(t)^{\sigma_\p}\hspace{1.5mm}\mrm{d}\mu^{\varpi,\Box}_{(\tau_S',h)}\\
=&\int_{\bb{X}_{\mrm{v}_\cal{S}}\cap\bb{X}_{\mrm{v}_\cal{S}'}}\phi(t)^{\sigma_\p}\hspace{1.5mm}\mrm{d}\mu^{\varpi,\Box}_{(\tau_S',h)}
    \end{split}\]
   where  the first equality follows from $\widetilde{g}_{\mrm{v}'_\cal{S}}=\widetilde{g}_{\mrm{v}_\cal{S}}\cdot\eta_{\varpi_\p}$ and $\widehat{\sigma}_{\p,a}=\eta_{\varpi_\p}\cdot\begin{pmatrix}
      1&0\\a&1 
   \end{pmatrix}$ which imply $\chi(g_{\mrm{v}'_\cal{S}})=\chi(g_{\mrm{v}_\cal{S}})\cdot\sigma_\p$ and   $(\widetilde{g}_{\mrm{v}_\cal{S}}\cdot \widehat{\sigma}_{\p,a}).\Phi_\cal{S}^\Box(g_{\mrm{v}_\cal{S}}\widehat{\sigma}_{\p,a}h)=\widetilde{g}_{\mrm{v}'_\cal{S}}.\Phi_\cal{S}^\Box(g_{\mrm{v}'_\cal{S}}h)$. The second equality follows from the decomposition 
   \[
   \cal{X}^\circ_{\p^1}\times\cal{X}^\p=\coprod_{a\in\kappa_\p}\widehat{\sigma}_{\p,a}.\big(\cal{X}^\circ_{\p^1}\times\cal{X}^\p\big)
   \]
  -- already noted in \eqref{decomposition domain of integration} -- and the relation $\chi_{\cal{S}}(g_{\mrm{v}'_S})=\varepsilon_{\p} \cdot \chi_{\cal{S}}(g_{\mrm{v}_S})$. The claim follows.
\end{proof}

We are ready to  study the properties of the measure
\begin{equation}\label{notation: pushforward measure}
\eta^{\varpi,\Box}_{(\tau_{\Sigma^{c}},h)}:=\chi_{\Sigma^\pm}(g_{\mrm{v}_\Sigma})\cdot \pi_{\Sigma,*}\Big(\mu^{\varpi,\Box}_{(\tau_S,h)}\Big)
\end{equation}
where $\chi_{\Sigma^\pm}\colon G_\Sigma\to\{\pm1\}$ is the character \eqref{chi+-}.
 In general, the measure $\eta^{\varpi,\Box}_{(\tau_{\Sigma^{c}},h)}$ depends on all components of $\tau_S$. Nonetheless, the next theorem justifies our slight abuse of notation.

\begin{proposition}\label{independence pushforward}
  Let $h\in G(\A^{S,\infty})$, $\tau_S\in \cal{H}_{S}^\Box$. Consider a non-empty subset $\Sigma\subseteq\cal{S}$ and a partition $\Sigma=\Sigma_1\amalg \Sigma_2$.  For any plectic zero divisor $D_{\Sigma_1}\in \Z_p[\cal{H}_{\Sigma_1}]^\plectic$, any collection $\omega_{\Sigma_1}=\{\omega_\p\in\p\cal{O}_\p\}_{\p\in\Sigma_1}$, any compact open subset $U_{\Sigma_2}\subseteq\bb{P}^1_{\Sigma_2}$, and any $\lambda_{\Sigma^c}\in\cal{H}_{\Sigma^c}$, the quantity
\[
 \left(\int_{\bb{P}^1_{\Sigma_1}\times U_{\Sigma_2} \times \bb{X}_{\mrm{v}_{\Sigma^c}}}\ell_{D_{\Sigma_1},\omega_{\Sigma_1}}\otimes\mathbbm{1}_{\Sigma_2}\otimes f_{\lambda_{\Sigma^c},\varpi}(\gamma)\hspace{2mm}\mrm{d}\eta^{\varpi,\Box}_{(\tau_{\Sigma^c},h)}(\gamma)\right)\otimes1
\]
 is independent of the $\Sigma$-component of the point $\tau_S$.
\end{proposition}
\begin{proof}
   The strategy of proof in \cite[Lemma 2.12]{BDHidaRational} can be adapted to our setting. The proof is by induction on the size of $\Sigma$. The idea is to vary $\tau_S$ one component at a time. Moreover, since for every $\p\in\cal{S}$ the Bruhat--Tits tree for $G_\p$ is a tree, it suffices to consider  two points $\tau_S,\tau_S'\in\cal{H}_{S}$ differing only at the $\p$ and such that the associated vertices $\mrm{v}_\p, \mrm{v}_\p'$  satisfy $m_{\mrm{v}'_\p}=m_{\mrm{v}_\p}+1$. Then, one can establish the inductive step by leveraging the covering
   \[
   \bb{P}^1(F_\p)=\pi_\p\big(\bb{X}_{\mrm{v}_\p}\cap\bb{X}_{\mrm{v}_\p'}\big)\hspace{1mm}\amalg\hspace{1mm}\pi_\p\big(\bb{X}_{\mrm{v}'_\p}\cap\beta_\p.\bb{X}_{\mrm{v}_\p}\big),
   \]
 Proposition \ref{prop: comparison of distributions}, Lemma \ref{lemma: T to Totimes}, the equality $(\alpha_{\varpi_\p}^\otimes)\big\lvert_{\underline{0}}=\varepsilon_\p$, and Theorem \ref{vanishing implies invariance}.
\end{proof}

\begin{remark}
    Proposition \ref{independence pushforward} plays a role in this article analogous to that of \cite[Lemma 2.12]{BDHidaRational} in \emph{loc.cit.} Note that we cannot simply prove an an equality of measures in our setting because that is only valid after localization at $\cal{P}$.
\end{remark}

\subsubsection{Further properties.}\label{sect: Further properties}
Let $\Sigma\subset\cal{S}$ be a subset and $\p\in\cal{S}\setminus \Sigma$.  Let $\tau_{\Sigma^c},\tau_{\Sigma^c}'\in\cal{H}_{\Sigma^c}$ such that
\[
m_{\mrm{v}'_\p}=m_{\mrm{v}_\p}+1\qquad\text{and}\qquad \tau_v=\tau'_v\qquad \forall\ v\in \Sigma^c\setminus \{\p\}.
\]
\begin{lemma}\label{intermediate-step}
Let $h\in G(\A^{S,\infty})$, $D_\Sigma\in\Z_p[\cal{H}_\Sigma]^\plectic$, $\omega_\Sigma=\{\omega_\p\in\p\cal{O}_\p\}_{\p\in\Sigma}$, and $\lambda_{\Sigma^c}\in\cal{H}_{\Sigma^c}$. Then,
\[\begin{split}
\int_{\bb{P}^1_\Sigma\times(\bb{X}_{\mrm{v}_{\Sigma^c}}\cup\bb{X}_{\mrm{v}_{\Sigma^c}'})}&\ell_{D_\Sigma,\omega_\Sigma}\otimes f_{\lambda_{\Sigma^c}}\hspace{3mm}\mrm{d}\big(\eta^{\varpi,\Box}_{(\tau_{\Sigma^c}',h)}- \eta^{\varpi,\Box}_{(\tau_{\Sigma^c},h)}\big)\\
&= \frac{\partial_\p\big(\alpha_{\varpi_\p}^\otimes\big)}{\alpha^\otimes_{\varpi_\p}}\bigg\lvert_{\underline{0}}\cdot\int_{\bb{P}^1_\Sigma\times \pi_{\{\p\}}(\bb{X}_{\mrm{v}_{\Sigma^c}}\cap\bb{X}_{\mrm{v}_{\Sigma^c}'})}\ell_{D_\Sigma,\omega_\Sigma}\otimes \mathbbm{1}_\p\otimes f_{\lambda_{\Sigma^c}\setminus\{\p\}}\hspace{3mm}\mrm{d}\eta^{\varpi,\Box}_{(\tau_{\Sigma^c\setminus\{\p\}},h)}.
    \end{split}\]
   
\end{lemma}
\begin{proof}
The proof follows the strategy of (\cite{BDHidaRational}, Lemmas 2.22 $\&$ 2.23). Using the partition
\[
\bb{X}_{\mrm{v}_{\Sigma^c}}\cup\bb{X}_{\mrm{v}_{\Sigma^c}'}=\big(\bb{X}_{\mrm{v}_{\Sigma^c}}\cap\bb{X}_{\mrm{v}_{\Sigma^c}'}\big)\hspace{1mm}\amalg\hspace{1mm}\big( \beta_\p.\bb{X}_{\mrm{v}_{\Sigma^c}}\cap\bb{X}_{\mrm{v}_{\Sigma^c}'}\big)\hspace{1mm}\amalg\hspace{1mm}\beta_\p^{-1}.\big( \beta_\p.\bb{X}_{\mrm{v}_{\Sigma^c}}\cap\bb{X}_{\mrm{v}_{\Sigma^c}'}\big),
\]
it suffices to compute that
 \[\begin{split}
\int_{\bb{P}^1_\Sigma\times(\bb{X}_{\mrm{v}_{\Sigma^c}}\cap\bb{X}_{\mrm{v}_{\Sigma^c}'})}&\ell_{D_\Sigma,\omega_\Sigma}\otimes f_{\lambda_{\Sigma^c}}\hspace{3mm}\mrm{d}\big(\eta^{\varpi,\Box}_{(\tau_{\Sigma^c}',h)}- \eta^{\varpi,\Box}_{(\tau_{\Sigma^c},h)}\big)\\
&=\frac{\partial_\p\big(\alpha_{\varpi_\p}^\otimes\big)}{2\alpha^\otimes_{\varpi_\p}}\bigg\lvert_{\underline{0}}\cdot\int_{\bb{P}^1_\Sigma\times(\bb{X}_{\mrm{v}_{\Sigma^c}}\cap\bb{X}_{\mrm{v}_{\Sigma^c}'})}\ell_{D_\Sigma,\omega_\Sigma}\otimes \mathbbm{1}_\p\otimes f_{\lambda_{\Sigma^c\setminus\{\p\}}}\hspace{3mm}\mrm{d}\eta^{\varpi,\Box}_{(\tau_{\Sigma^c},h)},
\end{split}\]
and
\[\begin{split}
\int_{\bb{P}^1_\Sigma\times(\beta_\p.\bb{X}_{\mrm{v}_{\Sigma^c}}\cap\bb{X}_{\mrm{v}_{\Sigma^c}'})}&\ell_{D_\Sigma,\omega_\Sigma}\otimes f_{\lambda_{\Sigma^c}}\hspace{3mm}\mrm{d} \big(\eta^{\varpi,\Box,\{\p\}}_{(\tau_{\Sigma^c},h)}-\eta^{\varpi,\Box}_{(\tau_{\Sigma^c}',h)}\big)\\
&=\frac{\partial_\p\big(\alpha_{\varpi_\p}^\otimes\big)}{2\alpha^\otimes_{\varpi_\p}}\bigg\lvert_{\underline{0}}\cdot\int_{\bb{P}^1_\Sigma\times(\beta_\p.\bb{X}_{\mrm{v}_{\Sigma^c}}\cap\bb{X}_{\mrm{v}_{\Sigma^c}'})}\ell_{D_\Sigma,\omega_\Sigma}\otimes \mathbbm{1}_\p\otimes f_{\lambda_{\Sigma^c\setminus\{\p\}}}\hspace{3mm}\mrm{d}\eta^{\varpi,\Box}_{(\tau_{\Sigma^c}',h)}.
\end{split}\]
These last two equalities follows from a direct application of Proposition \ref{prop: comparison of distributions}, Lemma \ref{lemma: T to Totimes}, and the vanishing result of Theorem \ref{vanishing implies invariance}.
\end{proof}

We now explain a convenient interpretation of the equality appearing in Lemma \ref{intermediate-step}. The lattices $L_{\mrm{v}_\p}, L_{\mrm{v}_\p'}$ associated to $\tau_\p, \tau_\p'$ as above equation \eqref{domain in terms of lattice} satisfy
\[
\p L_{\mrm{v}_\p}\subsetneq L_{\mrm{v}_\p'}\subsetneq L_{\mrm{v}_\p}.
\]
Therefore, if $e_\p=e(\mrm{v}_\p',\mrm{v}_\p)\in\cal{E}_\p$ denotes the oriented edge from $\mrm{v}_\p'$ to $\mrm{v}_\p$, we find that 
\[
U_{e_\p}:=\pi_\p\big(L_{\mrm{v}_\p}'\cap L_{\mrm{v}_\p'}'\big)
\]
is the standard compact open subset of $\bb{P}^1(F_\p)$ associated to $e_\p$ -- see \cite[Section 3.2]{plecticHeegner}. In particular,
\[
\pi_{\{\p\}}\Big(\bb{X}_{\mrm{v}_{\Sigma^c}}\cap\bb{X}_{\mrm{v}_{\Sigma^c}'}\Big)=U_{e_\p}\times\bb{X}_{\mrm{v}_{\Sigma^c\setminus\{\p\}}}.
\]
Moreover, the proof of \cite[Lemma 2.5]{Bertolini-Darmon-Green} shows that
\[
\mrm{\ord}_{\p}\left(\frac{t_\p-\tau_\p}{t_\p-\tau_\p'}\right)=\begin{cases}
    1& t_\p\in U_{e_\p}\\
    0&t_\p\not\in U_{e_\p}.
\end{cases}
\]
Putting all together, the following holds:
\begin{corollary}\label{ord-integral}
Let $\Sigma\subset \cal{S}$ be a subset and $\p\in\cal{S}\setminus \Sigma$.  Fix $h\in G(\A^{S,\infty})$, $D_\Sigma\in\Z_p[\cal{H}_\Sigma]^\plectic$, $\omega_\Sigma=\{\omega_\p\in\p\cal{O}_\p\}_{\p\in\Sigma}$, $\lambda_{\Sigma^c}\in\cal{H}_{\Sigma^c}$. Let $\tau_{\Sigma^c},\tau_{\Sigma^c}'\in\cal{H}_{\Sigma^c}$ such that
\[
m_{\mrm{v}'_\p}=m_{\mrm{v}_\p}+1\qquad\text{and}\qquad \tau_v=\tau'_v\qquad \forall\ v\in \Sigma^c\setminus \{\p\}.
\]
Then, 
    \[
   \begin{split}
\int_{\bb{P}^1_\Sigma\times(\bb{X}_{\mrm{v}_{\Sigma^c}}\cup\bb{X}_{\mrm{v}_{\Sigma^c}'})}&\ell_{D_\Sigma,\omega_\Sigma}\otimes f_{\lambda_{\Sigma^c}}\hspace{3mm}\mrm{d}\big(\eta^{\varpi,\Box}_{(\tau_{\Sigma^c}',h)}- \eta^{\varpi,\Box}_{(\tau_{\Sigma^c},h)}\big)\\
=&\frac{\partial_\p\big(\alpha_{\varpi_\p}^\otimes\big)}{\alpha^\otimes_{\varpi_\p}}\bigg\lvert_{\underline{0}}\cdot\int_{\bb{P}^1_{\Sigma\cup\{\p\}}\times \bb{X}_{\mrm{v}_{\Sigma^c\setminus\{\p\}}}}\ell_{D_\Sigma,\omega_\Sigma}\otimes \mrm{\ord}_{\p}\left(\frac{t_\p-\tau_\p}{t_\p-\tau_\p'}\right)\otimes f_{\lambda_{\Sigma^c\setminus\{\p\}}}\hspace{3mm}\mrm{d}\eta^{\varpi,\Box}_{(\tau_{\Sigma^c\setminus\{\p\}},h)}.
    \end{split}\]
\end{corollary}

\subsection{Reinterpretation}\label{sect: Reinterpretation}
The aim of this section is present a construction of the  class 
\[
\omega_{A,S}^\Box\hspace{1mm}\in\hspace{1mm}\mrm{H}^0\big(X_G^S(\frak{f});\ \mrm{St}_{\cal{S}^\pm}\otimes_{\Z_p} \Z_p[\cal{H}_w^\Box],\ A(E^\Box)\otimes_\Z\Z_p\big)_{\pi_A},
\]
introduced in \eqref{lift of the class}, in terms of the tower of Shimura curves with $K_m^\circ$-level structures
\[\xymatrix{
X_0^\circ& X_1^\circ\ar[l]&\cdots\ar[l]_-{\pi_1}& X_m^\circ\ar[l]&\ar[l]\cdots
}\] 
We give an overview of the transposition of the construction given in Section \ref{Section: construction}. Let $J_m^\circ$ denote the Jacobian of $X_m^\circ$. For $\Box\in\{\text{loc}, \text{glo}\}$ set 
 \[
 J_m^{\circ,\Box}:=\begin{cases}
    J_m^\circ(E_w)&\Box=\text{loc}\\
     J_m^\circ({^{w}}E_{\mrm{ab}})&\Box=\text{glo}.\\
\end{cases}
\]
   The direct limit
    \[
J_\Sigma^{\circ,\Box}:=\varinjlim_m \hspace{1mm}\Big(J_m^{\circ,\Box}\otimes\bb{Z}_p, \pi^*\Big)
\] 
    inherits the structure of a right  $\bb{T}_\Sigma^\circ$-module. 
    \begin{remark}\label{rmk: projection to A}
       The morphisms $U_\cal{S}^{1-m}\circ(\pi_2)_*^{m-1}\colon J_m^{\circ,\Box}\to A(E^\Box)$ are compatible and induce a $\bb{T}_\cal{S}^\circ$-equivariant map $\mrm{pr}_A^\circ\colon J_\cal{S}^{\circ,\Box}\to A(E^\Box)_{\Z_p}$ for $A(E^\Box)_{\Z_p}:=A(E^\Box)\otimes_\Z\Z_p$.
    \end{remark}
The $w$-adic uniformization of $E_w$-points produces trace-compatible automorphic functions 
\[
\phi_m^{\circ,\Box}\in\cal{A}_{w,\Sigma}\big(K_m^\circ; \cal{H}_w^\Box, J_\Sigma^{\circ,\Box}\big)
\]
where $J_\Sigma^{\circ,\Box}$ is considered with the trivial $K_m^\circ$-action. It is endowed with an action of the Hecke algebra $\frak{h}^\circ_{\Sigma,m}$. A major simplification happens in this setting: the $p$-adic domain obtained as the projective limits of $K_0/K_m^\circ$ turns out to be the projective space
\[
\bb{P}^1(F_\Sigma)\cong\varprojlim_m K_0/K_m^\circ
\]
which admits a natural action of $G_\Sigma$ extending that of $(K_0)_\Sigma$ thanks to Iwasawa's decomposition. Considering the space $\cal{M}^{\circ,\Box}_{\Sigma}$ of $J_\Sigma^{\circ,\Box}$-valued measures on $\bb{P}^1(F_\Sigma)$ with the induced $G_\Sigma$-action, the space of automorphic function
\[
\cal{A}_{w,\Sigma}\big(K_0; \cal{H}_w^\Box, \cal{M}^{\circ,\Box}_{\Sigma}\big)
\]
admits an action of $\bb{T}_\Sigma^\circ$ on the coefficients, and there is a natural $\bb{T}_\Sigma^\circ$-equivariant isomorphism
\[
 \mrm{sp}_\infty\colon \cal{A}_{w,\Sigma}\big(K_0; \cal{H}_w^\Box, \cal{M}^{\circ,\Box}_{\Sigma}\big)\overset{\sim}{\longrightarrow} \varprojlim_m\ \cal{A}_{w,\Sigma}\big(K_m^\circ; \cal{H}_w^\Box, J_\Sigma^{\circ,\Box}\big).
\]
Write $\Phi_\Sigma^{\circ,\Box}\in \cal{A}_{w,\Sigma}\big(K_0; \cal{H}_w^\Box, \cal{M}^{\circ,\Box}_{\Sigma}\big)$
for the automorphic function satisfying $\mrm{sp}_\infty\big(\Phi_\Sigma^{\circ,\Box}\big)=\big\{\phi_m^{\circ,\Box}\big\}_{m}$. Explicitly, we have
\begin{equation}\label{explicit values phi circ}
\Phi_\Sigma^{\circ,\Box}(b)(z)(\xi\cdot \cal{X}^\circ_m)=j_m\circ\varphi_m\big([z,b\xi]_{K_m^\circ}\big)
\end{equation}
for all $b\in G(\A^{w,\infty})$, $z\in\cal{H}_w^\Box$, $\xi\in\ (K_0)_\cal{S}$.
Moreover, $\Phi_\Sigma^{\circ,\Box}$ satisfies 
\[
    T_\q.\Phi_\Sigma^{\circ,\Box}=\Phi_\Sigma^{\circ,\Box}.\mbf{T}_\q,\qquad U_\p.\Phi_\Sigma^{\circ,\Box}=\Phi_\Sigma^{\circ,\Box}.\mbf{U}_{\p}
    \]
    for $\q\nmid\frak{f}_A$, $\p\mid\frak{f}$.
\begin{remark}
The projection map $\pi_{\Sigma}\colon \cal{X}_{\Sigma}\longrightarrow\bb{P}^1(F_\Sigma)$ has compact fibers. Moreover, the compatibility of the level structures provides a map $\iota_{\Sigma}\colon J_\Sigma^{\circ,\Box}\longrightarrow  J_{\Sigma}^\Box$. The morphisms of left $\Z_p[(K_0)_\Sigma]$-modules
\[
(\pi_{\Sigma})_*\colon\cal{M}_{\Sigma}^\Box\to\cal{M}(\bb{P}^1(F_\Sigma),J_{\Sigma}^\Box),\qquad (\iota_{\Sigma})_*\colon\cal{M}_{\Sigma}^{\circ,\Box}\to\cal{M}(\cal{X}_\Sigma,J_{\Sigma}^\Box)
\]
allows us to consider
\[\xymatrix{
\cal{A}_{w,\Sigma}\big(K_0; \cal{H}_w^\Box,\cal{M}_{\Sigma}^\Box\big)\ar[rr]^-{(\pi_{\Sigma})_*}&& \cal{A}_{w,\Sigma}\big(K_0; \cal{H}_w^\Box,\cal{M}(\bb{P}^1(F_\Sigma),J_{\Sigma}^\Box)\big)\\
&&\cal{A}_{w,\Sigma}\big(K_0; \cal{H}_w^\Box,\cal{M}_{\Sigma}^{\circ,\Box}\big)\ar[u]^-{(\iota_{\Sigma})_*}.
}\]
From the definitions it is not hard to see that
    \begin{equation}\label{circ-inter-compatibilities}
   (\pi_{\Sigma})_*\Phi_{\Sigma}^\Box= (\iota_{\Sigma})_*\Phi_\Sigma^{\circ,\Box}.
    \end{equation}
\end{remark}

Recall the $\bb{T}_\cal{S}^\circ$-equivariant map  $\mrm{pr}_A^\circ\colon J_\cal{S}^{\circ,\Box}\to A(E^\Box)_{\Z_p}$ of Remark \ref{rmk: projection to A} and set 
\[
\Phi^{\Box}_A:=(\mrm{pr}_A^\circ)_*\Phi_\cal{S}^{\circ,\Box}.
\]

\begin{lemma}\label{lemma: steinberg valued}
 Fix $b\in G(\A^{w,\infty})$, $z\in\cal{H}_w^\Box$, and $\p\in\cal{S}$. For any compact open subset $U^\p\subseteq\bb{P}^1(F_{\cal{S}\setminus\{\p\}})$
    \[
    \Phi^{\Box}_A(b)(z)\big(\bb{P}^1(F_\p)\times U^\p\big)=0.
    \]
\end{lemma}
\begin{proof}
The proof is based on the same idea as that of Theorem \ref{vanishing implies invariance}. By considering the basis   $\big\{\xi\cdot(K_m^\circ)_{\cal{S}}\ \lvert \ \xi\in (K_0)_\cal{S},\ m\ge1 \big\}$ of compact open subsets of $\bb{P}^1(F_{\cal{S}})$,  the formula \eqref{explicit values phi circ} shows that $ \Phi^{\Box}_A(b)(z)\big(\bb{P}^1(F_\p)\times U^\p\big)$ arises from points on a Shimura curve $X$ without level at $\p$. Since the elliptic curve $A_{/F}$ has multiplicative reduction at every prime in $\cal{S}$, the only homomorphism from the Jacobian of $X$ to $A_{/F}$ is the trivial one.
\end{proof}

\begin{lemma}\label{lemma: group action}
    The automorphic function $\Phi^{\Box}_A$ satisfies
    \[
    \Phi^{\Box}_A(-\cdot g)=\chi_{\cal{S}^\pm}(g)\cdot g^{-1}.\Phi^{\Box}_A(-)\qquad \forall\ g\in G_\cal{S}.
    \]
\end{lemma}
\begin{proof}
   Let us write $\Phi^{\circ}$ for $\Phi_\cal{S}^{\circ,\Box}$ to ease the notation. By Cartan's decomposition and the properties of $\Phi^{\circ}$, it suffices to verify the claim for $g=\eta_{\varpi_\p}$. Let $b\in G(\A^{w,\infty})$ and $z\in\cal{H}_w^\Box$, then Proposition \ref{prop: comparison of distributions} and equation \eqref{circ-inter-compatibilities} imply the existence of a compact open subset  $U\subset \bb{P}^1(F_\cal{S})$ depending on $b_\cal{S}$ such that
   \[
   \eta_{\varpi_\p}.\Phi^{\circ}(b\cdot \eta_{\varpi_\p})(z)=\Phi^{\circ}(b)(z)\otimes \mbf{U}_{\varpi_\p}\qquad \text{on}\ U,
   \]
   \[
   \Phi^{\circ}(b)(z)=\eta_{\varpi_\p}.\Phi^{\circ}(b\cdot \eta_{\varpi_\p})(z)\otimes \mbf{U}_{\varpi_\p}\qquad \text{on}\ U^c.
\]
   Since $\mrm{pr}_A^\circ(-\cdot \mbf{U}_{\varpi_\p})=\varepsilon_\p\cdot \mrm{pr}_A^\circ(-)$ and $\varepsilon_\p^2=1$, we deduce that 
   \[
   \eta_{\varpi_\p}.\Phi^{\Box}_A(b\cdot \eta_{\varpi_\p})(z)=\varepsilon_\p\cdot \Phi^{\Box}_A(b)(z).
   \]
\end{proof}

We are ready to describe $\omega_{A,S}^\Box$ using the tower of Shimura curves with $K_m^\circ$-level structures:
\begin{corollary}\label{description using towers}
 For any $h\in G(\A^{S,\infty})$ and $z\in\cal{H}_w^\Box$, the equality
    \[
\omega_{A,S}^\Box(h)(z)=\Phi^{\Box}_A(h)(z)
\]
holds in $\mrm{Hom}_{\Z_p}\big(\mrm{St}_{\cal{S}^\pm}, A(E^\Box)\otimes_\Z\Z_p\big)$.
\end{corollary}
\begin{proof}
Lemma \ref{lemma: steinberg valued} shows that $\Phi^{\Box}_A$ takes values in the dual of the Steinberg representation, i.e.,
\[
\Phi_A^\Box(b)(z)\hspace{1mm}\in\hspace{1mm} \mrm{Hom}_{\Z_p}\big(\mrm{St}_{\cal{S}^\pm}, A(E^\Box)\otimes_\Z\Z_p\big)
\]
for every $b\in G(\A^{w,\infty})$ and $z\in\cal{H}_w^\Box$. Now, if $h\in G(\A^{S,\infty})$ then for any $\delta\in G(F)$ we compute 
\[\begin{split}
(\delta^{-1}.\Phi_A^\Box)(h)(z)(\mathbbm{1}_U):=&\chi_{\cal{S}^\pm}(\delta_\cal{S})\cdot \Phi_A^\Box(\delta^Sh)(\delta_w.z)(\mathbbm{1}_{\delta_\cal{S}.U})\\
=&\Phi_A^\Box((\delta^w)^{-1}\cdot \delta^Sh\cdot\delta_\cal{S})(z)(\mathbbm{1}_U)\\
=&\Phi_A^\Box(h)(z)(\mathbbm{1}_U)
\end{split}\]
where the second equality follows from Lemma \ref{lemma: group action} and the properties of the automorphic function $\Phi_\cal{S}^{\circ,\Box}$, the third from $h\cdot\delta_\cal{S}=\delta_\cal{S}\cdot h$ and $\delta^w=\delta^S\cdot \delta_\cal{S}$. Thus, $\Phi_A^\Box$ defines a class in
\[
\mrm{H}^0\big(X_G^S(\frak{f});\ \mrm{St}_{\cal{S}^\pm}\otimes_{\Z_p} \Z_p[\cal{H}_w^\Box],\ A(E^\Box)\otimes_\Z\Z_p\big)_{\pi_A}.
\]
We are left to show that the image of the class in $\mrm{H}^0\big(X_G^w(\frak{f});\ \Z_p[\cal{H}_w^\Box],\ A(E^\Box)\otimes_\Z\Z_p\big)$
with respect to the map \eqref{defining property}  is the class defined by the modular parametrization $X^\circ_1\to A$. Explicitly, we need to evaluate $\Phi_A^\Box(b^S)(z)(b_\cal{S}.\mathbbm{1}_{U_{e_\circ}})$ for any $b\in G(\A^{w,\infty})$. As required, we compute that
\[\begin{split}
\Phi_A^\Box(b^S)(z)(b_\cal{S}.\mathbbm{1}_{U_{e_\circ}})
:=&\chi_{\cal{S}^\pm}(b_\cal{S})\cdot \Phi_A^\Box(b^S)(z)(\mathbbm{1}_{U_{b_\cal{S}.e_\circ}})\\
=&\Phi_A^\Box( b)(z)(\mathbbm{1}_{U_{e_\circ}})\\
=&\mrm{pr}_A^\circ([z,b]_{K_1^\circ}).
\end{split}\]
\end{proof}

\subsection{Proof of Theorem \ref{thm: comp circ and nocirc}}\label{Proof of theorem comp circ and nocirc}

It is worth highlighting the following consequence of Proposition \ref{independence pushforward} and equation \eqref{circ-inter-compatibilities}.
\begin{corollary}\label{cor: relation between circ and non-circ measures}
    Suppose that $\cal{S}=\Sigma=\Sigma_1\amalg\Sigma_2$. For any plectic zero divisor $D_{\Sigma_1}\in \Z_p[\cal{H}_{\Sigma_1}]^\plectic$, any collection $\omega_{\Sigma_1}=\{\omega_\p\in\p\cal{O}_\p\}_{\p\in\Sigma_1}$, and any compact open subset $U_{\Sigma_2}\subseteq\bb{P}^1_{\Sigma_2}$, the following equality holds
    \[\begin{split}
 \Bigg(\int_{\bb{P}^1_{\Sigma_1}\times U_{\Sigma_2} }\ell_{D_{\Sigma_1},\omega_{\Sigma_1}}\otimes\mathbbm{1}_{\Sigma_2}(\gamma)\hspace{2mm}\mrm{d}&\eta^{\varpi,\Box}_{(\tau_w,h)}(\gamma)\Bigg)\otimes1\\
 =& \left(\int_{\bb{P}^1_{\Sigma_1}\times U_{\Sigma_2} }\ell_{D_{\Sigma_1},\omega_{\Sigma_1}}\otimes\mathbbm{1}_{\Sigma_2}(\gamma)\hspace{2mm}\mrm{d}\hspace{0.5mm}\big(\iota_{*}\Phi_\cal{S}^{\circ,\Box}(h)(\tau_w)\big)(\gamma)\right)\otimes1.
    \end{split}\]
\end{corollary}
\begin{proof}
    By Proposition \ref{independence pushforward}, we can suppose that for every $\p\in\cal{S}$ the  $\p$-component of $\tau_S\in\cal{H}_{S}^\Box$ in the definition of $\eta^{\varpi,\Box}_{(\tau_w,h)}$ reduces to the standard vertex in the Bruhat--Tits tree $\cal{T}_\p$. Then, the claim follows from the definitions and equation \eqref{circ-inter-compatibilities}.
\end{proof}

\begin{proof}[Proof of Theorem \ref{thm: comp circ and nocirc}]
The goal is to show the equality
    \[
    (\iota_\plectic)_*\varphi_A^\Box=(\cal{K})_*\varphi_A^{\circ,\Box}
    \]
where  $\iota_\plectic\colon \Z_p[\cal{H}_{S}^\Box]^\plectic\hookrightarrow \Z_p[\cal{H}_{S}^\Box]$ is the natural inclusion and  $\cal{K}\colon A(E^\Box)^\wedge\hookrightarrow \mrm{H}^1(E^\Box,T_p(A))$ is the Kummer map. 
Fix $h\in G(\A^{S,\infty})$. After Corollaries \ref{cor: relation between circ and non-circ measures} and \ref{description using towers}, it suffices to show that for every $D\otimes[\tau_w]\in\Z_p[\cal{H}_{S}^\Box]^\plectic$
\[
(\varphi_\varpi^\Box\otimes1)(h)(D\otimes[\tau_{w}])=\left(\int_{\bb{P}^1(F_\cal{S})}\ell_{D,q}\hspace{2mm}\mrm{d}\eta^{\varpi,\Box}_{(\tau_{w},h)}\right)\otimes1.
\]

\medskip
\noindent Let $n\le\lvert\cal{S}\rvert$ be an integer. Suppose by inductive hypothesis that for every subset $\Sigma\subseteq\cal{S}$ of size $\lvert\Sigma\rvert<n$ we know that
\[
(\varphi_\varpi^\Box\otimes1)(h)(D_\Sigma\otimes[\tau_{\Sigma^c}])=\left(\int_{\bb{P}^1_\Sigma\times\bb{X}_{\mrm{v}_{\Sigma^c}}}\ell_{D_\Sigma,q_\Sigma}\otimes f_{\tau_{\Sigma^c}}\hspace{2mm}\mrm{d}\eta^{\varpi,\Box}_{(\tau_{\Sigma^c},h)}\right)\otimes1
\]
where $D_\Sigma\in\Z_p[\cal{H}_\Sigma]^\plectic$ and $\tau_{\Sigma^c}\in\cal{H}_{\Sigma^c}$. For the induction step it suffices to consider the following:
\begin{itemize}
    \item [$\bfcdot$] a subset $\Sigma\subset\cal{S}$ of size $\lvert \Sigma\rvert=n-1$, a prime $\p\not\in \Sigma$,
     \item [$\bfcdot$] two points $\tau_\p,\tau_\p'\in\cal{H}_\p$ such that $m_{\mrm{v}_\p'}=m_{\mrm{v}_\p}+1$, and $\tau_{\Sigma^c\setminus\{\p\}}\in\cal{H}_{\Sigma^c\setminus\{\p\}}$.
\end{itemize} 
After setting $\tau_{\Sigma^c}=(\tau_\p, \tau_{\Sigma^c\setminus\{\p\}})$ and $\tau_{\Sigma^c}'=(\tau_\p',\tau_{\Sigma^c\setminus\{\p\}})$,
 we can compute that
\[\begin{split}
        (\varphi_\varpi^\Box\otimes1)(h)(D_\Sigma\otimes[\tau_{\Sigma^c}])-(\varphi_\varpi^\Box&\otimes1)(h)(D_\Sigma\otimes[\tau'_{\Sigma^c}])\\   =\Bigg(\int_{\bb{P}^1_\Sigma\times\bb{X}_{\mrm{v}_{\Sigma^c}}}&\ell_{D_\Sigma,q_\Sigma}\otimes (f_{\tau_{\Sigma^c}}-f_{\tau'_{\Sigma^c}})\hspace{2mm}\mrm{d}\eta^{\varpi,\Box}_{(\tau_{\Sigma^c},h)}\Bigg)\otimes1\\
        &+\Bigg(\int_{\bb{P}^1_\Sigma\times(\bb{X}_{\mrm{v}_{\Sigma^c}}\cup\bb{X}_{\mrm{v}_{\Sigma^c}'})}\ell_{D_\Sigma,q_\Sigma}\otimes f_{\tau_{\Sigma^c}'}\hspace{2mm}\mrm{d}\big(\eta^{\varpi,\Box}_{(\tau_{\Sigma^c},h)}- \eta^{\varpi,\Box}_{(\tau_{\Sigma^c}',h)}\big)\Bigg)\otimes1.
    \end{split}
    \]
    The first addend is easily seen to be equal to
    \[
    \Bigg(\int_{\bb{P}^1_{\Sigma\cup\{\p\}}\times\bb{X}_{\mrm{v}_{\Sigma^c\setminus\{\p\}}}}\ell_{D_\Sigma,q_\Sigma}\otimes\log_{\varpi_\p}\left(\frac{t_\p-\tau_\p}{t_\p-\tau_\p'}\right)\otimes f_{\tau_{\Sigma^c\setminus\{\p\}}}\hspace{2mm}\mrm{d}\eta^{\varpi,\Box}_{(\tau_{\Sigma^c\setminus\{\p\}},h)}\Bigg)\otimes1,
  \]
    while the second is computed in Corollary \ref{ord-integral}. Since 
    \[
\frac{\partial_\p\big(\alpha_{\varpi_\p}^\otimes\big)}{\alpha^\otimes_{\varpi_\p}}\bigg\lvert_{\underline{0}}=-\frac{\log_{\varpi_\p}(q_\p)}{\mrm{ord}_\p(q_\p)}\otimes(\otimes_{\q\not=\p}1)
\]
by Proposition \ref{logarithmic derivative U_p}, we obtain
    \[\begin{split}
(\varphi_\varpi^\Box\otimes1)(h)(D_\Sigma\otimes&[\tau_{\Sigma^c}])-(\varphi_\varpi^\Box\otimes1)(h)(D_\Sigma\otimes[\tau'_{\Sigma^c}])\\
 &=\Bigg(\int_{\bb{P}^1_{\Sigma\cup\{\p\}}\times\bb{X}_{\mrm{v}_{\Sigma^c\setminus\{\p\}}}}\ell_{D_\Sigma,q_\Sigma}\otimes\log_{q_\p}\left(\frac{t_\p-\tau_\p}{t_\p-\tau_\p'}\right)\otimes f_{\tau_{\Sigma^c\setminus\{\p\}}}\hspace{2mm}\mrm{d}\eta^{\varpi,\Box}_{(\tau_{\Sigma^c\setminus\{\p\}},h)}\Bigg)\otimes1
\end{split} \]   
as required.
\end{proof}

\bibliography{Plectic}
\bibliographystyle{alpha}

\end{document}